\documentclass[10pt]{amsart}%
\usepackage[T1]{fontenc}
\usepackage{amssymb}
\usepackage{a4wide}
\usepackage{graphicx,color,subfig}
\usepackage{verbatim}
\usepackage{amsmath}
\usepackage{version}
\usepackage{amsfonts}%
\usepackage{dsfont}%

\usepackage{mathrsfs}
\usepackage[colorlinks=true]{hyper ref}

\newtheorem{theorem}{Theorem}[section]
\newtheorem{corollary}[theorem]{Corollary}
\newtheorem{definition}[theorem]{Definition}

\newtheorem{lemma}[theorem]{Lemma}

\newtheorem{proposition}[theorem]{Proposition}
\newtheorem{remark}[theorem]{Remark}

\newcommand{\ovl}[1]{\overline{#1}}
\newcommand{\W}{\ensuremath{\mathbb W}}
\renewcommand{\d}{\ensuremath{\partial}}
\newcommand{\scrU}{\ensuremath{\mathscr U}}
\newcommand{\scrV}{\ensuremath{\mathscr V}}
\DeclareMathOperator{\loc}{loc}
\DeclareMathOperator{\comp}{comp}

\newcommand{\nwc}{\newcommand}
\nwc{\nwt}{\newtheorem}
\nwt{coro}{Corollary}
\nwt{ex}{Example}
\nwt{prop}{Proposition}
\nwt{defin}{Definition}

\nwc{\mf}{\mathbf} 
\nwc{\blds}{\boldsymbol} 
\nwc{\ml}{\mathcal} 

\nwc{\lam}{\lambda}
\nwc{\del}{\delta}
\nwc{\Del}{\Delta}
\nwc{\Lam}{\Lambda}
\nwc{\elll}{\ell}

\nwc{\IA}{\mathbb{A}} 
\nwc{\IB}{\mathbb{B}} 
\nwc{\IC}{\mathbb{C}} 
\nwc{\ID}{\mathbb{D}} 
\nwc{\IE}{\mathbb{E}} 
\nwc{\IF}{\mathbb{F}} 
\nwc{\IG}{\mathbb{G}} 
\nwc{\IH}{\mathbb{H}} 
\nwc{\IN}{\mathbb{N}} 
\nwc{\IP}{\mathbb{P}} 
\nwc{\IQ}{\mathbb{Q}} 
\nwc{\IR}{\mathbb{R}} 
\nwc{\IS}{\mathbb{S}} 
\nwc{\IT}{\mathbb{T}} 
\nwc{\IZ}{\mathbb{Z}} 

\def\bbleft{{\mathchoice {[\mskip-3mu {[}} {[\mskip-3mu {[}}{[\mskip-4mu {[}}{[\mskip-5mu {[}}}}
\def\bbright{{\mathchoice {]\mskip-3mu {]}} {]\mskip-3mu {]}}{]\mskip-4mu {]}}{]\mskip-5mu {]}}}}
\nwc{\setK}{\bbleft 1,K \bbright}
\nwc{\setN}{\bbleft 1,\cN \bbright}
 \newcommand{\Lim}{\mathop{\longrightarrow}\limits}

\nwc{\va}{{\bf a}}
\nwc{\vb}{{\bf b}}
\nwc{\vc}{{\bf c}}
\nwc{\vd}{{\bf d}}
\nwc{\ve}{{\bf e}}
\nwc{\vf}{{\bf f}}
\nwc{\vg}{{\bf g}}
\nwc{\vh}{{\bf h}}
\nwc{\vi}{{\bf i}}
\nwc{\vI}{{\bf I}}
\nwc{\vj}{{\bf j}}
\nwc{\vk}{{\bf k}}
\nwc{\vl}{{\bf l}}
\nwc{\vm}{{\bf m}}
\nwc{\vM}{{\bf M}}
\nwc{\vn}{{\bf n}}
\nwc{\vo}{{\it o}}
\nwc{\vp}{{\bf p}}
\nwc{\vq}{{\bf q}}
\nwc{\vr}{{\bf r}}
\nwc{\vs}{{\bf s}}
\nwc{\vt}{{\bf t}}
\nwc{\vu}{{\bf u}}
\nwc{\vv}{{\bf v}}
\nwc{\vw}{{\bf w}}
\nwc{\vx}{{\bf x}}
\nwc{\vy}{{\bf y}}
\nwc{\vz}{{\bf z}}
\nwc{\bal}{\blds{\alpha}}
\nwc{\bep}{\blds{\epsilon}}
\nwc{\barbep}{\overline{\blds{\epsilon}}}
\nwc{\bnu}{\blds{\nu}}
\nwc{\bmu}{\blds{\mu}}
\nwc{\bet}{\blds{\eta}}

\nwc{\bk}{\blds{k}}
\nwc{\bm}{\blds{m}}
\nwc{\bM}{\blds{M}}
\nwc{\bp}{\blds{p}}
\nwc{\bq}{\blds{q}}
\nwc{\bn}{\blds{n}}
\nwc{\bv}{\blds{v}}
\nwc{\bw}{\blds{w}}
\nwc{\bx}{\blds{x}}
\nwc{\bxi}{\blds{\xi}}
\nwc{\by}{\blds{y}}
\nwc{\bz}{\blds{z}}

\nwc{\cA}{\ml{A}}
\nwc{\cB}{\ml{B}}
\nwc{\cC}{\ml{C}}
\nwc{\cD}{\ml{D}}
\nwc{\cE}{\ml{E}}
\nwc{\cF}{\ml{F}}
\nwc{\cG}{\ml{G}}
\nwc{\cH}{\ml{H}}
\nwc{\cI}{\ml{I}}
\nwc{\cJ}{\ml{J}}
\nwc{\cK}{\ml{K}}
\nwc{\cL}{\ml{L}}
\nwc{\cM}{\ml{M}}
\nwc{\cN}{\ml{N}}
\nwc{\cO}{\ml{O}}
\nwc{\cP}{\ml{P}}
\nwc{\cQ}{\ml{Q}}
\nwc{\cR}{\ml{R}}
\nwc{\cS}{\ml{S}}
\nwc{\cT}{\ml{T}}
\nwc{\cU}{\ml{U}}
\nwc{\cV}{\ml{V}}
\nwc{\cW}{\ml{W}}
\nwc{\cX}{\ml{X}}
\nwc{\cY}{\ml{Y}}
\nwc{\cZ}{\ml{Z}}

\nwc{\fA}{\mathfrak{a}}
\nwc{\fB}{\mathfrak{b}}
\nwc{\fC}{\mathfrak{c}}
\nwc{\fD}{\mathfrak{d}}
\nwc{\fE}{\mathfrak{e}}
\nwc{\fF}{\mathfrak{f}}
\nwc{\fG}{\mathfrak{g}}
\nwc{\fH}{\mathfrak{h}}
\nwc{\fI}{\mathfrak{i}}
\nwc{\fJ}{\mathfrak{j}}
\nwc{\fK}{\mathfrak{k}}
\nwc{\fL}{\mathfrak{l}}
\nwc{\fM}{\mathfrak{m}}
\nwc{\fN}{\mathfrak{n}}
\nwc{\fO}{\mathfrak{o}}
\nwc{\fP}{\mathfrak{p}}
\nwc{\fQ}{\mathfrak{q}}
\nwc{\fR}{\mathfrak{r}}
\nwc{\fS}{\mathfrak{s}}
\nwc{\fT}{\mathfrak{t}}
\nwc{\fU}{\mathfrak{u}}
\nwc{\fV}{\mathfrak{v}}
\nwc{\fW}{\mathfrak{w}}
\nwc{\fX}{\mathfrak{x}}
\nwc{\fY}{\mathfrak{y}}
\nwc{\fZ}{\mathfrak{z}}

\nwc{\tA}{\widetilde{A}}
\nwc{\tB}{\widetilde{B}}
\nwc{\tE}{E^{\vareps}}
\nwc{\tk}{\tilde k}
\nwc{\tN}{\tilde N}
\nwc{\tP}{\widetilde{P}}
\nwc{\tQ}{\widetilde{Q}}
\nwc{\tR}{\widetilde{R}}
\nwc{\tV}{\widetilde{V}}
\nwc{\tW}{\widetilde{W}}
\nwc{\ty}{\tilde y}
\nwc{\teta}{\tilde \eta}
\nwc{\tdelta}{\tilde \delta}
\nwc{\tlambda}{\tilde \lambda}
\nwc{\ttheta}{\tilde \theta}
\nwc{\tvartheta}{\tilde \vartheta}
\nwc{\tPhi}{\widetilde \Phi}
\nwc{\tpsi}{\tilde \psi}
\nwc{\tmu}{\tilde \mu}

\nwc{\To}{\longrightarrow} 

\nwc{\ad}{\rm ad}
\nwc{\eps}{\epsilon}
\nwc{\ep}{\epsilon}
\nwc{\vareps}{\varepsilon}

\def\ep{\epsilon}

\def\Tr{{\rm Tr}}

\def\e{{\rm e}}
\def\sq2{\sqrt{2}}

\def\t2{{\mathbb T}^2}
\def\s2{{\mathbb S}^2}

\def\N{\mathbb{N}}
\def\T{\mathbb{T}}
\def\R{\mathbb{R}}

\def\Z{\mathbb{Z}}

\nwc{\lap}{\bigtriangleup}
\nwc{\rest}{\restriction}
\nwc{\Diff}{\operatorname{Diff}}
\nwc{\diam}{\operatorname{diam}}
\nwc{\Res}{\operatorname{Res}}
\nwc{\Spec}{\operatorname{Spec}}
\nwc{\Vol}{\operatorname{Vol}}
\nwc{\Op}{\operatorname{Op}}
\nwc{\supp}{\operatorname{supp}}
\nwc{\Span}{\operatorname{span}}

\nwc{\dia}{\varepsilon}
\nwc{\cut}{f}
\nwc{\qm}{u_\hbar}

\def\hto0{\xrightarrow{\hbar\to 0}}

\def\rto0{\xrightarrow{r\to 0}}

\providecommand{\norm}[1]{\lVert#1\rVert}

\nwc{\la}{\langle}
\nwc{\ra}{\rangle}
\nwc{\lp}{\left(}
\nwc{\rp}{\right)}

\nwc{\bequ}{\begin{equation}}
\nwc{\be}{\begin{equation}}
\nwc{\ben}{\begin{equation*}}
\nwc{\bea}{\begin{eqnarray}}
\nwc{\bean}{\begin{eqnarray*}}
\nwc{\bit}{\begin{itemize}}
\nwc{\bver}{\begin{verbatim}}

\nwc{\eequ}{\end{equation}}
\nwc{\ee}{\end{equation}}
\nwc{\een}{\end{equation*}}
\nwc{\eea}{\end{eqnarray}}
\nwc{\eean}{\end{eqnarray*}}
\nwc{\eit}{\end{itemize}}
\nwc{\ever}{\end{verbatim}}

\numberwithin{equation}{section}

\begin{document}

\title[Wigner measures on the disk]{Wigner measures and observability for the Schr\"odinger equation on the disk
}
\author{Nalini Anantharaman}
\address{Universit\'{e} Paris-Sud 11, Math\'{e}matiques, B\^{a}t. 425, 91405 Orsay
Cedex, France} \email{Nalini.Anantharaman@math.u-psud.fr}
\author{Matthieu L\'eautaud}
\address{Universit\'e Paris Diderot, Institut de Math\'ematiques de Jussieu-Paris Rive Gauche, UMR 7586, B\^atiment Sophie Germain, 75205 Paris Cedex 13 France}
\email{leautaud@math.univ-paris-diderot.fr}
\author{Fabricio Maci\`a}
\address{Universidad Polit\'{e}cnica de Madrid. DCAIN, ETSI Navales. Avda. Arco de la
Victoria s/n. 28040 MADRID, SPAIN} \email{Fabricio.Macia@upm.es}

\begin{abstract}

We analyse the structure of semiclassical and microlocal Wigner measures for solutions to the linear
Schr\"{o}dinger equation on the disk, with Dirichlet boundary
conditions.

Our approach links the propagation of singularities beyond
geometric optics with the completely integrable nature of the
billiard in the disk. 
    We prove a ``structure theorem'', expressing the restriction of the Wigner measures on each invariant torus in terms of {\em second-microlocal measures}. They are obtained by performing a finer localization in phase space around each of these tori, at the limit of the uncertainty principle, and are shown to propagate according to Heisenberg equations on the circle.

Our construction yields as corollaries (a) that the disintegration
of the Wigner measures is absolutely continuous in the angular
variable, which is an expression of the dispersive properties of
the equation; (b) an observability inequality, saying that the
$L^2$-norm of a solution on any open subset intersecting the
boundary (resp. the $L^2$-norm of the Neumann trace on any
nonempty open set of the boundary) controls its full $L^2$-norm
(resp. $H^1$-norm).
These results show in particular that the energy of solutions
cannot concentrate on periodic trajectories of the billiard flow
other than the boundary.

%
%
\end{abstract}


\maketitle

\setcounter{tocdepth}{2} \tableofcontents

\section{Introduction}
\label{sec:intro}
\subsection{Motivation}
We consider the unit disk
$$
\ID = \{z=(x,y) \in \IR^2 , |z|^2 = x^2 + y^2 <1\} \subset \IR^2
$$
and denote by $\Delta$ the
euclidean Laplacian. We are interested in understanding dynamical
properties of the
(time-dependent) linear Schr\"{o}dinger equation
\begin{eqnarray}
\label{e:S} {\frac{1}{i}}\frac{\partial u}{\partial
t}(z,t)=\left( -\frac{1}{2}\Delta+V(t, z)\right) u(z,t), & t\in\R,
\quad z=(x, y)\in \ID,
\\
\label{e:Dirichlet} u\rceil_{t=0} = u^0 \in L^{2}(\ID)&&
\end{eqnarray}
with Dirichlet boundary condition $u\rceil_{\d \ID} = 0$ (we shall write $\Delta=\Delta_{D}$ when we want
to stress that we are using the Laplacian with that boundary
condition). We assume that $V$ is a smooth real-valued potential, say
$V\in C^\infty\left(  \mathbb{R\times \ovl{\ID}} ; \R \right)$. 
We shall denote by $U_{V}(t)$ the (unitary) propagator starting at time $0$, such that $u(\cdot, t)=U_{V}(t)u^0$ is the unique
solution of \eqref{e:S}-\eqref{e:Dirichlet}.

This equation is aimed at describing the evolution of a quantum particle trapped in a disk-shaped cavity, $u(\cdot, t)$ being the wave-function at time $t$. The total $L^2$-mass of the solution is preserved: $\|u(\cdot, t)\|_{L^2(\ID)} = \|u^0\|_{L^2(\ID)}$ for all time $t \in \R$. Thus, if the initial datum is normalized, $\|u^0\|_{L^2(\ID)}=1$, the quantity $|u(z,t)|^2 dz$ is, for every fixed $t$, a probability density on $\ID$; given $\Omega \subset \ID$, the expression:
$$
\int_{\Omega}|u(z,t)|^2 dz
$$
is the probability of finding the particle in the set $\Omega$ at time $t$. Having 
$\int_0^T \int_{\Omega}|u(z,t)|^2 dx dt \geq c_0 >0$ for all solutions of~\eqref{e:S} means that every quantum particle spends a positive fraction of time of the interval $(0,T)$ in the set $\Omega$. A major issue in mathematical quantum mechanics is to describe the possible localization -- or delocalization -- properties of solutions to the Schr\"odinger equation~\eqref{e:S}, by which we mean the description of the distribution of the probability densities  $|u(z,t)|^2 dz$ for all solutions $u$. 
A more tractable problem consists in considering instead of single, fixed solutions, sequences $(u_n)_{n \in \N}$ of solutions to~\eqref{e:S} and describe the asymptotic properties of the associated probability densities $|u_n(z,t)|^2 dz$ or $|u_n(z,t)|^2 dz dt$. This point of view still allows to deduce properties of single solutions $u$ and their distributions $|u(z,t)|^2 dz$, as we shall see in the sequel.

It is always possible to extract a subsequence that converges weakly: $$\int_{\ID\times \IR}\phi(z,t)|u_n(z,t)|^2 dz dt \longrightarrow \int_{\overline{\ID}\times \IR}\phi(z,t)\nu(dz, dt),\quad \text{ for every } \phi \in C_c(\overline{\ID}\times \IR),$$  where $\nu$ is a nonnegative Radon measure on $\overline{\ID}\times\IR$ that describes the asymptotic mass distribution of the sequence of solutions $(u_n)$. One of the goals of this paper is to understand how the fact that $(u_n)$ solves~\eqref{e:S} influences the structure of the associated measure $\nu$.

As an application, we aim at understanding the observability problem for the Schr\"odinger equation: given an open set $\Omega \subset \ID$ and a time $T>0$, does there exist a constant $C =C(\Omega,T)>0$ such that we have:
\begin{equation}
\label{eq:IOintro}
\int_0^T\int_{\Omega}|u(z,t)|^2 dz dt \geq C\|u^0\|_{L^2(\ID)}^2, \text{ for all } u^0 \in L^2(\ID)\text{ and } u \text{ associated solution of~\eqref{e:S}} ?
\end{equation}
If such an estimate holds, then every quantum particle must leave a trace on the set $\Omega$ during the time interval $(0,T)$; in other words: it is observable from $\Omega\times (0,T)$. This question is linked to that of understanding the structure of the limiting measures $\nu$. Estimate \eqref{eq:IOintro} is {\em not} satisfied if and only if there exists a sequence of data $(u_n^0)$ such that $\|u_n^0\|_{L^2(\ID)}=1$ and $\int_0^T\int_{\Omega}|u_n(z,t)|^2 dz dt \to 0$, where $u_n$ is the solution of~\eqref{e:S} issued from $u^0_n$. After the extraction of a subsequence, this holds if and only if the associated limit measure $\nu$ satisfies
$$
\int_0^T \int_{\overline{\ID}} \nu(dz, dt) = T  ,\quad \int_0^T \int_\Omega \nu(dz, dt) = 0 .
$$
The question of observability from $\Omega\times (0,T)$ may hence be reformulated as: can sequences of solution of~\eqref{e:S} concentrate on sets which do not intersect $\Omega\times (0,T)$? From the point of view of applications, it is of primary interest to understand which sets $\Omega$ do observe all quantum particles trapped in a disk. Moreover, the observability of~\eqref{e:S} is equivalent to the controllability of the Schr\"odinger equation (see e.g.~\cite{Leb:92}), which means that it is possible to drive any initial condition to any final condition at time $T$, with a control (a forcing term in the right-hand side of~\eqref{e:S}) located within $\Omega$.

\bigskip
It is well-known that the space of position variables $(z,t)$ does not suffice to describe the propagation properties of solutions to Schr\"odinger equations (or more generally wave equations) in the high frequency r\'egime. To take the latter into account, one has to add the associated dual variables, $(\xi,H)\in\R^2\times \R$ (momentum and energy) and lift the measure $\nu$ to the {\em phase space}, associated to the variables $(z,t,\xi,H)$: this gives rise to the so-called \emph{Wigner measures}~\cite{Wigner32}. We shall hence investigate the regularity and localization properties in position and momentum variables of
the  Wigner measures  associated with sequences of normalized solutions of~\eqref{e:S}. They describe how the solutions are
distributed over phase space. We shall develop both the
\emph{microlocal} and \emph{semiclassical} points of view. These
are two slightly different, but closely related, approaches to the
problem~: the semiclassical approach is more suitable when our initial data possess a well-defined oscillation rate, whereas the microlocal approach describes the singularities of solutions, independently of the choice of a scale of oscillation, at the price of giving slightly less precise results. 

Our study fits in the regime of the ``quantum-classical correspondence principle'' , which asserts that the high-frequency dynamics of the  solutions to~\eqref{e:S} are described in terms of the corresponding classical dynamics; in our context the underlying classical system is the billiard flow on $\ID$. Wigner measures carry this information, for they are known to be invariant by this flow.

Of course, one may consider similar questions for any bounded domain of $\R^d$ or any Riemannian manifold, and not only the  disk $\ID$. As a matter of fact, the answer to these questions depends strongly on the dynamics of the billiard flow (resp. the geodesic flow on a Riemannian manifold), and, to our knowledge, it is known only in few cases (see Section~\ref{s:comments}).  For instance, on negatively curved manifolds, the celebrated Quantum Unique Ergodicity conjecture remains to this day open. Two geometries for which the observation problem is well-understood, and the Wigner measures are rather well-described, are the torus $\T^d$ (see~\cite{JaffardPlaques,Kom:92,MaciaDispersion,BZ:12} and~\cite{JakobsonTori97, BourgainLatt, MaciaTorus, AnantharamanMaciaTore}) and the sphere $\mathbb{S}^d$, or more generally, manifolds all of whose geodesics are closed (see~\cite{JZ:99, MaciaZoll, MaciaAv, AM:10}), on which the classical dynamics is completely integrable. We shall later on compare these two situations with our results on the disk $\ID$. 
We refer to the article~\cite{AMSurvey} for a survey of recent results concerning Wigner measures associated to sequences of solutions to the time-dependent Schr\"odinger equation in various geometries and to the review article~\cite{Laurent14} on the observability question.

%

\subsection{Some consequences of our structure theorem}
Our central results are Theorems \ref{t:precise} and
\ref{t:preciseml} below, which provide a detailed structure of the
Wigner measures associated to sequences of solutions to the
Schr\"{o}dinger equation, using notions of second-microlocal
calculus. 
As corollaries of these structure Theorems, we obtain:
\begin{itemize}
\item  Corollary \ref{t:example} (see also Theorem \ref{t:main}), which
reflects the dispersive character of the Schr\"{o}dinger equation
\eqref{e:S}; \item Theorem \ref{t:obs-i} (resp.~Theorem
\ref{t:obs-b}), which states the observability/controllability of
the equation from any nonempty open set touching the boundary of
the disk (resp. from any nonempty open set of the boundary).
\end{itemize}
Let us first state these corollaries in order to motivate the more
technical results of this paper.

\begin{corollary}
\label{t:example}Let $(u^0_{n})$ be a sequence in $L^{2}(\ID)$,
such that $\norm{u^0_n}_{L^{2}(\ID)}=1$ for all $n$. Consider the
sequence of nonnegative Radon measures $\nu_{n}$ on $\ovl{\ID}\times
\IR$, defined by
\begin{equation}
\nu_{n}(dz, dt)=|U_{V}(t)u^0_{n}(z)|^{2}dz dt .
\label{probmes}%
\end{equation}
Let $\nu$ be any weak-$\ast$ limit of the sequence $(\nu_{n})$:
then $\nu(dz, dt)= \nu_t(dz) dt$ where, for almost every $t$,
$\nu_t$ is a probability measure on $\ovl{\ID}$, and
$\nu_t\rceil_{\ID}$ is absolutely continuous.
\end{corollary}
This result shows that the weak-$\ast$ accumulation points of the
densities (\ref{probmes}) possess some regularity in the interior
of the disk. This result cannot be extended to $\ovl{\ID}$, since
it is easy to exhibit sequences of solutions that
concentrate singularly on the boundary (the so-called
whispering-gallery modes, see Section~\ref{s:eigenfcts}). In
Theorem~\ref{t:main} below, we present a stronger version of
Corollary~\ref{t:example} describing (in phase space) the
regularity of microlocal lifts of such limit measures $\nu$. This
precise description (as well as all results of this paper) relies
on the complete integrability of the billiard flow on the disk.
Its statement needs the introduction of action angle coordinates
and associated invariant tori, and is postponed to
Section~\ref{s:appregularity}.

\bigskip
The second class of results mentioned above is related to unique
continuation-type properties of the Schr\"{o}dinger equation
\eqref{e:S}. We consider the following condition on an open set
$\Omega\subset \ovl{\ID}$, a time $T>0$ and a potential~$V$:
\begin{equation}
\label{UCP} \left( \tag{UCP$_{V ,\Omega,T}$} u^0 \in L^2(\ID), \quad
U_V(t)u^0 \rceil_{(0,T)\times \Omega} = 0 \right)
\quad  \Longrightarrow \quad 
u^0 =0 .
 \end{equation}
As a consequence of Theorem \ref{t:preciseml}, we shall also prove
the following quantitative version of~\eqref{UCP}.

\begin{theorem}
\label{t:obs-i} Let $\Omega\subset \ovl{\ID}$ be an open set such
that $\Omega\cap \partial \ID \not= \emptyset$ and $T>0$.
Assume one of the following statements holds:
\begin{itemize}
\item the potential $V\in C^\infty([0,T] \times \ovl{\ID}; \R)$, the time $T$,
and the open set $\Omega$ satisfy \eqref{UCP}, \item the potential
$V\in C^\infty(\ovl{\ID}; \R)$ does not depend on $t$.
\end{itemize}
Then there exists $C=C(V,\Omega,T)>0$ such that:
\begin{equation}
\left\Vert u^0\right\Vert _{L^{2}\left(  \ID\right)  }^{2}\leq
C\int_{0}^{T}\left\Vert U_{V}(t)  u^0\right\Vert _{L^{2}\left(
\Omega\right)  }^{2}dt, \label{e:oi}
\end{equation}
for every initial datum $u^0\in L^{2}\left(  \ID\right)  $.
\end{theorem}
Roughly speaking, this means that any set $\Omega$ touching $\d \ID$ observes all quantum particles trapped in the disk. As we shall see, these are the only sets satisfying this property (see Section~\ref{s:eigenfcts} and Remark~\ref{rem:nonobsnonbord}).

\bigskip
We are also interested in the boundary analogue of \eqref{UCP} for
a given potential $V$, a time $T>0$ and an open set~$\Gamma \subset
\partial \ID$:
\begin{equation}
\label{UCPbord} 
\tag{UCP$_{V ,\Gamma,T}$} 
\left( u^0 \in H^1_0(\ID), \quad
\partial_n(U_V(t)u^0) \rceil_{(0,T)\times \Gamma} = 0
\right)
\quad \Longrightarrow 
\quad u^0 =0,
\end{equation}
where $\partial_n = \frac{\partial}{\partial n}$ denotes the
exterior normal derivative to $\d \ID$. As a consequence of Theorem~\ref{t:preciseml},
we shall also prove the following quantitative
version of~\eqref{UCPbord}.

\begin{theorem}\label{t:obs-b}
Let $\Gamma$ be any nonempty subset of $\partial \ID$ and $T>0$.
Suppose one of the following holds:
\begin{itemize}
\item the potential $V\in C^\infty([0,T] \times \ovl{\ID})$, the
time $T$ and the set $\Gamma$ satisfy \eqref{UCPbord},

\item $V\in C^\infty(\ovl{\ID})$ does not depend on $t$.
\end{itemize}
Then there exists $C=C(V,\Gamma, T)>0$ such that:%
\begin{equation}\label{e:ob}
\left\Vert u^{0}\right\Vert _{H^1\left(  \ID\right)   }^{2}\leq
C\int_{0}^{T}\left\Vert \partial_n(U_{V}(t)  u^{0})\right\Vert
_{L^{2}\left( \Gamma\right)  }^{2}dt,
\end{equation}
for every initial datum $u^{0}\in H^1_0(\ID)$.
\end{theorem}

Note that the unique continuation properties \eqref{UCP} and \eqref{UCPbord} are known to hold, for instance, when $V$ is analytic in $(t,z)$, as a consequence of the Holmgren uniqueness theorem as stated by H\" ormander (see e.g.~\cite[Theorem~5.3.1]{Hormander:LPDO}). 

These three results express a delocalization property of the
energy of solutions to~\eqref{e:S}. The observation of the
$L^2$-norm restricted to {\em any} open set of the disk touching
the boundary is sufficient to recover linearly the norm of the
data. In particular, the $L^2$-mass of solutions cannot
concentrate on periodic trajectories of the billiard.
The observability inequalities~\eqref{e:oi} and~\eqref{e:ob} are especially relevant in control
theory (see \cite{LionsSurvey88, BLR:92, Leb:92}): in turn, they imply a controllability result from the set $\Omega$
or $\Gamma$.

\bigskip

As a consequence of the observability inequality~\ref{e:oi}, we have the following result (where we use the notation of Corollary~\ref{t:example}).

\begin{corollary} \label{c:positiveopen}For every open set $\Omega\subset\ovl{\ID}$ touching the boundary, for every $T>0$, there exists a constant $C(T,\Omega)>0$ such that for any initial data $(u_n^0)$ and any weak-$\ast$ limit $\nu$ of the sequence $(\nu_{n})$ as in Corollary~\ref{t:example}, we have
$$\int_0^T \nu_t( \Omega)dt\geq \frac{1}{C(T, \Omega)}.$$
\end{corollary}
This translates the fact that any solution has to leave positive mass on any set $\Omega$ touching the boundary $\d \ID$ during the time interval $(0,T)$. This may be rephrased by saying that any such set observes all quantum particles trapped in the disk.


\subsection{Stationary solutions to~\eqref{e:S}: eigenfunctions on the disk}
\label{s:eigenfcts}
If the potential $V(t,z)$ does not depend on the time variable $t$, we have as particular solutions of the Schr\"odinger
equation the ``stationary solutions'', those with initial data given by
eigenfunctions of the elliptic operator $-\Delta_D + 2 V(z)$ involved. 

In the
absence of potential, i.e. if $V=0$,  these solutions are well understood: the eigenfunctions of $-\Delta_D$ on $\ID$ are the functions
whose (non-normalized) expression in polar coordinates $(x=-r\sin u, y=r\cos u)$ is
\begin{equation}
\label{e:psidisk} \psi^\pm_{n, k}(r e^{iu})=J_n(\alpha_{n, k}r)
e^{\pm i nu} ,
\end{equation}
where $n, k$ are non-negative integers, $J_n$ is the $n$-th Bessel
function, and the $\alpha_{n, k}$ are its positive zeros ordered
increasingly with respect to $k$. The corresponding eigenvalue is
$\alpha_{n, k}^2$. Putting then $u^0 = \psi^\pm_{n, k}$ gives a time-periodic solution $u (\cdot, t) = e^{-it\frac{\alpha_{n, k}^2}{2}}\psi^\pm_{n, k}$ to~\eqref{e:S}-\eqref{e:Dirichlet}.
Moreover, the eigenvalues of $- \Delta_D$ have multiplicity two. This is a consequence of a celebrated result by Siegel~\cite{Siegel29}, showing that $J_n, J_m$ have no common zeroes
for $n \neq m$.
In particular, the limit measures associated to sequences of eigenfunctions are explicitly computable in terms of the limits of the stationary distributions: 
\[
\frac{|\psi_{n,k}^\pm(z)|^2}{\|\psi_{n,k}^\pm \|_{L^2(\ID)}^2}dz=\frac{|J_n(\alpha_{n, k}r)|^2}{\|\psi_{n,k}^\pm \|_{L^2(\ID)}^2}rdrdu,
\] 
as the frequency $\alpha_{n,k}$ tends to infinity (this expression has to be slightly modified when considering linear combinations of the two eigenfunctions $\psi_{n,k}^+$ and $\psi_{n,k}^-$, corresponding to the same eigenvalue, with $n$ fixed and $k$ tending to infinity) . 
Let us recall some particular cases of this construction. For fixed $k$ and for $n \to \infty$, it is classical ~\cite[Lemma~3.1]{Lagnese} that
$$
\frac{|\psi_{n,k}^\pm(z)|^2}{\|\psi_{n,k}^\pm \|_{L^2(\ID)}^2}dz\rightharpoonup (2\pi)^{-1} \delta_{\partial\ID},
$$
which corresponds to the so-called whispering gallery modes. On the other hand, letting $k, n \to \infty$ with $n/k$ being constant, one may obtain for any $\gamma \in [0,1)$ depending on the ratio $n/k$  ~\cite[Section~4.1]{PTZ:14}
$$
\frac{|\psi_{n,k}^\pm(z)|^2}{\|\psi_{n,k}^\pm \|_{L^2(\ID)}^2}dz\rightharpoonup \frac{1}{2\pi (1 -\gamma^2)^{1/2}} \frac{1}{(|z|^2- \gamma^2)^{1/2}} \mathds{1}_{(\gamma,1)}(|z|) dz .
$$
Except the Dirac measure on the boundary, these measures all belong to $L^{p}(\ID)$ for any $p<2$ (hence satisfying Corollary~\ref{t:example}) and are invariant by rotation and positive on the boundary (hence satisfying Corollary~\ref{c:positiveopen}). These measures in fact enjoy more regularity and symmetry than those asserted by Corollaries~\ref{t:example} and~\ref{c:positiveopen}.

The observability question for eigenfunctions can also be simply handled in account of the bounded multiplicity of the spectrum. For any non-empty open set
$\Omega_{I_1, I_2}=\{ re^{iu}, r\in I_1, u\in I_2\}\subset \ID$
(where $I_1$ is an open subset of $[0,1]$, $I_2$ an open interval of $\IS^1$), for any
eigenfunction $\psi$ of $-\Delta_D$, one has:
$$ \norm{\psi}_{L^2(\Omega_{I_1, \IS^1})}\leq C(|I_2|)  \norm{\psi}_{L^2(\Omega_{I_1, I_2})}$$
where $C(|I_2|)$ is a positive constant depending only on the size
of $I_2$. On the other hand, if $\Omega_{I_1, I_2}$ touches the
boundary ($1\in I_1$) it automatically satisfies the geometric control condition as defined in~\cite{BLR:92, Leb:92}. The results on those references imply that:
$$ \norm{\psi}_{L^2(\ID)}\leq C'(I_1)  \norm{\psi}_{L^2(\Omega_{I_1, \IS^1})}.$$
Therefore, for such $\Omega_{I_1, I_2}$, we have
$$ \norm{\psi}_{L^2(\ID)}\leq C(|I_2|)  C'(I_1)  \norm{\psi}_{L^2(\Omega_{I_1, I_2})}.$$

It is  not known to the authors whether or not any of the results of the present article could be deduced directly from the result for eigenfunctions, even when the potential vanishes identically. This does not seem to appear in the literature. On flat tori, proving observability or regularity of Wigner measures associated to the Schr\"odinger equation from the explicit expression of the solutions in terms of Fourier series requires a careful analysis of the distribution of lattice points on paraboloids \cite{JakobsonTori97, BourgainLatt} or sophisticated arguments on lacunary Fourier series \cite{JaffardPlaques, Kom:92}.
On the disk, and in absence of a potential, one could try to expand the kernel of
$e^{it\Delta_D/2}$ in terms of Bessel functions:
$$e^{it\Delta_D/2}=\sum_{n, k, \pm} e^{-it\alpha_{n, k}^2/2} |\psi_{n, k}^\pm \ra\la \psi_{n, k}^\pm|$$
and to use some of their known properties. Such an approach would anyway require some very technical work on the
spacings between the $\alpha_{n, k}$.

Here, instead, we establish directly the links between the completely integrable nature of the dynamics of the billiard flow
and the delocalization and dispersion properties of the solutions to the Schr\"odinger equation. 
Note that all results of this paper also hold for eigenfunctions of the operator $-\Delta_D + 2V(z)$ (as stationary solutions to~\eqref{e:S}). As a matter of fact, our approach is more general for it applies as well to quasimodes and clusters of eigenfunctions of the operator $-\Delta_D + 2V(z)$. The reader is referred to~\cite{ALMcras} and Remark~\ref{r:solo1} for more details on this matter.

\subsection{The semiclassical viewpoint}
\label{s:semiclass view}

In spite of the fact that our statements and proofs are formulated exclusively in terms of the non-semiclassical Schr\"odinger equation \eqref{e:S}, our results do have an interpretation in the light of the semiclassical limit for the Sch\"odinger equation. Suppose that $v_h$ solves the semiclassical Schr\"odinger equation:
\begin{equation}\label{e:Ssc} \frac{h}{i}\frac{\partial
v_h}{\partial t}(z,t)=\left( -\frac{h^2}{2}\Delta+h^2V(ht, z)\right)
v_h (z,t),  \quad v_h \rceil_{t=0} = u^0 .
\end{equation}
It turns out that $ u(\cdot , t):=v_h(\cdot,t/h)$ is in fact the solution to the (non-semiclassical) Schr\"odinger equation~\eqref{e:S} with initial datum $u^0$.  As a consequence, describing properties of solutions to~\eqref{e:S} on time intervals of size of order $1$ amounts to describing properties of solutions to the semiclassical Schr\"odinger equation~\eqref{e:Ssc} up to times of order $1/h$. Our results show that the semiclassical approximation (meaning that the solution to~\eqref{e:Ssc} should be well-approximated by its initial datum propagated through the billiard flow) breaks down in time $1/h$. For instance, if we take as initial datum in~\eqref{e:Ssc} a coherent state localized at $(z_0 , \xi_0) \in \ID \times \R^2$, 
$$u_n^0 = \frac{1}{h_n^{\alpha}}\rho\left(\frac{z-z_0}{h_n^{\alpha}}\right) e^{\frac{i}{h_n}z \cdot \xi_0}, \quad \rho \in C^\infty_c(\ID),\quad \rho(0)=1 ,\quad \alpha \in (0,1), \quad h_n \to 0 ,$$ our results imply that the associated solution of~\eqref{e:Ssc} on the time interval $(0,1/h_n)$ is no longer concentrated on the billiard trajectory issued from $(z_0 , \xi_0)$. Instead, we show that it spreads on the disk $\ID$ (the associated measure is absolutely continuous) and it leaves a positive mass on any set touching the boundary (even if the trajectory of the billiard issued from $(z_0,\xi_0)$ avoids this set).

Hence, our analysis goes far beyond the well-understood  semiclassical limit for times of order $1$, or even of order $\log(1/h)$ (known as  the Ehrenfest time, see~\cite{BouzouinaRobert}). Such a long time analysis is possible thanks to the complete integrability of the system. In fact, in the paper~\cite{AnantharamanKMacia}, which deals with the Schr\"odinger equation (and more general completely integrable systems) on the flat torus, it is shown that the time scale 
$1/h$ is exactly the one at which the delocalization of solutions takes place; for chaotic systems on the contrary, the semiclassical approximation is expected to break down at the Ehrenfest time~\cite{NAQUE,AnRiv, AMSurvey}.

\subsection{The structure theorem}
We would like to stress the fact that all these results are obtained as consequences of our main theorem,
Theorem \ref{t:precise} or its variant Theorem \ref{t:preciseml}, that gives a precise description of the structure of Wigner measures arising from solutions to \eqref{e:S}.
It provides a unified
framework from which to derive simultaneously the absolute
continuity of projections of semiclassical measures (a fact that
is related to dispersive effects) on the one hand, and, on the other
hand, the observability estimates \eqref{e:oi} and \eqref{e:ob},
which are quantitative unique continuation-type properties. Since a precise statement requires the introduction of many other objects, we postpone it to Sections \ref{s:mainthsc} and \ref{s:mlst} (semiclassical and microlocal formulations respectively), and only give a rough idea of the method for the moment.

The standard construction of the Wigner measures, outlined in Section \ref{s:mlst}, allows to lift a measure $\nu$ 
to a measure $\mu^\infty$ on phase space (or $\mu_{sc}$ in the semiclassical setting): this is the associated microlocal defect measure~\cite{GerardMDM91}. The law of propagation of singularities for equation \eqref{e:S} implies that $\mu^\infty$ is invariant by the billiard flow in the disk, and we want to exploit the complete integrability of this flow. 

For this, we use action-angle coordinates to integrate the dynamics of the billiard flow and describe associated invariant tori (Section \ref{s:coord}). The angular momentum $J$ of a point $(z,\xi)$ in phase space is preserved by the flow, and so is the Hamiltonian $E = |\xi|$. The actions $J$ and $E$ are in involution and independent, except at the points of $\d \ID$ with tangent momentum. The angle $\alpha$ that a trajectory makes when bouncing on the boundary is a also preserved quantity (in fact a function of $J/E$).  
 The key point of our proof is to analyze in detail the possible concentration of sequences on the 
 sets $\cI_{\alpha_0} = \{\alpha = \alpha_0\}$ of all points of phase space sharing a common incidence/reflection angle $\alpha_0$. To this aim, we perform a second microlocalization on this set, in the spirit of \cite{MaciaTorus,AnantharamanMaciaTore,AnantharamanKMacia}.
We decompose a Wigner measure as a sum of measures supported on these invariant sets. The case $\alpha_0 \not\in \pi \IQ$ corresponds to trajectories hitting the boundary in a dense set, and is trivial for us since it supports only one invariant measure. We focus on those $\cI_{\alpha_0}$ for which $\alpha_0 \in \pi \IQ$. Any trajectory of the billiard having this angle is periodic. We wish to ``zoom'' on this torus to describe the concentration of the associated measure. Assuming that the initial sequence has a typical oscillation scale of order $1/h$, we perform a second microlocalization at scale $1$, which is the limit of the Heisenberg uncertainty principle. Roughly speaking, the idea is to relocalize in the action variable $J$ at scale $1$ (i.e. $h$ times $1/h$), so that the Heisenberg uncertainty principle implies delocalization in the conjugated angle variable.  We obtain two limit objects, interpreted as {\em second-microlocal measures}. The first one captures the part of our sequence of solutions whose derivatives in directions ``transverse to the flow'' remain bounded; the second one captures the part of the solution rapidly oscillating in these directions.
Understanding the notion of transversality adapted to this problem is achieved  by constructing a flow that interpolates between the billiard flow (generated by the Hamiltonian $E$) and the rotation flow (generated by the Hamiltonian $J$).
The second measure is a usual microlocal/semiclassical measure whereas the first one is a less usual operator-valued measure taking into account non-oscillatory phenomena.
We prove that both second-microlocal measures enjoy additional invariance properties: the first one is invariant by the rotation flow, whereas the second one propagates through a Heisenberg equation on the circle. This translates, respectively, into Theorem \ref{t:precise} (ii) and (iii).
 
This program was already completed in \cite{MaciaTorus,AnantharamanMaciaTore,AnantharamanKMacia} for the Schr\"odinger equation on flat tori, but carrying it out in the disk induces considerable additional difficulties. Our phase space does not directly come equipped with its action-angle coordinates, so that we need first to change variables. This requires in particular to build a Fourier Integral Operator to switch from $(z,\xi)$-variables to action-angle coordinates. These coordinates are very nice to understand the dynamics and are necessary to perform the second microlocalization, but they are extremely nasty to treat
the boundary condition, for which the use of polar coordinates is more suitable. It seems that we cannot avoid having to go back and forth between the two sets of coordinates. 
Our approach to that particular technical aspect is inspired by \cite{GerLeich93}; however, the second-microlocal nature of the problem requires to perform the asymptotic expansions of \cite{GerLeich93} one step further.




\subsection{Relations to other works}
\label{s:comments}
\subsubsection{Regularity of semiclassical measures}
{
This work pertains to the longstanding study of the so-called
``quantum-classical correspondence'', which aims at understanding
 the links between high frequency solutions of the
Schr\"odinger equation and the dynamics of the underlying billiard flow (see for instance the survey article
\cite{AMSurvey}).

More precisely, it is concerned with a case of completely
integrable billiard flow. This particular dynamical situation has
already been addressed in \cite{MaciaTorus}
and \cite{AnantharamanMaciaTore} in the case of flat tori, and in
\cite{AnantharamanKMacia} for more general integrable systems (without boundary). 
These three papers use in a central way a ``second microlocalization'' to understand the concentration of measures on invariant tori. The main tools are second-microlocal semiclassical measures, introduced in the local Euclidean setting in~\cite{NierScat, Fermanian2micro, FermanianShocks, FermanianGerardCroisements,MillerThesis,Miller:97}, and defined in~\cite{MaciaTorus,AnantharamanMaciaTore,AnantharamanKMacia} as global objects.

On the sphere $\mathbb{S}^d$, or more generally, on a manifold with periodic geodesic flow, the situation is radically different. The geodesic flow for this type of geometries is still completely integrable, but it is known~\cite{MaciaAv} (see also~\cite{JZ:99,MaciaZoll,AM:10} for the special case of eigenfunctions) that every invariant measure is a Wigner measure; those are not necessarily absolutely continuous when projected in the position space. The difference with the previous situation is that the underlying dynamical system, though completely integrable, is {\em degenerate}.
What was evidenced in \cite{AnantharamanKMacia} is that a sufficient and necessary for the absolute continuity of Wigner measures, is that the hamiltonian be a {\em strictly convex/concave} function of the action variables -- a condition that is even stronger than non-degeneracy.
In the case of the disk, the complete integrability of the billiard flow on $\ID$ degenerates on the boundary. There, both actions coincide, which allows for the concentration of solutions on the invariant torus at the boundary (as was the case with the aforementioned whispering gallery modes).

Note that on the torus and on the disk, it remains an open question to fully characterize the set of Wigner measures associated to sequences of solutions to the time-dependent Schr\"odinger equation. In the case of flat tori, the papers \cite{JakobsonTori97, AJM} provide additional information about the regularity of the measures. 


\subsubsection{Observability of the Schr\"odinger equation}\label{s:ointroobsschrod}
Since the pioneering work of Lebeau~\cite{Leb:92}, it is known
that observability inequalities like \eqref{e:oi}-\eqref{e:ob}
always hold if all trajectories of the billiard enter the
observation region $\Omega$ or $\Gamma$ in finite time. However,
since~\cite{Har:89plaque, JaffardPlaques}, we know that this
strong geometric control condition is not necessary:~\eqref{e:oi}
holds on the two-torus as soon as $\Omega \neq \emptyset$; for different proofs and extensions of this result see \cite{Kom:92,BurqZworski04,MaciaDispersion,BZ:12,BBZ14,AnantharamanMaciaTore}. These
properties seem to deeply depend on the global dynamics of the
billiard flow.

On manifolds with periodic geodesic flow, it is {\em necessary} that $\overline{\Omega}$ meets all geodesics for an observation inequality as \eqref{e:oi}
to hold \cite{MaciaDispersion}. This is due to the strong stability properties of the geodesic flow.

To our knowledge, apart from the case of flat tori, few results are known concerning the
observability of the Schr\"odinger equation in situations where
the geometric control condition fails. The paper \cite{AnantharamanKMacia} extends \cite{AnantharamanMaciaTore} to general completely integrable
systems under a convexity assumption for the hamiltonian. Note also that the boundary observability
\eqref{e:ob} holds in the square if (and only if) the observation
region $\Gamma$ contains both a horizontal and a vertical nonempty
segments~\cite{RTTT:05}. Finally, for chaotic systems, the observability
inequality~\eqref{e:oi} is also valid on manifolds with negative
curvature if the set of uncontrolled
trajectories is sufficiently small \cite{NAQUE,AnRiv}.

Our Theorems~\ref{t:obs-i} and \ref{t:obs-b}
provide necessary and sufficient conditions for the observability
of the Schr\"odinger group on the disk. This is clear in the case
of boundary observability, and in the case of internal
observability, if $\Omega \subset \ID$ is such that $\Omega \cap
\d \ID=\emptyset$, the observability inequality~\eqref{e:oi}
fails. When $V=0$ this comes from the existence of 
whispering-gallery modes, see Section~\ref{s:eigenfcts}, and this remains true for any $V$, as proved in Remark~\ref{rem:nonobsnonbord}.

\bigskip
Let us conclude this introduction with a few more remarks.
\begin{remark}
In this article, we only treat the case of Dirichlet boundary conditions. The extension of our
method to the Neumann or mixed boundary condition deserves further
investigation.
\end{remark}

\begin{remark}
Let us comment on the regularity required on the potential $V$. 
Arguments developed in \cite{AnantharamanMaciaTore} show that
all the results of this paper could actually be weakened to
$V\in C^0\left(  \mathbb{R\times \ovl{\ID}}; \R \right)$ or even
to the case where $V$ is continuous outside a set of zero
measure. Corollary~\ref{t:example} in fact also holds for any $V \in L^2_{\loc}(\R ; \mathcal{L}(L^2(\ID)))$, and in particular for any bounded potentials. See also Remark~\ref{r:solo1} below.
\end{remark}

\begin{remark}
Our results directly yield a polynomial decay rate for
the energy of the damped wave equation $(\d_t^2 -\Delta
+b(z)\d_t)u=0$ with Dirichlet Boundary conditions on the disk.
More precisely, \cite[Theorem~2.3]{AnL} and Theorem~\ref{t:obs-i}
imply that if $b\geq0$ is positive on an open set $\Omega$
such that $\Omega\cap \partial \ID \not= \emptyset$, then the
$H^1_0 \times L^2$ norm of solutions decays at rate
$1/\sqrt{t}$ for data in $(H^2 \cap H^1_0)\times H^1_0$. This rate
is better than the {\em a priori} logarithmic decay rate given by the
Lebeau theorem~\cite{Leb:96}. The latter is however optimal if $\supp(b) \cap \d \ID =¬†\emptyset$ as a consequence of the whispering gallery phenomenon (see e.g.~\cite{LR:97}).
\end{remark}

\bigskip
\noindent\textbf{Acknowledgement.} We thank Patrick G\'erard for
his continued encouragement and very helpful discussions.

NA and ML are supported by the Agence Nationale de la Recherche under grant GERASIC ANR-13-BS01-0007-01.

NA acknowledges support by the National Science Foundation under
agreement no. DMS 1128155, by the Fernholz foundation, by Agence
Nationale de la Recherche under the grant ANR-09-JCJC-0099-01, and
by Institut Universitaire de France. Any opinions, findings and
conclusions or recommendations expressed in this material are
those of the authors and do not necessarily reflect the views of
the NSF.

FM  takes part into the visiting faculty program of ICMAT
and is partially supported by grants MTM2013-41780-P (MEC) and ERC Starting Grant 277778.

\section{The structure Theorem\label{s:preliminaries}}
In this section, we give the main definitions used in the article and state our main structure theorems.
We first define microlocal and semiclassical Wigner measures (which are the main objects discussed in the paper) in Section~\ref{s:versus}.
We then briefly describe the billiard flow and introduce adapted action-angle coordinates in Section~\ref{s:billiard}. This allows us to formulate our main results (Sections~\ref{s:mainthsc} and~\ref{s:mlst}), both in the semiclassical (Theorem~\ref{t:precise}) and in the microlocal (Theorem~\ref{t:preciseml}) framework. Next, in Section~\ref{s:appregularity}, we define various measures at
the boundary of the disk, that will be useful in the proofs, and explain their links with the Wigner measures in the interior.

\subsection{Wigner measures: microlocal versus semiclassical point of view\label{s:versus}}
Let $T^* \R^2=\R^2\times \R^2$ be the cotangent bundle over
$\R^2$, and $T^* \R=\R\times \R$ be the cotangent bundle over
$\R$. We shall denote by $z\in \R^2$ (resp. $t \in \R$) the space
(resp. time) variable and $\xi \in \R^2$ (resp. $H\in \R$) the
associated frequency.

Our main results can be formulated in two different and
complementary settings. We first introduce the symbol class needed
to formulate their microlocal version, allowing to define {\em
microlocal} Wigner distributions. We then define {\em
semiclassical} Wigner distributions and briefly compare these two
objects.

\begin{definition}
 Let us call $\cS_0$ the space of functions $a\in C^\infty (T^*\R^2\times T^* \R)$, $a(z, \xi, t, H)$ such that
\begin{enumerate}
\item[(a)] $a$ is compactly supported in the variables $z, t$.
\item[(b)] $a$ is homogeneous at infinity in $(\xi, H)$ in the
following sense: there exists $R_{0}>0$ such that
\begin{equation}\label{e:homogene}
a(z, \xi, t, H)   =a(z, \lambda \xi, t, \lambda^2 H) \text{,\quad for } |\xi|^2 +| H| >R_{0} \text{ and } \lambda\geq 1.%
\end{equation}
Equivalently, there is $a_{{\rm{hom}}}\in C^{\infty}\left(T^*\R^2\times T^* \R \setminus\{ (\xi, H)=(0, 0)\}  \right)$ satisfying \eqref{e:homogene} for all $\lambda >0$, such that%
\[
a(z, \xi, t, H)   =a_{{\rm{hom}}}\left(z, \xi, t, H   \right)  \text{,\quad for } |\xi|^2 +| H| >R_{0} .%
\]
Such a homogeneous function $a_{{\rm{hom}}}$ is entirely
determined by its restriction to the set $\{|\xi|^2+
2|H|=2\}\subset \R^2\times \R$, which is homeomorphic to a
$2$-dimensional sphere $\IS^2$. Thus we may (and will, when
convenient) identify $a_{{\rm{hom}}}$ with a function in the space
$C^{\infty}\left( \R^2_{z}\times \R_t\times \IS^2_{\xi, H}
\right)$.
\end{enumerate}
\end{definition}
Note that the different homogeneity with respect to the $H$ and
$\xi$ variables is adapted to the scaling of the Schr\"odinger
operator.

Let $(u^0_{n})$ be a sequence in $L^{2}(\ID)$, such that
$\norm{u^0_n}_{L^{2}(\ID)}=1$ for all $n$. For $z\in\ID$ and
$t\in\R$ we denote $u_n(z,t)= U_V(t)u_n^0(z)$. In what follows
(e.g. in formula \eqref{e:Wml} below), we shall systematically extend
the functions $u_n$, a priori defined on $\ID$, by the value $0$
outside $\ID$ as done in~\cite{GerLeich93}, where
semiclassical Wigner measures for boundary value problems were
first considered. The extended sequence now satisfies the equation
\begin{equation*}\left(-\frac{1}2\Delta
  +V- D_t\right)u_n= \frac{1}2 \frac{\partial u_n}{\partial n}\otimes \delta_{\partial \ID}   ,
 \quad (z,t) \in \R^2 \times \R,
  \end{equation*}
where $\Delta$ denotes the Laplacian on $\R^2$. Remark that the term $\frac{\partial u_n}{\partial n}\rceil_{\d\ID}$ has no straightforward meaning at this level of regularity. We shall see below how to give a signification to this equation, both in the semiclassical (see Remark~\ref{r:H1}) and in the microlocal (see Section~\ref{s:micromesboundary}) settings).

\medskip
The {\em microlocal Wigner distributions} associated to $(u_n)$
act on symbols $a \in \cS_0$ by
\begin{equation}\label{e:Wml}
W_{u_n}(a):=\la u_n, \Op_1(a)u_n\ra_{L^2(\R^2_z\times\R_t)},
\end{equation}
where $\Op_{1}(a) = a (z, D_z, t, D_t)$ (with the standard notation $D = - i \d$) is a pseudodifferential
operator defined by the standard quantization procedure. In what
follows, $\Op_{\eps}(a)= a(z, \eps D_z, t, \eps D_t)$ will stand
for the operator acting on $L^{2}(\IR^2\times \R)$ by:
\begin{equation}\label{e:defop}
\big(\Op_{\epsilon}(a) u \big)(z, t) = 
 \frac{1}{(2\pi \eps)^3} \int_{\R^2\times \R} \int_{\R^2\times \R}
 e^{\frac{i\xi \cdot (z-z')+iH(t-t')}\eps} a \left( z , \xi, t, H \right) u(z',t') \,  dz'dt' \, d\xi dH.
\end{equation}
Usual estimates on pseudodifferential operators imply that
$W_{u_n}$ is well defined, and forms a bounded sequence in
$\cS_0'$. The main goal of this article is to understand
properties of weak limits of $(W_{u_n})$ that are valid for any
sequence of initial conditions $(u_n^0)$.

\medskip
The problem also has a  semiclassical variant. In this version,
one considers $a\in C_c^\infty (T^*\R^2\times T^*\R)$, a real
parameter $h>0$, and one defines the {\em semiclassical Wigner
distributions} at scale $h$ by
\begin{equation}
W^h_{u_n}(a):=\la u_n, \Op_1(a(z, h\xi, t,
h^2H))u_n\ra_{L^2(\R^2_z\times \R_t)},\label{e:Wsc}
\end{equation}
where $\Op_1(a(z, h\xi, t, h^2H)) = a(z, hD_z, t, h^2
D_t)=\Op_h(a(z, \xi, t, hH))$, see \eqref{e:defop}. Note that this
scaling relation is the natural one for
solutions of \eqref{e:S}, and its interest will be made clear
below. Again $W^h_{u_n}$ is well defined, and forms a bounded
sequence in $\cD'(T^*\R^2\times T^*\R)$ if $h$ stays bounded. This
formulation is most meaningful if the parameter $h=h_n$ is chosen
in relation with the typical scale of oscillation of our sequence
of initial conditions $(u_n^0)$.

\begin{definition}
Given a bounded sequence $\left(w_{n}\right)$ in $L^2(\ID)$, we shall say that it is $(h_n)$-oscillating from above (resp. $(h_n)$-oscillating from below) if the sequence $(w_n)$ extended by zero outside of $\ID$ satisfies:
\[
\lim_{R\rightarrow\infty}\limsup_{n\rightarrow\infty}\int_{\left\vert
\xi\right\vert \geq R/h_{n}}\left\vert \widehat{w_{n}}\left(
\xi\right)\right\vert ^{2}  d\xi=0,
\]
(resp. \[ \lim_{\eps\rightarrow
0}\limsup_{n\rightarrow\infty}\int_{\left\vert \xi\right\vert \leq
\eps/h_{n}}\left\vert \widehat{w_{n}}\left( \xi\right)\right\vert
^{2}  d\xi=0 \quad)
\]
where $\widehat{w_{n}}$ is the Fourier transform of $w_n$ on
$\R^2$.
\end{definition}

The property of being $(h_n)$-oscillating from above is only
relevant if $h_n\To 0$; if $u_n^0$ is $(h_n)$-oscillating for
$(h_n)$ bounded away from $0$, the (extended) sequence $(u_n^0)$
is compact in $L^2$ and the structure of the accumulation points
of $(W^{h_n}_{u_n})$ is trivial. Therefore, we shall always assume
that $h_n\To 0$. Note that one can always find $(h_n)$ tending to
zero such that $(u_n^0)$ is $h_n$-oscillating from above (to see
that, note that for fixed $n$ one may choose $h_n$ such that
$\int_{\left\vert \xi\right\vert \geq 1/h_{n}}\left\vert
\widehat{u_{n}^0}\left( \xi\right)\right\vert ^{2}  d\xi\leq
n^{-1}$). However, the choice of the sequence $h_n$ is by no means
unique ($h_n$-oscillating sequences are also $h'_n$-oscillating as
soon as $h'_n\leq h_n$), although in many cases there is a natural
scale $h_n$ given by the problem under consideration.

One can find $(h'_n)$ such that
$(u_n^0)$ is $h'_n$-oscillating from below if and only if the
extended $(u_n^0)$ converges to $0$ weakly in $L^2(\R^2)$. It is
not always possible to find a common $(h_n)$ such that $(u_n^0)$
is $h_n$-oscillating both from above and below (see
\cite{GerardSobolev} for an example of a sequence with this
behavior). However, when it is the case, the semiclassical Wigner
distributions contain more information that the microlocal ones
(see Section~\ref{s:link}). On the other hand, if no $h_n$ exists such that
$(u_n^0)$ is $h_n$-oscillating from above and below, the
accumulation points of $W^{h_n}_{u_n}$ may fail to describe
completely the asymptotic phase-space distribution of the sequence
$(u_n)$, either because some mass will escape to $|\xi|=\infty$ or
because the fraction of the mass going to infinity at a rate
slower that $h_n^{-1}$ will give a contribution concentrated on
$\xi=0$. In those cases, the microlocal formulation is still able
to describe the asymptotic distribution of the sequence on the
reduced phase-space $\R^2_{z}\times \R_t\times \IS^2_{\xi, H}$.

This is one of the motivations that has lead us to study both
points of view, semiclassical and microlocal.

\subsection{The billiard flow\label{s:billiard}}

Microlocal or semiclassical analysis provide a connection between
the Schr\"odinger equation and the billiard on the underlying phase space. 
 In this section we first clarify what we mean by ``billiard flow'' in the disk.
 The phase space associated with the billiard flow on the disk can be defined as a quotient of $\ovl{\ID}\times \R^2$ (position $\times$ frequency). We first define the symmetry with respect to the line tangent to the circle $\partial \ID$ at $z \in \partial \ID$ by
$$
\sigma_z(\xi) = \xi - 2(z\cdot \xi)z, \quad \sigma(z,\xi) =(z,
\sigma_z(\xi)) ,\quad z \in \partial \ID .
$$
Then, we work on the quotient space
$$
\W = \ovl{\ID}\times \R^2 / \sim \quad \text{where } (z, \xi) \sim
\sigma(z,\xi) \text{ for } |z| = 1 .
$$
We denote by $\pi$ the canonical projection $\ovl{\ID}\times \R^2
\to \W$ which maps a point $(z, \xi)$ to its equivalence class
modulo $\sim$. Note that $\pi$ is one-one on $\ID\times \R^2$, so
that $\ID\times \R^2$ may be seen as a subset of $\W$.

A function $a \in C^0(\W)$ can be identified with the function
$\tilde{a} = a \circ \pi \in C^0(\ovl{\ID}\times \R^2)$ satisfying
$\tilde{a} (z, \xi) = \tilde{a}\circ \sigma (z, \xi)$ for $(z,
\xi) \in \d \ID\times \R^2$.

The billiard flow $(\phi^\tau)_{\tau\in \R}$ on $\W$ is the
(uniquely defined) action of $\R$ on $\W$ such that the map
$(\tau, z, \xi)\mapsto \phi^\tau(z, \xi)$ is continuous on
$\IR\times \W$, satisfies
$\phi^{\tau+\tau'}=\phi^\tau\circ\phi^{\tau'}$, and such that
$$\phi^\tau(z, \xi)=(z+\tau\xi, \xi)$$
whenever $z\in\ID$ and $z+\tau\xi\in\ID$.

\medskip
In order to understand how the completely integrable dynamics of
the flow $\phi^\tau$ influences the structure of Wigner measures,
we need to introduce coordinates adapted to this dynamics. We
denote by
\begin{equation}
\label{e:def Phi}
\Phi~:(s,\theta,E,J)\mapsto(x,y,\xi_{x},\xi_{y}),
\end{equation}
the set of ``action-angle'' coordinates for the billiard flow (see
also Section \ref{s:coord}), defined by:

\begin{equation*}
\begin{cases}x=\frac{J}{E}\cos\theta-s\sin\theta,\\
y=\frac{J}{E}\sin\theta+s\cos\theta,\\
\xi_x=-E\sin\theta,\\
\xi_y=E\cos\theta.
 \end{cases}
\end{equation*}
These coordinates are illustrated in
Figure~\ref{fig:aacoordinates}. 
The inverse map is given by the formulas
\[%
\begin{cases}
E=\sqrt{\xi_{x}^{2}+\xi_{y}^{2}},\mbox{ (velocity)}\\
J=x\xi_{y}-y\xi_{x},\mbox{ (angular momentum)}\\
\theta=-\arctan\left(  \frac{\xi_{x}}{\xi_{y}}\right),
\mbox{ (angle of $\xi$ with the vertical)}\\
s=-x\sin\theta+y\cos\theta,\mbox{ (abscissa of }(x,y)\mbox{ along
the line $\left(\frac{J}{E}\cos\theta,
\frac{J}{E}\sin\theta\right) + \R\xi$ ).}
\end{cases}
\]
In other words, we have:
\begin{equation*}
\begin{cases}E = |\xi| , \\ J = z \cdot \xi^\perp ,\\
s=z\cdot\frac{\xi}{|\xi|},
\end{cases}
\end{equation*}
where $\xi^{\perp}:= ( \xi_{y},-\xi_{x} )$, and
\begin{equation*}
\begin{cases}
\xi = (\xi_x ,\xi_y) = E(-\sin(\theta) , \cos(\theta)), \smallskip\\

z   = (x,y)          = s(-\sin(\theta) , \cos(\theta)) +
\frac{J}{E}(\cos(\theta), \sin(\theta))   =
\left(z\cdot\frac{\xi}{|\xi|}\right)\frac{\xi}{|\xi|} +
   \left(z\cdot\frac{\xi^\perp}{|\xi|}\right)\frac{\xi^\perp}{|\xi|}.
\end{cases}
\end{equation*}

\begin{figure}[h!]
  \begin{center}
    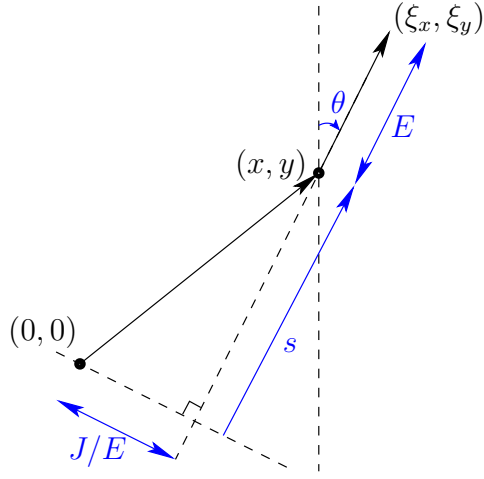
    \caption{Action-angle coordinates for the billiard flow on the disk.}
    \label{fig:aacoordinates}
  \end{center}
\end{figure}
Note that the velocity $E$  and the angular
momentum $J$ are preserved along the free transport flow in $\R^2
\times \R^2$, but also along $\phi^\tau$; the variables $s$ and
$\theta$ play the role of ``angle'' coordinates. We call
$\alpha=-\arcsin\left(\frac{J}E\right)$ the angle that a billiard
trajectory makes with the normal to the circle, when it hits the
boundary (see Figure~\ref{fig:angle-disk}). The quantity $\alpha$
is preserved by the billiard flow.

\begin{figure}[h!]
  \begin{center}
    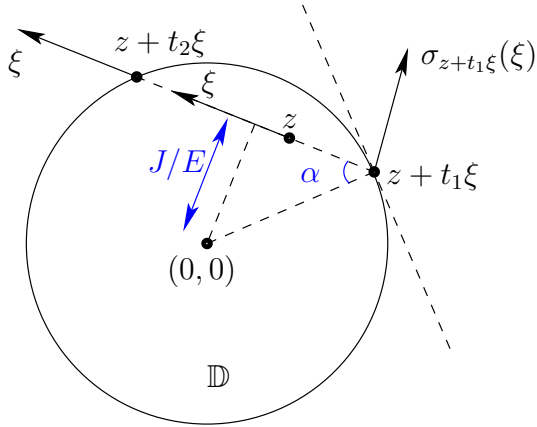
    \caption{Angle $\alpha$.}
    \label{fig:angle-disk}
  \end{center}
\end{figure}

Let us denote $T_{(E, J)}$ the level sets of the pair $(E, J)$,
namely
\begin{equation}
T_{(E, J)} = \{(z,\xi) \in \ovl{\ID} \times \R^2 \;:\; (|\xi|, z
\cdot \xi^\perp) =  (E,J) \}.
\end{equation}
For $E\not=0$ let us denote $\lambda_{E, J}$ the probability
measure on $T_{(E, J)}$ that is both invariant under the billiard
flow and invariant under rotations. In the coordinates $(s,
\theta, E, J)$, we have
$$
\lambda_{E, J}(ds, d\theta)  = c(E,J)ds  d\theta , \quad c(E,J) =
\left(\int_{T(E,J)}ds d\theta \right)^{-1}>0 .
$$
Note that for $E\not= 0$ and $\alpha\in \pi \IQ$ the billiard flow
is periodic on $T_{(E, J)}$ whereas $\alpha\not\in \pi \IQ$
corresponds to trajectories that hit the boundary on a dense set.
More precisely, if $\alpha\not\in \pi \IQ$ then the billiard flow
restricted to $T_{(E, J)}$ has a unique invariant probability
measure, namely $\lambda_{E, J}$.

\subsection{Standard facts about Wigner measures\label{s:prelim-sc}}
  We start formulating the question and results in a semiclassical framework: we have a parameter $h_n$
going to $0$, meant to represent the typical scale of oscillation
of our sequence of initial conditions $(u_n^0)$.

 We simplify the notation by writing $h=h_n$, $u_h^0=u_n^0$.
We will always assume that the functions $u_{h}^0$ are normalized
in $L^{2}(\ID)$. We define $u_h(z, t)=U_V(t) u_h^0 (z)$  (the reader should be aware that $u_h$ satisfies the classical Schr\"odinger Equation~\eqref{e:S}; the index $h$ only reminds its oscillation scale). Since
this is a function on $\ID\times \R$ it is natural to do a
frequency analysis both in $z$ and $t$. Recall that we keep the
notation $u_h$ after the extension by zero outside $\ID$. Recall
that the \emph{semiclassical Wigner distribution} associated to
$u_{h}$ (at scale $h$) is a distribution on the cotangent bundle
$T^{\ast}\IR^2\times T^{\ast}\IR=\IR^2_z\times \IR^2_\xi\times
\R_t\times \R_H$, defined by
\begin{equation}
W_{u_h}^h : a\mapsto
 \left\langle
u_{h},\Op_{1}(a(z, h\xi, t, h^2 H))u_{h}\right\rangle
_{L^{2}(\R^2\times \R)},\qquad \mbox{ for all }a\in
C_{c}^{\infty}(T^{\ast}\IR^2\times T^{\ast}\IR).\label{e:defW}
\end{equation}

The scaling $\Op_{1}(a(z, h\xi, t, h^2 H))$ is performed in order
to capture all the information whenever $u_h$ is $h$-oscillating from above
(if $u_h$ is not $h$-oscillating from above, the discussion below remains
entirely valid but part of the information about $u_h(z,t)$ is
lost when studying $W^h_{u_h}(a)$). Under this assumption,  if $a$ is a function on $T^{\ast}\IR^2\times T^{\ast}\IR $ that depends only on $(z, t)$, we have
\begin{equation}
W^h_{u_h}(a)=\int_{\ID%
}a(z, t)|u_{h}(z, t)|^{2}dzdt . \label{e:proj}%
\end{equation}
When no confusion arises, we shall denote $W_{h}$ for $W^h_{u_h}$.

By standard estimates on the norm of $\Op_{1}(a)$, it follows that $W_h$ belongs to $\cD^{\prime}\left(  T^{\ast}\IR^2\times T^{\ast}\IR\right)  $, and is uniformly bounded in that space as
$h\To0^{+}$. Thus, one can extract subsequences that converge in
the weak-$\ast$ topology of $\cD^{\prime}\left(
T^{\ast}\IR^2\times T^{\ast}\IR\right)$. In other words, after
possibly extracting a subsequence, we have
\begin{equation}\label{e:wlimite}
W_h(a)\Lim_{h\To0} \mu_{sc}(a)\end{equation} for all $a\in
C_c^\infty\left(  T^{\ast}\IR^2\times T^*\R\right)$.

In this paper such a measure $\mu_{sc}$ will be called a {\em
semiclassical Wigner measure}, or in short {\em semiclassical
measure}, associated with the initial conditions $(u_h^0)$ and the
scale $h$.

\begin{remark} \label{r:H1}Fix $R>0$; Remark \ref{r:osc} below tells us that in order to compute the restriction of $\mu_{sc}$ to the set $\{|H|<R\}$ we may, without loss of generality, assume that $u_h^0\in H_0^1(\ID)$ and $\norm{\nabla u_h^0}_{L^2(\ID)}=O_R(h^{-1})$. In that case Proposition \ref{p:regboundary} says that the boundary data
$h\partial_{n}\left( U_V(t){u}_{h}^0\right) $ form a bounded
sequence in $L_{\loc}^{2}\left(\mathbb{R}\times \partial\ID
 \right)  $. We can work under these assumptions when necessary. This determines $\mu_{sc}$ completely as $R$ is arbitrary.
\end{remark}

It follows from standard properties of pseudodifferential
operators that the limit $\mu_{sc}$ in \eqref{e:wlimite} has the
following properties:

\begin{itemize}
\item $\mu_{sc}$ is a nonnegative measure, of the form $\mu_{sc}(dz,
d\xi, dt, dH)=\mu_{sc}(dz, d\xi, t, dH)dt$ where $t\mapsto \mu_{sc}(t) \in
L^{\infty}(\R_t;\cM_{+}(T^* \IR^2\times \R_H))$. Moreover, for
a.e. $t\in \R$, $\mu_{sc}(t)$ is supported in $\{|\xi|^2=2H\} \cap
\left(\ovl{\ID} \times \IR^2\times \R_H\right)$. See \cite{GMMP,
LionsPaul} for a proof of nonnegativity; the time regularity and the localization of the support are shown
in Proposition \ref{p:regt}.

\item From the normalization of $u_{h}^0$ in $L^2$, we have for
a.e. $t$:
$$\int_{\ovl{\ID}\times \R^2 \times \R}\mu_{sc}(dz,d\xi, t,
dH)\leq1,$$ the inequality coming from the fact that
$\ovl{\ID}\times \R^2\times\R$ is not compact, and that there may
be an escape of mass to infinity (however, if $u_h^0$ is
$h$-oscillating from above, escape of mass does not occur and we
have $\int_{\ovl{\ID}\times \R^2 \times \R}\mu_{sc}(dz,d\xi, t,
dH)=1$).

\item The standard quantization enjoys the following property:
\begin{equation}
\left[-\frac{ih}2\Delta,
\Op_{h}(a)\right]=\Op_{h}\left(\xi\cdot\partial_z
a-\frac{ih}2\Delta_z a\right) ,
\label{e:weylcomm}
\end{equation}
where $\Delta$ is the Laplacian on $\R^2$. From this identity, one
can show that
\begin{equation}\label{e:transport}\int_{\ovl{\ID}\times
\R^2\times\R}\xi\cdot\partial_z a\,\, \mu_{sc}(dz,d\xi, t,
dH)=0\end{equation} for a.e. $t$ and for every smooth $a$ such
that $a(z, \xi, t, H)=a(z, \sigma_z(\xi), t, H)$ for $|z|=1$.
Equivalently,
$$\int_{\ovl{\ID}\times \R^2\times\R}  a\circ \phi^\tau \circ \pi (z, \xi,t, H)
\mu_{sc}(dz,d\xi, t, dH)=\int_{\ovl{\ID}\times \R^2\times\R}  a \circ
\pi (z, \xi, t, H) \mu_{sc}(dz,d\xi, t, dH)$$ for every $a\in C^0(\W)$,
$\tau\in \R$ -- where $ \phi^\tau$ is the billiard flow in the
disk and $\pi :  \ovl{\ID}\times \R^2 \to \W$ the canonical
projection, defined in Section~\ref{s:billiard}. In other words, $\pi_*
\mu$ is an invariant measure of the billiard flow.

We refer to Section~\ref{s:mesboun} for a more general version
of~\eqref{e:weylcomm} (as formulated in~\cite{GerLeich93}, see also~\cite{RZ:09})
involving a measure associated to boundary traces.

 \end{itemize}

\subsection{The structure theorem: semiclassical formulation\label{s:mainthsc}}

Now comes our central result, giving the structure of
semiclassical measures arising as weak-$\ast$ limits of sequences
$\left(  W_h\right)  $ associated to solutions of \eqref{e:S}. As
a by-product it clarifies the dependence of $\mu_{sc}(t,\cdot)$ on
the time parameter $t$. It was already noted in \cite{MaciaAv}
that the dependence of $\mu_{sc}(t,\cdot)$ on the sequence of
initial conditions is a subtle issue.

The statement of Theorem \ref{t:precise} is technical and needs
introducing some notation. We first restrict our attention to the
case where the initial conditions $(u_h^0)$ are $h$-oscillating
from below, or equivalently $\mu_{sc}$ does not charge $\{\xi=0\}$
(otherwise, the restriction of $\mu_{sc}$ to $\{\xi=0\}$ will be
better understood at the end of Section~\ref{s:link}).

The notation $(s, \theta, E, J)$, $\alpha$ is as in Section~\ref{s:billiard}. 
Here we restrict our discussion to $E\not=0$.
For each $\alpha_0\in \pi \IQ\cap (-\pi/2, \pi/2)$ we will
introduce a flow $(\phi^\tau_{\alpha_0})$ on the billiard phase
space $\W$, all of whose orbits are periodic (Lemma
\ref{l:modes}). It coincides with the billiard flow on the set
$$
\cI_{\alpha_0}=\{(s, \theta, E, J) \in \Phi^{-1}(\ovl{\ID} \times \R^2) , J=-\sin\alpha_0 E\} = \{\alpha = \alpha_0\},
$$ which is the union of all the lagrangian manifolds $T_{(E, J)}$ with $J=-\sin\alpha_0 E$.
If $a$ is a function on $\W$, we shall denote by $\la
a\ra_{\alpha_0}$ its average along the orbits of
$\phi^\tau_{\alpha_0}$ (actually, $\la a\ra_{\alpha_0}$ is well
defined even if $a$ is not symmetric with respect to the boundary,
since the set of hitting times of the boundary has measure $0$).
In the coordinates $(s, \theta, E, J)$, this is a function whose
restriction to $\cI_{\alpha_0}$ does not depend on $s$.

We will denote by
$$
m_{ a }^{\alpha_0}(s, E, t, H)
$$
the operator on $L^2_{\loc}( \R_{\theta})$ acting by
multiplication by the function
$$
a \left( \Phi(s, \theta,  E, - E \sin\alpha_0), t,H \right).
$$
If $a$ is a symmetric function (i.e. a function on $\W$), remark
that $m_{\left\langle a\right\rangle _{\alpha_0}}^{\alpha_0}$ does
not depend on the variable $s$. For our potential $V$, the
function $\left\langle V\right\rangle _{\alpha_0}\circ\Phi$
depends only on $\theta$ (and $t$ if $V$ is time-dependent).

Given $\omega\in \R/2\pi\Z$, we denote by $U_{\alpha_0,
\omega}(t)$ the unitary propagator of the equation
\[
-\cos^2\alpha_0 D_t v(t,\theta)+\left(-\frac{1}{2}\partial_\theta^2+
\cos^2\alpha_0\la
V\ra_{\alpha_0}\circ\Phi\right) v(t,
\theta)=0
\]
acting on the Hilbert space
$$ \cH_\omega= \{v \in L^2_{\loc}(\IR) : v(\theta + 2\pi) =
v(\theta)e^{i\omega}, \text{ for a.e. } \theta \in \R \} ,
$$
i.e. with Floquet-periodic condition. In the statements below, each
$\cH_\omega$ is identified with $L^2(0, 2\pi)$ by taking
restriction of functions to $(0, 2\pi)$.

\begin{theorem}
\label{t:precise} Let $(u^0_{h})$ be a family of initial data,
assumed to be $h$-oscillating from below. There exists a 
subsequence of $W_h $ converging weakly-$\ast$ to a
positive measure $\mu_{sc}$ that can be decomposed into a countable sum of non-negative measures: 
$$\mu_{sc}=\nu_{Leb}+\sum_{\alpha_0\in \pi\IQ\cap[-\pi/2, \pi/2]}\nu_{ \alpha_0},$$
satisfying:

\begin{itemize}
\item[(i)] Each of the measures in the decomposition above is carried by the
set $\{H=\frac{E^2}2\}$ and is invariant under the billiard flow.

\item[(ii)] The measure $\nu_{Leb}$ is constant in $t$;
$\nu_{Leb}$ is of the form $\int_{E>0, |J|\leq E} \lambda_{E, J}
d\nu'(E, J) dt$ for some nonnegative measure $\nu'$ on $\R^2$. In
other words $\nu_{Leb}$ is a combination of Lebesgue measures on
the invariant ``tori'' $T_{(E, J)}$.

\item[(iii)] For every $\alpha_0\in \pi\IQ\cap(-\pi/2, \pi/2)$, the measure $\nu_{\alpha_0}$ is carried by the set
$\cI_{\alpha_0}\cap\{H=E^2/2\}$ and is characterized by
\begin{align}
\label{e:structnualpha}
\int_{\cI_{\alpha_0}} a  \, d\nu_{\alpha_0} = \int_{\cI_{\alpha_0}} \Tr_{L ^{2}(0, 2\pi)}\left(
m_{\la a\ra_{\alpha_0}  }^{\alpha_0}\,\sigma_{\alpha_0 }\,\right)
d \ell_{\alpha_0}, \quad \text{for all } a\in C^\infty_c(T^*\R^2 \times T^*\R), 
\end{align}
where 
$\ell_{\alpha_0}(d\omega, dE,  dH, dt)$ is a nonnegative measure 
on $\R/2\pi\Z\times \R_E\times \R_H
\times\R_t$, and 
$$
\sigma_{\alpha_0}  : (\R/2\pi\Z)_\omega \times  \R_E\times \R_H
\times\R_t \to \mathcal{L}^1_+ \big(L^2(0,2\pi)\big) ,
$$
is integrable with respect to $\ell_{\alpha_0}$, continuous with respect to $t$ and takes values in the
set of nonnegative trace-class operators on $L^2(0,2\pi)$. In addition,  $\sigma_{\alpha_0}$ satisfies, for $\ell_{\alpha_0}$-almost every $(\omega, E,H)$:
\begin{equation}
\label{e:propagsigmaell}
\sigma_{\alpha_0}(\omega, E,H , t) = U_{\alpha_0 , \omega}(t)
\sigma_{\alpha_0}(\omega, E,H , 0)U_{\alpha_0 , \omega}^*(t) .
\end{equation}
\noindent Finally, $\ell_{\alpha_0}$ and $\sigma_{\alpha_0}(\cdot , 0)$ only depend on the sequence of initial conditions $(u_h^0)$. 

\item[(iv)] For $\alpha_0=\pm\frac\pi{2}$,  $\nu_{\alpha_0}$ is a
measure that does not depend on $t$, carried by $\{H=|\xi |^2/2\}\cap T^*\d\ID$ and
 is invariant under rotations
around the origin.

\end{itemize}
\end{theorem}

\begin{remark}
\label{r:solo1}
The conclusion of the above results (as well as their counterparts in the next section) also holds for semiclassical
measures associated to sequences of approximate solutions of the
Schr\"odinger equation, i.e. satisfying
\begin{equation*}
\left( D_t +\frac{1}{2}\Delta-V(t,z)\right) u_h (z,t) =
o_{L^2_{\loc}(\ID\times\R)}(1) .
\end{equation*}
Note that, as in the torus
case~\cite{AnantharamanMaciaTore,AnantharamanKMacia},
Corollary~\ref{t:example} also holds for solutions of
\begin{equation*}
\left( D_t +\frac{1}{2}\Delta\right) u_h (z,t) =
O_{L^2_{\loc}(\ID\times\R)}(1) ,
\end{equation*}
which includes for instance the case of potentials $V \in L^\infty(\R ;\mathcal{L}(L^2(\ID)))$ (see also~\cite{Burq:13} for related results).
\end{remark}

\subsection{The structure theorem: microlocal formulation\label{s:mlst}}

We now give the microlocal version of Theorem \ref{t:precise}. The
main difference is that we now use the class of test functions
$\cS_0$ defined in Section~\ref{s:versus}.

Let $(u_n^0)$ be a sequence of initial conditions, normalized in
$L^2(\ID)$. Denote $u_n(\cdot, t):=U_V(t)u_n^0$ and recall that $u_n$ also denotes the extended function by zero to whole $\R^2$. All over the paper
we let $\chi\in C_{c}^{\infty}\left(  \mathbb{R}\right)  $ be a
nonnegative cut-off function that is identically equal to one near
the origin. Let $R>0$. For
$a\in\mathcal{S}_0 $, we define%
\[
\left\langle W_{n, R}^{\infty}   ,a\right\rangle :=
 \left\la u_n,  \Op_1\left( \left(  1-\chi\left(  \frac{|\xi|^2+ |H|}{R^2 }\right)\right)  a(z, \xi, t, H )\right) u_n\right\ra_{L^2(\IR^2 \times \IR)}  ,
\]
and%
\begin{equation}\left\langle W_{c, n, R}   ,a\right\rangle :=
 \left\la u_n,  \Op_1\left( \chi\left(  \frac{|\xi|^2+ |H|}{R^2}\right)  a( z, \xi, t, H  )\right)u_n\right\ra_{L^2(\IR^2 \times \IR)}.
\end{equation}
The Calder\'{o}n-Vaillancourt theorem~\cite{CV:71} ensures that both
$W_{n,R}^{\infty}$ and $W_{c,n, R}$ are bounded in $
\mathcal{S}_0^{\prime}   $. After possibly extracting
subsequences, we
have the existence of a limit: for every $a\in \mathcal{S}_0 $,%
\begin{equation}\label{e:convmuinfty}
  \left\langle {\mu}^{\infty
}   ,a\right\rangle :=\lim_{R\rightarrow\infty}%
\lim_{n\rightarrow +\infty}   \left\langle W_{n,R}^{\infty}
,a\right\rangle ,
\end{equation}
and%
\begin{equation}
   \left\langle {\mu}_{c
}   ,a\right\rangle  :=\lim_{R\rightarrow\infty}%
\lim_{n\rightarrow +\infty}   \left\langle
W_{c,n, R}   ,a\right\rangle .  %
\end{equation}

As a consequence, after extraction, the subsequence $W_n$ converges weakly-$\ast$ to a
limit $\mu_{ml}\in \cS_0'$, which can be
decomposed into $$\mu_{ml}=\mu^\infty+\mu_c. $$

The two limit objects $\mu_c$ and $\mu^\infty$ enjoy the following first properties:
\begin{itemize}
\item  The distribution $\mu_c$ vanishes if and only if the family $(u^0_n)$ converges weakly to $0$ in $L^2(\ID)$; in other words $\mu_c$ reflects the ``compact part'' of the sequence $(u^0_n)$, hence the subscript $_c$ in the notation.
\item The distribution ${\mu}^{\infty}
$ is nonnegative, $0$-homogeneous and supported at infinity in the
variable $(\xi, H)$ ($i.e.$, it vanishes when paired with a
compactly supported function). 
As a consequence, ${\mu}^{\infty}
$ may be identified
 with a nonnegative Radon measure on $\IR^2_z \times\R_t\times \IS^2_{\xi, H}$. 
Actually, $\mu^\infty$ is the microlocal defect measure of~\cite{GerardMDM91} (with the appropriate class of symbols $\cS_0$). 
 \item In addition, $\mu^\infty$ is of the form $\mu^\infty(dz,
d\xi, dt, dH)=\mu^\infty(dz, d\xi, t, dH)dt$ where $t\mapsto \mu^\infty(t) \in
L^{\infty}(\R_t;\cM_{+}(\IR^2_z\times \IS^2_{\xi, H} ))$. Moreover, for
a.e. $t\in \R$, $\mu^\infty(t)$ is supported in $\{|\xi|^2=2H\} \cap
(\ovl{\ID} \times \IS^2_{\xi, H} )$.
\item The projection of the distribution $\mu_{ml} = \mu_c + \mu^\infty$ on the $(z,t)$-variables is the Radon measure $\nu$ defined in the introduction (Section~\ref{sec:intro}). From the normalization of $u_{n}^0$ in $L^2$, we have for
a.e. $t$: 
$$\int_{\ovl{\ID}\times  \IS^2_{\xi, H} }\mu^\infty(dz,d\xi, t,
dH)\leq1 ;$$
if $u_n^0 \rightharpoonup 0$ in $L^2(\ID)$, then we have $\int_{\ovl{\ID}\times  \IS^2_{\xi, H} }\mu^\infty (dz,d\xi, t,
dH)=1$.
\item  The measure $\mu^{\infty}  $ satisfies the invariance property:%
\begin{equation}\label{e:transport2}
\left\la {\mu}^{\infty}  ,  \frac{\xi}{\sqrt{2H}}.\partial_z a  \right\ra =0 %
\end{equation}
for $a$ satisfying the symmetry condition $a(z, \xi,t,  H)=a(z,
\sigma_z(\xi), t, H)$ for $|z|=1$. In other words, $\pi_* \mu^\infty$ is invariant by the billiard flow.

\end{itemize}
These properties are well-known and won't be proven
in detail here (the fact that it is carried on
$\{H=\frac{|\xi|^2}2\}$ follows from Appendix~\ref{s:osc} and the proof
of invariance is essentially contained in \cite{GerLeich93} or~\cite{RZ:09}). 

\bigskip
Let us now discuss separately the finer properties of $\mu^\infty$ (high frequencies) and of $\mu_c$ (low frequencies).

%
%
%
%
%
%
%
%

We first describe $\mu^\infty$ and state the analogue of Theorem~\ref{t:precise} in the microlocal setting.
%
As previously we call $T_{(E, J)}$
the level sets of $(E, J)$ and $\cI_{\alpha_0}=\{J=-\sin\alpha_0
E\}$. The only difference with the semiclassical formalism is that
the test functions are homogeneous and thus the measure
$\mu^\infty$ is naturally defined on $\IR^2_z \times\R_t\times
\IS^2_{\xi, H} $ supported by $\ovl{\ID}\times\R_t\times
\IS^2_{\xi, H} $. The microlocal version of Theorem
\ref{t:precise} reads as follows:

\begin{theorem}
\label{t:preciseml} Let
$(u_{n}^0)$ be normalized in $L^2(\ID)$, and such that \eqref{e:convmuinfty} holds. Then the measure
$\mu^\infty$ can be decomposed as a countable sum of nonnegative
measures on $\IR^2 \times\R_t\times \IS^2 $:
$$\mu^\infty =\mu_{Leb}+\sum_{\alpha_0\in \pi\IQ\cap[-\pi/2, \pi/2]}\mu_{ \alpha_0} ,$$
satisfying:
\begin{itemize}
\item[(i)] Each of the measures in the above decomposition is carried by
$\ovl{\ID}\times\R_t\times \IS^2 \cap \{|\xi|^2=2H\}$ and by the cone $\{|J|\leq E\}$, and is invariant
under the billiard flow.

\item[(ii)] The measure $\mu_{Leb}$ does not depend on $t$;
$\mu_{Leb}$ is of the form $\int_{E>0, |J|\leq E} \lambda_{E, J}
d\mu'(E, J) dt$ for some nonnegative measure $\mu'$ on $\IS^1$ (i.e.
the set of pairs $(E, J)$ modulo homotheties).

\item[(iii)]For every $\alpha_0\in \pi\IQ\cap(-\pi/2, \pi/2)$, the measure $\mu_{\alpha_0}$ is carried by the set
$\cI_{\alpha_0}\cap\{H=E^2/2\}$ and is defined by:
\begin{align*}
\int_{\cI_{\alpha_0}} a   d\mu_{\alpha_0} = \int_{\cI_{\alpha_0}} \Tr_{L ^{2}(0, 2\pi)}\left( m_{
\langle a\rangle_{\alpha_0}  }^{\alpha_0}\,\sigma_{\alpha_0 }\,\right) d \ell_{\alpha_0},\quad \text{for all } a\in\cS_0,
\end{align*}
where $\ell_{\alpha_0}(d\omega, dE,  dH, dt)$
is a non-negative measure on
$\R/2\pi\Z\times\{E^2+2|H|=2\}\times\R_t$ carried by $\{H=E^2/2\}$ and 
$$ \sigma_{\alpha_0}  : (\R/2\pi\Z)_\omega \times
\{E^2+2|H|=2\}\times\R_t \to \mathcal{L}^1_+ \big(L^2(0,2\pi)\big)
,
$$
is integrable with respect to $\ell_{\alpha_0}$, continuous in $t$ and takes values in the
set of nonnegative trace-class operators on $L^2(0,2\pi)$.

\noindent Moreover, for $\ell_{\alpha_0}$-almost every $(\omega, E,H)$, we
have
$$
\sigma_{\alpha_0}(\omega, E,H , t) = U_{\alpha_0 , \omega}(t)
\sigma_{\alpha_0}(\omega, E,H , 0)U_{\alpha_0 , \omega}^*(t) .
$$
Finally, $\ell_{\alpha_0}$ and $\sigma_{\alpha_0}(\cdot , 0)$ only depend on the sequence of initial conditions $(u_n^0)$.

\item[(iv)] For $\alpha_0=\pm\frac\pi{2}$,  $\mu_{\alpha_0}$ does
not depend on $t$, it is a measure carried by the set
$\cI_{\pm\frac\pi{2}}\cap\{H=E^2/2\}$ (which consists of vectors tangent
to $\partial\ID$) and is invariant under rotations around the
origin.
\end{itemize}

\end{theorem}

To conclude the description of $\mu_{ml}$, it now remains to describe more precisely $\mu_c$.
 \begin{theorem}\label{t:muc}
There exists a nonnegative trace class operator $\rho_0$ on the
Hilbert space $L^2(\ID)$ such that
\begin{equation}\label{e:defmuc}\la{\mu}_{c
}, a\ra= \int \Tr_{L^2(\ID)} \left\{ U_V(t)^{-1}\mathds{1}_{\ID}\Op_1 (a(x, \xi, t,
H)) \mathds{1}_{\ID} U_V(t)\rho_0 \right\} dt\end{equation}
(the meaning of this
expression is clarified in Section~\ref{s:muc}).

As a consequence, the projection of ${\mu}_{c }$ on $\ID\times
\R_t$ is a nonnegative Radon measure, which is absolutely continuous,
and continuous with respect to $t$.

\end{theorem}
Note that the ambiguity in the meaning of formula \eqref{e:defmuc} arises when $a$ depends
on $H$. If $a$ is independent of $H$, \eqref{e:defmuc} is the
well-defined expression
$$ \la{\mu}_{c
}, a\ra= \int \Tr_{{L^2(\ID)}} \left\{ U_V(t)^{-1} \mathds{1}_{\ID}\Op_1 (a(x, \xi,
t)) \mathds{1}_{\ID} U_V(t)\rho_0 \right\} dt .
$$

\subsection{Link between microlocal and semiclassical Wigner measures\label{s:link}}Let us clarify the link between the two approaches{ in the context of the present article (see also~\cite{GerardMesuresSemi91, GerLeich93} for a related discussion).}

As was said, if $(u^0_n)$ is $h_n$-oscillating from above and
below, the semiclassical Wigner measures convey more information
than the microlocal ones.
 In fact, if $(u_n^0)$ is $h_n$-oscillating from above and below (with $h_n\rightarrow 0$), we have for $a\in \cS_0$
\begin{multline}W_{u_n}(a)=\la u_n, \Op_1\left(a(z, \xi, t, H)\left(\chi\left(\frac{h_n^2(|\xi|^2+|H|)}{R}\right)-\chi\left(\frac{h_n^2(|\xi|^2+|H|)}{\eps}\right)\right)\right)u_n\ra_{L^2(\R^2\times \R)}\\+o(1)_{\eps\To 0, R\To +\infty}\\
=\la u_n, \Op_1\left(a_{{\rm{hom}}}(z, h_n\xi, t,h_n^2
H)\left(\chi\left(\frac{h_n^2(|\xi|^2+|H|)}{R}\right)-\chi\left(\frac{h_n^2(|\xi|^2+|H|)}{\eps}\right)\right)\right)u_n\ra_{L^2(\R^2\times
\R)}\\+o(1)_{\eps\To 0, R\To +\infty}
\\= W^{h_n}_{u_n}\left(a_{{\rm{hom}}}\left(\chi\left(\frac{(|\xi|^2+|H|)}{R}\right)-\chi\left(\frac{(|\xi|^2+|H|)}{\eps}\right)\right)\right)+o(1)_{\eps\To 0, R\To +\infty}\label{e:scml}
\end{multline}

From \eqref{e:scml}, one sees that if $W_{u_n}$ converges weakly to $\mu_{ml}$ and $W^{h_n}_{u_n}$ converges weakly to $\mu_{sc}$, and if $(u^0_n)$ is $h_n$-oscillating from above and below, we have
$$\mu_{ml}(a)=\mu_{sc}(a_{{\rm{hom}}}).$$
The right-hand side is well-defined since $\mu_{sc}$ is a nonnegative
measure which is bounded on sets of the form $\ID\times \R^2\times
[-T, T] \times \R$ (for any $T$).

\bigskip

On the other hand, if in Theorem~\ref{t:precise} the sequence $(u^0_n)$ is not $h_n$-oscillating from
below, then $\mu_{sc}$ does charge the set $\{\xi= 0\}$, and we
have for any compactly supported function $a$ :
 \begin{eqnarray}\label{e:mlsc}
 \mu_{sc}\rceil_{(\xi,H)=0}(a)
 &=& \lim_{\eps \To 0} \lim_{n\To +\infty}
W_{\chi\left(\frac{h_n^2(|D_z|^2+|D_t|)}{\eps}\right)u_n} (a(z, 0, t, 0) )\nonumber \\
&=& \lim_{\eps \To 0} \lim_{n\To +\infty}
W_{\chi\left(\frac{3h_n^2|D_t|}{\eps}\right)u_n} (a(z, 0, t, 0) )\nonumber \\
&=&\lim_{\eps \To 0} \lim_{n\To +\infty}
W_{U_V(t) v^0_{n,\eps}} (a(z, 0, t, 0) )
\end{eqnarray}
where $v^0_{n,\eps} =
\chi\left(\frac{3h_n^2|D_t|}{\eps}\right)u_n\rceil_{t=0}$.
Equality of the first two lines comes from the fact that the
measures asymptotically concentrate on $\{|\xi|^2=2H\}$, and
equality of the last two lines is proven in Appendix~\ref{s:osc}.
We see that the microlocal Wigner measures associated with $U_V(t)
v^0_{n,\eps} $ encompass the description of $\mu_{sc}\rceil_{(\xi,
H)=0}${
: we
have
$$\mu_{sc}\rceil_{(\xi,H)=0}(a(z, \xi, t, H))=\mu_{ml, 0}(a(z, 0, t, 0)), $$
where $\mu_{ml, 0}$ 
possesses the structure described in Theorem
\ref{t:preciseml}.

\bigskip
Finally, if $(u^0_n)$ is not $h_n$-oscillating from above, we see
that
$$\lim_{R \To +\infty} \lim_{n\To +\infty}
W^{h_n}_{(1-\chi)\left(\frac{h_n^2(|D_z|^2+|D_t|)}{R}\right)u_n}(a)=0$$ for
compactly supported $a$, whereas the limit
$$\lim_{R \To +\infty} \lim_{n\To +\infty}
W_{(1-\chi)\left(\frac{h_n^2(|D_z|^2+|D_t|)}{R}\right)u_n}(a)$$ does not
necessarily vanish for homogeneous $a$. This last limit coincides with
$$\lim_{R \To +\infty} \lim_{n\To +\infty}
W_{(1-\chi)\left(\frac{3h_n^2|D_t|}{R}\right)u_n}(a)= \lim_{R \To +\infty} \lim_{n\To +\infty}
W_{U_V(t) w_{n, R}^0}(a)$$
where $w_{n, R}^0=(1-\chi)\left(\frac{3h_n^2|D_t|}{R}\right)u_n\rceil_{t=0}$, and equality of the limits is proven in Appendix~\ref{s:osc}.
Thus, the frequencies of
$u_n^0$ that are of order $\gg h_n^{-1}$ or $\ll  h_n^{-1}$ are
better captured by the microlocal approach.


 \subsection{Application to the regularity of limit measures}
\label{s:appregularity} Theorem \ref{t:preciseml}, applied to test
functions $a\in\cS_0$ that do not depend on $t$ and $H$, implies
Corollary \ref{t:example}. To state a precise version of this
result (say, in the semiclassical setting), we first need the following proposition.

\begin{proposition}\label{p:barmu}
Suppose that $\mu_{sc}$ is a semiclassical measure associated to
$(u_h)$ solution of \eqref{e:S}-\eqref{e:Dirichlet}. Denote by
$\bar{\mu}_{sc}(dE, dJ, t)$ the image of the measure $
{\mu}_{sc}(dz,d\xi,t,dH)$ under the moment map $$M:(z=(x, y),\xi,
H)\mapsto (E, J)=(|\xi|, x \xi_y -y\xi_x)$$ (velocity and angular
momentum). Then $\bar{\mu}_{sc}$ does not depend on $t$.
\end{proposition}
This proposition is proved in Section~\ref{s:proofbarmu}. Arguments developed in \cite{AnantharamanMaciaTore} (and that we
do not reproduce here) show that Corollary~\ref{t:example} can be refined as follows.

\begin{theorem}
\label{t:main}Define by $\mu_{E, J }(t,\cdot)$ is the
disintegration of $\mu_{sc}(t,\cdot)$ with respect to the
variables $(E, J)$, carried on the $2$-dimensional (lagrangian)
manifold $T_{(E, J)}=\{(z, \xi), (|\xi|, x \xi_y -y\xi_x) =(E,
J)\}$, i.e.
\begin{multline*}
\int_{ \R_H}\int_{\ovl{\ID}\times\mathbb{R}^{2}}f(z,\xi, t, H)\mu_{sc}(dz,d\xi, t, dH) \\
= \int_{\mathbb{R}^{2}}\left( \int_{T_{(E, J)}  }f\left(z,\xi, t,
\frac{E^2}2\right)\mu_{E, J }(t, dz, d\xi)\right)
\bar{\mu}_{sc}(dE, dJ) ,
\end{multline*}
for every bounded measurable function $f$, for $t\in\R$.

Then for $\bar{\mu}_{sc}$-almost every $(E, J)$ with $|J|\not= E$, the measure $\mu_{E, J}%
(t,\cdot)$ is absolutely continuous on $T_{(E, J)} $.
\end{theorem}

Note that $|J|=E$, with $E\not =0$, means that $T_{(E,
J)}\cap\left(\ovl{\ID}\times \IR^2\right)$ is contained in the set
$\{(z, \xi), |z|=1, z\perp\xi\}$ of tangent rays to the boundary.
The restriction of $\mu_{sc}(t)$ to that set may be considered
trivial, since \eqref{e:transport} implies that it is invariant
under rotation.

Finally, for $J=E=0$, we can use the last lines of Section~\ref{s:link}, combined with Theorem \ref{t:preciseml}: the
measure $\mu_{sc}$ restricted to $\{\xi=0\}= \ovl{\ID}\times
\{0\}$ is the sum of an absolutely continuous measure carried by
the interior $\ID$ and a multiple of the Lebesgue measure on $\d
\ID$.

{
\begin{remark}
\label{r:barmicro aussi}
The analogues of Proposition~\ref{p:barmu} and Theorem~\ref{t:main} hold as well in the microlocal setting. In particular, if $\bar\mu^\infty$ is the image of $\mu^\infty(t)$ under the map
$(z, \xi, t, H)\mapsto (E, J)$, this measure is independent of $t$. 
\end{remark}
}

\begin{remark}
\label{rem:nonobsnonbord}
Proposition~\ref{p:barmu} (and Remark~\ref{r:barmicro aussi}) allows us to complete the proof of the necessity of the assumption $\Omega \cap \d \ID \neq \emptyset$ in Theorem \ref{t:obs-i} when $V$ does not identically vanish (see the discussion in Section~\ref{s:ointroobsschrod}). Taking for
instance as initial data $u^0_n := \psi_{n,0}^\pm/\|\psi_{n,0}^\pm\|_{L^2(\ID)}$
(see~\eqref{e:psidisk}) with and $n \to \infty$, then one has $
|u^0_n|^2 dx \rightharpoonup (2\pi)^{-1} \delta_{\partial\ID}$ (see Section~\ref{s:eigenfcts});
more precisely, the Wigner measures associated with the initial
data $u^0_n := \psi_{n,0}^\pm/\|\psi_{n,0}^\pm\|_{L^2(\ID)}$ concentrate on the set $\{|J|=E\}$.
Combined with Proposition \ref{p:barmu} (and Remark~\ref{r:barmicro aussi}), this shows
that $\bar\mu^{ml}$ is entirely carried by the set $\{|J|=E\}$,
and thus $\mu^{ml}$ itself does not charge the interior of the
disk, where $|J|<E$. This shows that \eqref{e:oi} cannot hold if
$\Omega$ does not touch the boundary.
\end{remark}

\subsection{Measures at the boundary}
\label{s:mesboun} In this section, we define and compare different
measures on $\d \ID$. Given an invariant measure for the billiard
flow, we first define the associated ``projected measure'' on the
boundary. Second, we define semiclassical and microlocal measures
associated with the Neumann trace at the boundary of sequences of
solutions of~\eqref{e:S}. We finally explain the links between
these three objects.

\subsubsection{Projection on the boundary of an invariant measure}
\label{s:projbou} We observe the following standard construction
from the theory of Poincar\'e sections in dynamical systems. Let
$S=\{(z, \xi), |z|=1, \xi\cdot z \not = 0\}$, union of $S^+=\{(z,
\xi), |z|=1, \xi\cdot z > 0\}$ (vectors pointing outwards) and of
$S^-=\{(z, \xi), |z|=1, \xi\cdot z < 0\}$ (vectors pointing
inwards). When $(z, \xi)\in S^+$, we denote as above by $\alpha(z,
\xi)=-\arcsin\left(\frac{J(z, \xi)}{|\xi|}\right)$ the angle of
the vector $\xi$ with the normal at $z$ to the disk. The map
\begin{eqnarray*} P :\{(z, \xi, \tau) \in S^+\times \R,  \tau\in [0, 2\cos\alpha(z, \xi)] \}
&\To & \ovl{\ID}\times \R^2\\
(z, \xi, \tau)&\mapsto & \left(z+ \frac{\tau}{|\xi|}
\sigma_z(\xi), \sigma_z(\xi)\right)
\end{eqnarray*}
is a measurable bijection onto its image $S \cup ({\ID}\times
\R^2)$, and $\pi\circ P$ is a measurable bijection onto its image
(recall that $\pi$ is the projection from $\ovl{\ID}¬†\times \R^2$ to
$\W$). If $\mu$ is a nonnegative measure on $S \cup
({\ID}\times \R^2)$ which does not charge $S$, and such that
$\pi_* \mu$ is invariant under the billiard flow, then $P^{-1}_*
\mu$ must be of the form
$$P^{-1}_* \mu =\mu^S \otimes d\tau$$
where $\mu^S$ is a measure on $ S^+$ which is invariant under the
first return map
$$
(z, \xi)\mapsto \left(z+ \frac{2\cos\alpha(z, \xi) }{|\xi|}
\sigma_z(\xi), \sigma_z(\xi)\right).
$$
This implies that
\begin{eqnarray}\label{e:sec}\int_{\overline{\ID}\times \R^2} \xi.\partial_z a\,\, d\mu
&=& \int |\xi| \partial_\tau (a\circ P) d\mu^S \otimes d\tau \nonumber \\
& =& \int_{S^+} |\xi |\left(a\left(z+ \frac{2\cos\alpha }{|\xi|} \sigma_z(\xi), \sigma_z(\xi)\right) -a(z , \sigma_z(\xi))\right) \mu^S(dz, d\xi) \nonumber \\
&=& \int_{S^+}|\xi | \left(a(z, \xi) -a(z , \sigma_z(\xi))\right)
\mu^S(dz, d\xi).
\end{eqnarray}

Note that the total mass of $\mu$ is $\int d\mu =
\int_{S^+}2\cos\alpha(z, \xi) d\mu^S(z, \xi).$

\subsubsection{Semiclassical measure associated to Neumann trace\label{s:trace}}
Let $(u_h^0)$ be a family of initial conditions, normalized in
$L^2(\ID)$. When we look at the semiclassical Wigner distributions
\eqref{e:Wsc}, where we use {\em compactly supported} symbols,
Remarks~\ref{r:H1} and \ref{r:osc} show that we may truncate
$(u_h^0)$ in frequency and assume, without changing the limit as
$h\To 0$, that $u_h^0\in H_0^1(\ID)$, $\norm{\nabla
u_h^0}_{L^2(\ID)}=O(h^{-1})$. Proposition \ref{p:regboundary} then
entails that the boundary data $h\partial_{n}\left(
U_V(t){u}_{h}^0\right) $ form a bounded sequence in
$L_{\loc}^{2}\left(  \mathbb{R}\times
\partial\ID\right)  $.

Now, let $\mu_{sc}^\d \in \cM_+(T^* \partial \ID\times T^*\R)$ be
a semiclassical measure associated with the boundary data
$h\partial_n u_h(t)$ {defined by quantizing test functions on $T^* \partial \ID\times T^*\R$ with the same scaling $(hj,h^2 H)$ in the cotangent variables as in~\eqref{e:defW}}. Then $\mu_{sc}^\d$ is carried by the set
$\{(u, j, t, H)\in T^* \partial \ID\times  T^*\R, |j|\leq
\sqrt{2H}\}$. If $\mu_{sc}$ and $\mu_{sc}^\d$ are obtained through
the same sequence of initial data, then we have the relation
(see~\cite{GerLeich93})
\begin{equation}\label{e:section}\int_{\ovl{\ID}\times \R^2\times\R\times \R}\xi\cdot\partial_z a\,\,
\mu_{sc}(dz,d\xi, t, dH)dt=\int \frac{a(u, \xi^+(j, H)) -a(u,
\xi^-(j, H)) }{2\sqrt{2H-j^2}} \mu_{sc}^\d (du, dj, dt,
dH),\end{equation} valid for any smooth function $a$. For $(u,
j)\in T^*\partial \ID$ with $|j|\leq \sqrt{2H}$, the vectors
$\xi^{\pm}(j, H)$ are the two vectors (pointing outwards and
inwards) in $T^*_u\R^2$ of norm $=\sqrt{2H}$, whose projection to
$T^*_u\partial \ID$ is $j$. Note that the expression under the
integral on the right hand side of \eqref{e:section} has a
well-defined finite limit as $|j|\To \sqrt{2H}$. Identity
\eqref{e:section} has three consequences:
\begin{itemize}
\item First, the measure
$\mu_{sc}$ does not charge the set $S$ defined in Section~\ref{s:billiard} (otherwise the left-hand side of
\eqref{e:section} would define a distribution of order $1$ which
is not a measure). Note that \eqref{e:section} is stronger
than~\eqref{e:transport}.
\item Second, let $\mu^S_{sc}(t)$ be the
measure associated to $\mu_{sc}(t)$ as in Section~\ref{s:projbou}.
Comparing~\eqref{e:section} with \eqref{e:sec}, we see that for any $a$ defined on
$S^+$,
\begin{multline*}
\int_{(u, j)\in T^*\partial \ID, |j|<\sqrt{2H}} a(u, \xi^+(j, H), t, H)  \mu_{sc}^\d (du, dj, dt, dH)\\
=\int_{S^+} 2|\xi|^2 \cos\alpha(z, \xi) a(z, \xi) \mu^S_{sc}(dz,
d\xi, t, dH)dt.
\end{multline*}
\item Third, \eqref{e:section} implies
\begin{equation*} \int_{T^*\d\ID \times\R \times\R} |\xi|^2 a(z, \xi, t, H)
\mu_{sc}(dz,d\xi, t, dH) dt=\int_{|j|=\sqrt{2H}}  a(u, j, t,
H)\mu_{sc}^\d (du, dj, dt, dH).
\end{equation*}
In particular, note that $\mu_{sc}^\d\rceil_{H=0}$ vanishes, since
$H=0$ corresponds to $\xi=0$ on the left-hand side.
\end{itemize}

Identities \eqref{e:transport} and \eqref{e:section} are
essentially proven in \cite{GerLeich93} (see also~\cite{RZ:09}) for general domains (for
time-independent solutions of \eqref{e:S}); we do not reproduce
the proofs here.

\subsubsection{Microlocal measure associated to Neumann trace}
\label{s:micromesboundary} The sequences considered here $u_n =
U_V(t)u_n^0$ are bounded in $L^\infty(\IR; L^2(\ID))$. Since
normal traces are not convenient to work with at this level of
regularity, the definition of associated microlocal measures needs
a little care.

For this, let us first define $\psi \in C^\infty(\R)$, such that
$\psi = 0$ on $(-\infty ,1]$ and $\psi = 1$ on $[2,+\infty)$ and
the operator $A(D_t)=
\Op_1\left(\frac{\psi(H)}{\sqrt{2H}}\right)$. We have the
following regularity result.
\begin{lemma}
\label{l:regutemps} For all $\varphi \in C^\infty(\R_t \times
\ovl{\ID}_z)$ with compact support in the first variable $t \in
\R$, there exists a constant $C=C(\varphi, \psi)>0$ such that for
any $u^0 \in L^2(\ID)$, the associated solution $u(t) = U_V(t)u^0$
satisfies
\begin{equation*}
\|A(D_t)\varphi u\|_{L^2(\R ; H^1(\ID))} \leq C \|u^0\|_{L^2(\ID)}.
\end{equation*}
\end{lemma}
This Lemma is proved at the end of Appendix~\ref{s:osc}.
We now define, for any $g\in C_c^\infty(\R)$, the
sequence $\tilde{u}_n = A(D_t) g(t) u_n$, solution of
\begin{equation}
\label{eqtildeu} \left( D_t + \frac{1}{2}\Delta \right)
\tilde{u}_n =     A(D_t) \left(g(t) V(t,z)  + ig'(t) \right) u_n .
\end{equation}
As a consequence of Lemma~\ref{l:regutemps}, we have
$\|\tilde{u}_n\|_{L^2(\R ; H^1(\ID))} \leq C \|u_n^0\|_{L^2(\ID)}$
together with $$\|A(D_t) \left(g(t) V(t,z) u_n + ig'(t) \right)
u_n \|_{L^2(\R ; H^1(\ID))} \leq C \|u_n^0\|_{L^2(\ID)}.$$
Equation~\eqref{eqtildeu} then implies that
$\|\tilde{u}_n\|_{L^\infty (\R ; H^1(\ID))} \leq C
\|u_n^0\|_{L^2(\ID)}$ and that $A(D_t) g(t) \d_n u_n = \d_n
\tilde{u}_n$ is bounded in $L^2(\IR \times \d \ID)$ by
$\|u^0_n\|_{L^2(\ID)}$, according to the hidden regularity of
Proposition~\ref{p:regboundary}. Hence, if we take $g$ to be
constant equal to $1$ on the support of $a$, after extraction of
subsequences, the following limit exists
$$
\la \mu_{ml}^\d , a\ra = \lim_{R\rightarrow\infty}
\lim_{n\rightarrow +\infty}
 \left\la \partial_n \tilde{u}_n,  \Op_1\left(  \left(  1-\chi\left(  \frac{|H|}{R^2 }\right)\right) a(u,j, t, H )\right) \partial_n \tilde{u}_n \right\ra_{L^2(\d \ID\times \IR)}  ,
$$
for symbols  $a\in C^\infty (T^*(\d \ID \times \R))$, compactly
supported in the variables $z, t$, such that
\begin{equation*}
a(u, j, t, H)   =a(u , \lambda j,  t, \lambda^2 H) \text{,\quad for } | H| >R_{0} \text{ and } \lambda\geq 1.%
\end{equation*}
Then $\mu_{ml}^\d$ is carried by the set $\{(u, j, t,
H)\in T^* \partial \ID\times  T^*\R, |j|\leq \sqrt{2H}\}$. If
moreover $\mu_{ml}$ and $\mu_{ml}^\d$ are obtained through the
same sequence of initial data, then we have the relation (see
again~\cite{GerLeich93})
\begin{multline}
\label{e:sectionml} \int_{\ovl{\ID}\times \R_t \times
\IS^2_{\xi,H}}\frac{\xi}{\sqrt{2H}}\cdot\partial_z a\,\,
\mu^\infty (dz,d\xi, t, dH)dt   \\
=\int_{ |j|\leq \sqrt{2H}} \frac{1}{2} \left(\frac{2H}{2H-j^2}
\right)^\frac12 \left( a(u, \xi^+(j, H)) -a(u, \xi^-(j, H))
\right)\mu_{ml}^\d (du, dj, dt, dH)
\end{multline}
valid for any $a \in \mathcal{S}_0$. The vectors $\xi^{\pm}(j, H)$ are the two vectors (pointing outwards and inwards) in $T^*_u\R^2$ of norm $=\sqrt{2H}$, whose projection to $T^*_u\partial \ID$ is $j$. As above, this implies that $\mu^\infty$ does not charge the set $S$; we then denote by $\mu^S_{ml}(t)$ the measure associated to $\mu^\infty(t)$ as in Section~\ref{s:projbou}. Comparing with \eqref{e:sec}, we see that for any $a \in \mathcal{S}_0$, we have 
\begin{multline}
\label{e:mudmuS}
\int_{(u, j)\in T^*\partial \ID, |j|<\sqrt{2H}} a(u, \xi^+(j, H), t, H)  \mu_{ml}^\d (du, dj, dt, dH)\\
=\int_{S^+}2 \cos\alpha(z, \xi) a(z, \xi) \mu_{ml}^S (dz, d\xi, t,
dH)dt.
\end{multline}
Moreover, \eqref{e:sectionml} implies
\begin{multline}\label{e:mudmuint}
 \int_{(u, \xi)\in T^*\partial \ID, |\xi|=\sqrt{2H}} 
a(z, \xi, t, H)\mu^\infty (dz,d\xi, t, dH) dt
\\
=\int_{|j|=\sqrt{2H}}  a(u, j, t, H)\mu_{ml}^\d (du, dj, dt, dH) .
\end{multline}

These links between the different measures shall be in particular
useful when proving the boundary observability result of
Theorem~\ref{t:obs-b}.

\bigskip

\subsection{Plan of the proofs} 
Section~\ref{s:coorddecomp} first deals with the understanding of action-angle coordinates and the appropriate decomposition of measures that are invariant by the billiard flow.
Section \ref{s:coord} discusses in more detail the coordinates described in the
introduction, in which the dynamics of the billiard can be
integrated and introduces the Fourier Integral Operator
corresponding to this change of coordinates.
Section \ref{s:decompo} reduces the study of invariant measures
on the disk to their restriction to all invariant tori of the dynamics
(more precisely, their restriction to the level sets $\cI_\alpha$, which are unions of invariant tori)

Sections \ref{s:second} and \ref{s:pause} are devoted to the proof of
Theorem~\ref{t:precise} (semiclassical version of the result). In Section~\ref{s:second}, we perform the second microlocalization on a level set $\cI_\alpha$: we start by introducing the adapted class of symbols in Section~\ref{s:symbols} and the appropriate coordinates (which are a modification of the action-angle coordinates) in Section~\ref{s:coordIla0}. This allows us to construct the two different second-microlocal measures in Section~\ref{s:secondmicro}. We then prove their structure properties in Sections~\ref{s:strucprop} and~\ref{s:sy}. To complete the analysis, we prove that they obey invariance laws in Section~\ref{s:strucprop} and~\ref{s:propagation} respectively. 
 Section \ref{s:pause} then concludes the proof of Theorem~\ref{t:precise}.

In Section \ref{s:H0} we explain how to adapt the proof to obtain the
microlocal version, Theorem~\ref{t:preciseml}.

The observability inequalities of Theorems~\ref{t:obs-i} and~\ref{t:obs-b} are then derived in Section~\ref{s:obs}.

Appendices \ref{coord-polar} and \ref{s:commutator} are devoted to
the technical calculations needed to change coordinates from polar to action-angle ones.
Appendix \ref{s:osc} is a technical elaboration on the ``hidden
regularities'' of solutions of Schr\"odinger equations.
Finally, Appendix~\ref{s:timeregwign} states and proves the $L^\infty$ regularity in time of Wigner measures associated to solutions of the Schr\"odinger equation.

\section{Action-angle coordinates and decomposition of invariant measures\label{s:coorddecomp}}

\subsection{Action-angle coordinates and their quantization\label{s:coord}}

Recall that the change of coordinates $\Phi$, mapping action-angle coordinates to cartesian ones, is introduced in Section~\ref{s:billiard} (see~\eqref{e:def Phi}).
%
The map
\[
\Phi:\{(s,\theta,E,J)\;:\;E>0,\theta\in\R/2\pi\Z,s\in\R,J\in\R\}\longrightarrow
\{(z,\xi)\in\IR^{2}\times\IR^{2}\;:\;\xi\neq0\}
\]
is a diffeomorphism satisfying, in particular,
\[
\Phi^{-1}\big(\ID\times(\R^{2}\setminus\{0\})\big)=\left\{
(s,\theta ,E,J)\;:\;(\theta,E)\in\R/2\pi\Z\times
(0,\infty),(J/E)^{2}+s^{2}<1\right\} .
\]
Note that the hamiltonian flow of the energy $\frac{E^{2}}{2}$
(the geodesic flow) reads
\[
G^{\tau}:(s,\theta,E,J)\mapsto(s+\tau
E,\theta,J,E),\qquad\tau\in\R,
\]
and the hamiltonian flow of $J$ (unit speed rotation) is given by:
\[
R^{\tau}:(s,\theta,E,J)\mapsto(s,\theta+\tau,J,E),\qquad\tau\in\R.
\]
Write for $\theta \in \R/2\pi\Z$,
\[
\omega\left(  \theta\right)  :=\left(
-\sin\theta,\cos\theta\right)  ;
\]
the transformation $\Phi$ admits the generating function
\[
S(z,\theta,s,E)=E\omega\left(  \theta\right)  \cdot z-Es,
\]
meaning that
\begin{align*}
\mathrm{Graph}\,\Phi & =\left\{  (s,\theta,E,J,z,\xi)\;:\;(z,\xi
)=\Phi(s,\theta,E,J)\right\}  \\
& =\left\{  (s,\theta,E,J,z,\xi)\;:\;\frac{\partial S}{\partial
E}=0,\xi
=\frac{\partial S}{\partial z},J=-\frac{\partial S}{\partial\theta}%
,E=-\frac{\partial S}{\partial s}\right\}.
\end{align*}
The existence of such a generating function implies that the
diffeomorphism $\Phi$ preserves the symplectic form (see for instance~\cite[Theorem~2.7]{Zwobook}), i.e.
\[
d\xi_{x}\wedge dx+d\xi_{y}\wedge dy=dE\wedge ds+dJ\wedge d\theta.
\]

Using this generating function we define a unitary operator that
quantises the canonical transformation $\Phi$. The operator
\begin{equation}
\mathscr Uf(s,\theta)=(2\pi h)^{-3/2}\int_{0}^{\infty}\int_{\mathbb{R}^{2}%
}e^{-i\frac{S(z,\theta,s,E)}{h}}f(z)\sqrt{E}dzdE,\label{e:U}%
\end{equation}
mapping functions on $\R_{z}^{2}$ to functions on
$\R_{s}\times\R_{\theta }$ is in fact a semiclassical Fourier
Integral Operator associated with $\Phi
$ (the choice of the term $\sqrt{E}$ in this expression is explain by Lemma~\ref{p:FIO} below). Note that $\mathscr Uf$ can be also written independently of $h$ as:%
\[
\mathscr Uf(s,\theta)=\int_{0}^{\infty}e^{iEs}\widehat{f}\left(
E\omega\left(  \theta\right)  \right)  \sqrt{E}\frac{dE}{\left(
2\pi\right) ^{3/2}},
\]
where $\widehat{f}$ stands for the Fourier transform of $f$.
Therefore, the
Fourier transform with respect to $s$ of $\mathscr Uf(s,\theta)$ is merely:%
\[
\left(  2\pi\right)  ^{-1/2}\widehat{f}\left(  E\omega\left(
\theta\right) \right)   \mathds{1}_{\left[  0,\infty\right)
}\left(  E\right)  \sqrt{E}.
\]
From this it is clear that for any symbol
$\phi:\mathbb{R\rightarrow R}$ one
has:%
\[
\phi\left(  D_{s}\right)  \mathscr Uf=\mathscr U\phi\left(
\left\vert D_{z}\right\vert \right)  f,
\]
and, by Placherel's theorem,%
\begin{align*}
\int_0^{2\pi}\int_0^{\infty}\mathscr
Uf(s,\theta)\overline{\mathscr
Ug\left(  s,\theta\right)  }dsd\theta & =\int_0^{2\pi}\int_{0}%
^{\infty}\widehat{f}\left(  E\omega\left(  \theta\right)  \right)
\overline{\widehat{g}\left(  E\omega\left(  \theta\right)  \right)  }%
E\frac{dsd\theta}{\left(  2\pi\right)  ^{2}}\\
& =\left\langle g,f\right\rangle _{L^{2}\left(
\mathbb{R}^{2}\right)  }.
\end{align*}
In particular, the following Lemma has been proved:

\begin{lemma}
\label{p:FIO} (i) The operator $\mathscr U$ is unitary from $L^{2}%
(\mathbb{R}^{2})$ to $L^{2}\left(  \R\times\R/2\pi\Z\right) $:
$\mathscr U^{\ast}\mathscr U=I$ .

(ii) For $f\in C_{c}^{\infty}(\R^{2})$, we have $\partial^{2}_{s}
\mathscr U f=\mathscr U \Delta f$.

As a consequence,
\[
-h^{2}\mathscr U\Delta\mathscr U^{\ast}=-h^{2}\partial _{s}^{2}.
\]
\end{lemma}

\noindent\textbf{Notation}. We denote by $P_{0}(z,
\xi)=\frac{|\xi|^{2}}{2}$ the hamiltonian generating the geodesic
flow in $\R^{2}\times\R^{2}$; and $P_{1}(z, \xi)= x\xi_{y} - y
\xi_{x}$ the hamiltonian generating the (unit speed) rotation. We
denote by $X_{P_{0}} = \xi\cdot\partial_{z}$ and $X_{P_{1}} =
z^{\perp}\cdot\partial_{z} + \xi^{\perp}\cdot\partial_{\xi}$ the
corresponding hamiltonian vector fields on $T^{*}\IR^{2}$. We
denote by $G^{\tau}(z, \xi)=(z+\tau\xi, \xi)$ the geodesic flow
(generated by $P_{0}$) and $R^{\tau}$ the flow generated by
$P_{1}$ (rotation of angle $\tau$ of both $z$ and $\xi$).

\medskip
In the new coordinates, these hamiltonians and vector fields are
slightly simpler since $P_{0}\circ\Phi=\frac{E^{2}}{2}$,
$P_{1}\circ\Phi= J$ together with $X_{P_{0} \circ\Phi} = E
\partial_{s}$ and $X_{P_{1}\circ\Phi} =
\partial_{\theta}$. Very often, we shall (with a slight abuse of notation) use
the letter $J$ to mean the function $P_{1}$, and $E$ for the
function $\sqrt{2P_{0}}$.

\subsection{Decomposition of an invariant measure of the billiard \label{s:decompo}}
This section aims at describing properties shared by all measures
$\mu$ invariant by the billiard flow (even if they are not
necessarily linked with solutions to a partial differential
equation). It essentially collects a few simple facts that will be
useful in the next sections when studying measures arising from
solutions of the Schr\"odinger equation~\eqref{e:S}.

\medskip
Let $(z, \xi)\in \ovl{\ID} \times \R^2$. There exist $t_1\leq 0,
t_2\geq 0$ such that $|z+t_1\xi|= |z+t_2\xi|=1$. Note that if $(z,
\xi)\in \ID \times \R^2$, then $t_1$ and $t_2$ are unique and
$t_1>0$, $t_2<0$.

Recall that $\alpha\in [-\pi/2, \pi/2]$ (defined in Section~\ref{s:billiard}) is the oriented
angle between $-(z+t_1\xi)$ and $\xi$ (that is, the angle between
the velocity $\xi$ and the inner normal to the disk, at the point
where the oriented straight line $\{z+t\xi, t\in \R\}$ first hits
the disk). See Figure~\ref{fig:angle-disk}. One has the expression
$$\alpha = -\arcsin \left(\frac{x\xi_y - y \xi_x}{|\xi|}\right).$$

Our work is based on the following partition of phase space:
$$
\ovl{\ID}\times (\R^2\setminus\{0\})  =\alpha^{-1}\left(\pi\IQ\cap[-\pi/2, \pi/2]\right) \sqcup \alpha^{-1}\left(\R\setminus\pi\IQ\right),
$$
from which the following lemma follows.
\begin{lemma}
\label{Lemma decomposition}Let $\mu$ be any finite, nonnegative Radon
measure\footnote{We denote by $\mathcal{M}_{+}\left( \ovl{\ID}
\times \R^2 \right)$ the set of all such measures.} on
$\ovl{\ID}\times \R^2$.
Then $\mu$ decomposes as a sum of nonnegative measures:%
\begin{equation}
\mu=\mu\rceil_{ \alpha\not\in \pi\IQ}+\sum_{r\in \IQ\cap[-1/2,
1/2]}\mu\rceil_{ \alpha=r\pi}+\mu\rceil_{\xi=0} .
\label{dec}%
\end{equation}
\end{lemma}

 Note that the functions $P_0$, $P_1$, and thus also $\alpha$, are well-defined on the billiard phase space $\W$. Thus the previous lemma applies as well to measures on $\W$.

In what follows, we shall call {\em nonnegative invariant measure} a
nonnegative Radon measure on $\W$ which is invariant under the
billiard flow. We shall extend this terminology to measures $\mu$
defined a priori on $\ovl{\ID}\times \R^2$, to mean that $\pi_*
\mu$ (the image of $\mu$ under the projection $\pi$) is invariant under the
billiard flow $\phi^\tau$ on $\W$.
\begin{lemma}
\label{LemmaFourierInvariant}Let $\mu$ be a nonnegative invariant
measure on $\W$. Then every term in the decomposition (\ref{dec})
is a nonnegative invariant measure, and $\mu\rceil_{ \alpha\not\in
\pi\IQ}$ is invariant under the rotation flow $(R^\tau)$, as well
as $\mu\rceil_{ \alpha=\pm \pi/2}$.

\end{lemma}
The rotation flow $(R^\tau)$ is well defined on $\W$, so the last
sentence makes sense. The assertion for $\alpha=\pm \pi/2$ comes
from the fact that the rotation flow coincides with the billiard
flow (up to time change) on the set $\{\alpha=\pm \pi/2\}$. The
assertion for $\alpha\not\in \pi\IQ$ is a standard fact. It comes
from the remark that, for any given value $\alpha_0$ (such that
$\alpha_0\not\in\pi\IQ$) we can find $T=T(\alpha_0)>0$ such that
$\phi^{T}$ coincides with an irrational rotation on the set
$\{\alpha=\alpha_0\}$.

Thus, for $\alpha\not\in \pi\IQ$ or $\alpha=\pm \pi/2$, there is
nothing to prove to get Theorem \ref{t:main}.

\bigskip
Now consider a term $\mu\rceil_{  \alpha=r_0\pi}$, where $r_0\in
\IQ\cap (-1/2, 1/2)$ is fixed. Let us denote $\alpha_0=\pi r_0$.
Introduce the vector field on $T^*\R^2$:
 $$ (\alpha_0-\alpha)X_{P_1}+\frac{\cos\alpha}E X_{P_0} $$
On the set $\cI_{\alpha_0}=\{J=-\sin\alpha_0 E\}$ it coincides
with $X_{P_0}$ up to a constant factor. Denote by
$\phi^\tau_{\alpha_0}$ the flow on $\W$ generated by $
(\alpha_0-\alpha)X_{P_1}+\frac{\cos\alpha}E X_{P_0}$ with
reflection on the boundary of the disk. More precisely, it is the
unique continuous flow defined on $\W$ such that
$$\phi^\tau_{\alpha_0}(z, \xi)= R^{(\alpha_0-\alpha)\tau} \left(z+\tau\frac{\cos\alpha}{|\xi|} \xi, \xi\right)$$
whenever $z\in \ID$ and $z+\tau\frac{\cos\alpha}{|\xi|} \xi\in\ID$
(with $\alpha=-\arcsin \frac{P_1(z, \xi)}{|\xi|}$ as previously).

In the coordinates $(s, \theta, E,J)$, this flow simply reads
$$\phi^\tau_{\alpha_0}\circ \Phi (s, \theta, E,J)
= \Phi (s+  \tau\cos\alpha , \theta +(\alpha_0-\alpha)\tau, E,J),
\quad \alpha = -\arcsin(J/E),
$$
with reflection on the boundary of the disk.

All its orbits are periodic: actually, we determined the
coefficients  $(\alpha_0-\alpha)$ and $\frac{\cos\alpha}E $
precisely for that purpose, see
Figure~\ref{fig:phi-alpha-0-calcul}. Some trajectories of the flow
are represented on Figure~\ref{fig:phi-alpha-0-trajectories}.

\begin{figure}[h!]
  \begin{center}
    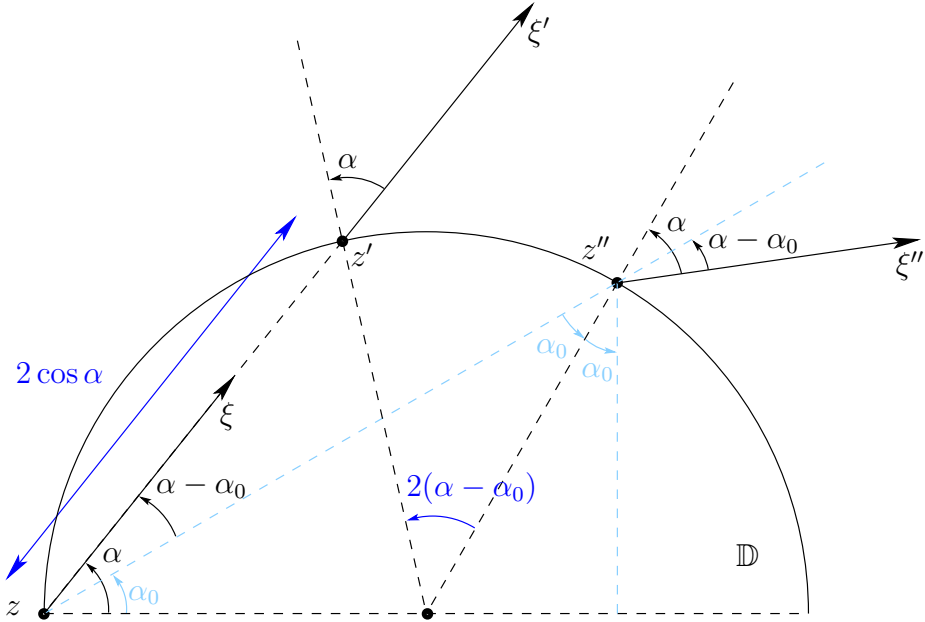
    \caption{Construction of the flow $\phi^\tau_{\alpha_0}$ with $\alpha_0 = \pi/6$. }
    \label{fig:phi-alpha-0-calcul}
    On the figure,  $(z',\xi') = (z+2 \frac{\cos \alpha}{|\xi|}\xi , \xi)$ and $(z'' , \xi'') = R^{2(\alpha_0 - \alpha)}(z',\xi') = \phi^2_{\alpha_0}(z,\xi)$.
  \end{center}
\end{figure}

\begin{figure}[h!]
  \begin{center}
    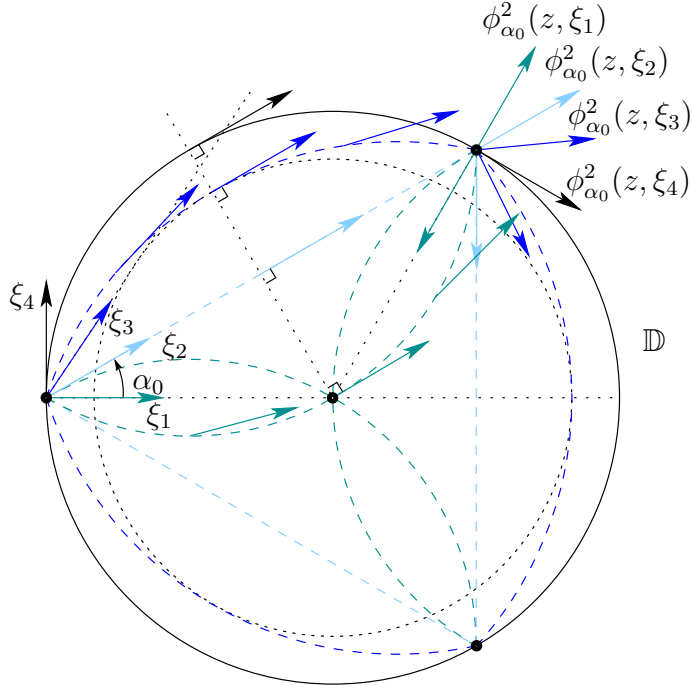
    \caption{Approximate representation of some trajectories of the flow $\phi^\tau_{\alpha_0}$ with $\alpha_0 = \pi/6$ issued from $(z, \xi_j)$ with $z = (-1,0)$ and $\xi_j, j \in \{1,2,3,4\}$ such that $\alpha(z , \xi_1) = 0$,  $\alpha(z , \xi_2) = \alpha_0$, $\alpha(z , \xi_3) \in (\alpha_0 , \pi/2)$ and $\alpha(z , \xi_4) = \pi/2$.}
    \label{fig:phi-alpha-0-trajectories}
  \end{center}
\end{figure}

 The following lemma is now obvious.
\begin{lemma} \label{l:modes}
Let $\mu$ be a nonnegative invariant measure on $\W$.

Let $a\in C^0(\W)$. Then
$$\int  a\circ \phi^\tau_{\alpha_0}d\mu\rceil_{  \alpha=\alpha_0}=\int  ad\mu\rceil_{ \alpha=\alpha_0}$$
for every $t\in \IR$.

Equivalently, we have
$$\int  ad\mu\rceil_{  \alpha=\alpha_0}= \int  \la a\ra_{\alpha_0}d\mu\rceil_{  \alpha=\alpha_0}$$
where $$\la a\ra_{\alpha_0}=\lim_{T\To +\infty} \frac1T \int_0^T
a\circ  \phi^\tau_{\alpha_0} dt.$$
\end{lemma}
Remark that $\la a\ra_{\alpha_0}$ is well defined even if $a$ is a
bounded measurable function on $\ID$ (the times $\tau$ where the
trajectories of $\phi^\tau_{\alpha_0}$ hit the boundary form a set
of measure $0$ in $\IR$).

\section{Second microlocalization on a rational angle}
\label{s:second}This section and the next one are devoted to
proving the semiclassical version of our result, Theorem
\ref{t:precise}.

Let $\left(  u^0_{h}\right)$ be a bounded family in $L^{2}\left(
\ID\right)$. Denote by $u_h(z, t)=U_V(t) u_h^0(z)$ the corresponding solutions to \eqref{e:S}. After
extracting a subsequence, we suppose that its Wigner distributions
$W_h  $ (defined by~\eqref{e:Wsc}) converge to a
semiclassical measure $\mu_{sc}$
in the weak-$\ast$ topology of $\mathcal{D}^{\prime}\left( \R^2
\times \R^2\times\mathbb{R}_t\times \mathbb{R}_H\right)  $. The
measure $\mu_{sc} \in L^{\infty}\left(  \mathbb{R}_t;\mathcal{M}%
_{+}\left(  \R^2 \times \R^2\times \mathbb{R}_H\right)  \right)$ is for a.e. $t\in \R$ supported by $\ovl{\ID} \times
\R^2\times \mathbb{R}_H \cap \{H=\frac{E^2}{2}\}$.

\medskip
From now on, we skip the index $sc$ since there is no possible confusion here (only semiclassical measures are considered until Section~\ref{s:H0}) to lighten the notation.

The aim of this section is to understand the term $\mu\rceil_{
\cI_{\alpha_0}}$, where $\cI_{\alpha_0}=\{\alpha=\alpha_0\}$ and
$\alpha_0\in \pi\IQ$. In view of Lemma \ref{l:modes}, it suffices
to characterize the action of $\mu\rceil_{ \cI_{\alpha_0}}$ on
test functions that are $(\phi_{\alpha_0}^\tau)$-invariant.

 \subsection{Classes of test functions\label{s:symbols}}
Here is a list of properties that we may want to impose on our
symbols in the course of our proof. We express these properties
both in the ``old'' coordinates $(x, y, \xi_x, \xi_y)$ and in the
``new'' ones $(s, \theta, E, J)$.

Let $a$ be a smooth function of $(x, y, \xi_x, \xi_y, t, H)$,
supported away from $\{\xi=0\}$. Then $a\circ\Phi$ is a smooth
function of $(s, \theta, E, J, t, H)$ supported away from
$\{E=0\}$.
The properties we shall use are the following:

\begin{enumerate}
\item[(A)] The symbol $a$ is compactly supported w.r.t. $\xi$, $t$
and $H$. This is equivalent to $a\circ\Phi$ being compactly
supported w.r.t. $E, J$, $t$ and $H$. Note also that $a\circ\Phi$
is $2\pi$-periodic w.r.t. $\theta$.

\item[(B)] For $|z|=1$, we have $a\left(  z,\xi\right)
=a\circ\sigma\left(  z,\xi\right)  $ where $\sigma$ is the
orthogonal symmetry with respect to the boundary of the disk at
$z$. In the coordinates of Section \ref{s:coord}, this reads
(forgetting to write the $(t,H)$-dependence of $a$)
$$a\circ\Phi \left(\cos\alpha, \theta, E, J\right)=
a\circ\Phi \left(-\cos\alpha,\theta + \pi+2\alpha , E, J\right)$$
for all $\theta, E, J$ and for
$\alpha=-\arcsin\left(\frac{J}{E}\right)$.\smallskip

{\noindent\bf Terminology.} We shall say that $a$ is a smooth
function on $\W$ if $a$ is a smooth functon on $\ovl{\ID}\times
\R^2$ that satisfies (B).

\item[(C)] If $a$ satisfies (B), we want in addition that
$a\circ\pi\circ\phi^\tau$ defines a smooth function on $\W$ for
all $\tau$. This is equivalent to requiring that
$$\partial_s^k(a\circ\Phi)\left(\cos\alpha, \theta, E, J\right)=
\partial_s^k(a\circ\Phi) \left(-\cos\alpha,\theta + \pi+2\alpha , E, J\right)$$
for all $k$, for all $\theta, E, J$ and for
$\alpha=-\arcsin\left(\frac{J}{E}\right)$. In other words, all the
derivatives of $a\circ\Phi$ w.r.t. $s$ satisfy the symmetry
condition (B).

\item[(D)] The function $a$ satisfies (C), and in addition $a$ is
$\phi_{\alpha_0}^\tau$-invariant, which reads
$$\left((\alpha_0-\alpha)X_{P_1}+\frac{\cos\alpha}E X_{P_0}\right)
a =0,$$ or, in the new coordinates,
$$[(\alpha_0-\alpha)\partial_\theta+\cos\alpha\partial_s]a\circ\Phi=0.$$

\end{enumerate}

Furthermore, to fix ideas, let us assume that the support of $a$
with respect to $t$ is contained in $(-1, 1)$. This implies that
\begin{eqnarray}\label{e:localt}
  W_h(a) &=&\la g (t) u_h, \Op_1(a(z, h\xi, t, h^2 H) g (t) u_h\ra_{L^2(\R^2\times \R)} +O(h^\infty)\\
  &=& \la g (t) u_h, \Op_h(a(z, \xi, t, h H) g (t) u_h\ra_{L^2(\R^2\times \R)} +O(h^\infty)\nonumber ,
 \end{eqnarray}
 for any smooth cut-off function $g$ supported in $(-2, 2)$ and taking the value $1$ on $(-1, 1)$. In other words, we need only consider the restriction of $u_h(z, t)$ to $t\in (-2, 2)$.

\subsection{Coordinates adapted to the second microlocalization on $\cI_{\alpha_0}$\label{s:coordIla0}}

 We wish to study the concentration of $W_h$ around the set $\{J=-E\sin\alpha_0\}$. If the limit measure $(\Phi^{-1})_* \mu$ is supported on the set $\{E = \sqrt{2H}\}$ this is equivalent to studying the concentration of $W_h$ around $\{J=-\sqrt{2H}\sin\alpha_0\}$. Since this assumption is satisfied for sequences $u_h$ satisfying~\eqref{e:Ssc}, we shall study the concentration of $W_h$ around this set.

 Thus we make the (symplectic) change of variables
 $$\left(s, \theta, E, J'-\sqrt{2H}\sin\alpha_0, t'+\frac{\theta\sin\alpha_0}{\sqrt{2H}}, H\right)= (s, \theta, E, J, t, H)$$
 which sends $\{J'=0\}$ to $\{J=-\sqrt{2H}\sin\alpha_0\}$ and leaves untouched the variables $(s, E)$.

 Consider the following corresponding Fourier Integral Operator (which leaves untouched the variables $(s, E)$, omitted here from the notation):
 $$\scrV  f(\theta, H)= (2\pi )^{-1/2}e^{i\sqrt{2H}\sin\alpha_0\theta /h} \int f(\theta, ht)e^{-iHt/h}   dt.$$

 \begin{lemma} If $b\in C_c^\infty(\R_\theta\times \R_J\times \R_t\times\R_H)$, we have
 \begin{equation}
 \label{conj V}
 \scrV \Op_1(b(\theta, hJ, t, h^2 H))\scrV^*=\Op_h(\tilde b(\theta, J', H, ht))+O(h),
 \end{equation}
 where
  $\tilde b(\theta, J', H, ht)=b\left(\theta, J'-\sqrt{2H}\sin\alpha_0,  -ht , H\right),$
 and
 \begin{equation}
 \label{VV*=I}
 \scrV\scrV^*=I .
 \end{equation}

 \end{lemma}
 \begin{proof}
 First notice that we have
 $$
 \scrV^* g (\theta, t)= (2\pi )^{-1/2} h^{-1} \int g(\theta, H)e^{iHt/h^2} e^{-i\sqrt{2H}\sin\alpha_0\theta /h}dH.
 $$
 Second, we may now compute $A := \scrV \Op_1(b(\theta, hJ, t, h^2 H))\scrV^* e^{it_0 H/h}e^{iJ_0\theta/h}\rceil_{H=H_0, \theta=\theta_0}$. We have the exact formula
 \begin{multline*}
A
 = {(2\pi h)^{-1}} e^{i\sqrt{2H_0}\sin\alpha_0\theta_0 /h}\int  b(\theta_0, J_0-\sqrt{2H}\sin\alpha_0,  ht, H) \\
e^{iHt/h} e^{iHt_0/h}e^{iJ_0\theta_0/h}
e^{-i\sqrt{2H}\sin\alpha_0\theta_0 /h} e^{-iH_0t/h}dH dt .
 \end{multline*}
Note that this expression is exact and thus does not involve the
derivatives of $b$ w.r.t. $\theta$ or $J$. Taking $b = 1$ in this
expression gives the exact formula $A = e^{it_0
H_0/h}e^{iJ_0\theta_0/h}$, which proves \eqref{VV*=I}.

We carry on the computations with a general $b$. After a change of
variables, $A$ is now equal to
 \begin{multline*}
A = {(2\pi h)^{-1}} e^{i\sqrt{2H_0}\sin\alpha_0\theta_0 /h}e^{iJ_0\theta_0/h}\int  b(\theta_0, J_0-\sqrt{2H}\sin\alpha_0,  ht, H)\\
  e^{iH(t+t_0)/h} e^{-i\sqrt{2H}\sin\alpha_0\theta_0 /h} e^{-iH_0t/h}dHdt \\
= {(2\pi h)^{-1}}e^{iH_0t_0/h}e^{i\sqrt{2H_0}\sin\alpha_0\theta_0 /h}e^{iJ_0\theta_0/h}\int  b(\theta_0, J_0-\sqrt{2{(H+H_0)}}\sin\alpha_0,  h(t-t_0), H_0+H)\\
  e^{iHt/h} e^{-i\sqrt{2(H+H_0)}\sin\alpha_0\theta_0 /h}  dHdt  .
 \end{multline*}
 Standard application of the method of stationary phase shows that this expression is of order $O(h^\infty)$ if $H_0$ is away from the support of $b$. Besides, the phase has a single nondegenerate critical point at $(t,H) = ( \frac{\sin\alpha_0\theta_0}{\sqrt{2H_0}}, 0)$, so that {\em uniformly in $t_0\in\R$} the method of stationary phase yields
\begin{multline*}
A = {(2\pi h)^{-1}}
e^{iH_0t_0/h}e^{i\sqrt{2H_0}\sin\alpha_0\theta_0 /h}e^{iJ_0\theta_0/h} \\
\left( {(2\pi h) e^{- i\sqrt{2H_0}\sin\alpha_0\theta_0 /h}}
b\left(\theta_0, J_0- {\sqrt{2H_0}} \sin\alpha_0,
h(-t_0+\frac{\sin\alpha_0\theta_0}{\sqrt{2H}}), H_0\right) \right)
  +O(h) .
\end{multline*}
  This is
  $$ {
  A = e^{iH_0t_0/h}e^{iJ_0\theta_0/h} b\left(\theta_0, J_0-\sqrt{2H_0}\sin\alpha_0,  -ht_0 , H_0\right)
  +O(h)},
  $$
  where $O(h)$ is uniform if $\theta_0$ stays in a fixed compact set. This concludes the proof of~\eqref{conj V}.
 \end{proof}

Recalling the definition of $W_h(a)$ in~\eqref{e:Wsc}, we thus
have
 \begin{eqnarray*}
 W_h(a)&=&\la \scrV\scrU u_h, (\scrV\scrU\Op_h(a(z, \xi, t, hH)\scrU^*\scrV^*) \scrV\scrU u_h\ra_{L^2(\R_s\times \T_\theta\times \R_H)}\\
 &=&\la \scrV\scrU u_h, \Op_h( \tilde b(s, \theta, J', E, H, ht)) \scrV\scrU u_h\ra_{L^2(\R_s\times \T_\theta\times \R_H)}+O(h)
 \end{eqnarray*}
 where
\begin{equation}
\label{e:btilde}
\tilde b(s, \theta,J', E , H, ht))=a\circ\Phi (s, \theta, E,
J'-\sin\alpha_0\sqrt{2H}, -ht, H),
\end{equation}
 and $\T_\theta=\R/2\pi\Z$ is
the circle in which the variable $\theta$ takes values. By
\eqref{e:localt}, $u_h$ may actually be restricted to $|t|<2$ in
this formula, so that it is safe to apply $\scrV$ to $\scrU  u_h$.

To work in our new coordinates, we now define
\begin{equation}\label{e:braper}\la w_h, b\ra=\la \scrV\scrU u_h, \Op_h(  b(s, \theta,E, J', H, ht)) \scrV\scrU u_h\ra_{L^2(\R_s\times \T_\theta\times \R_H)}\end{equation}
for symbols $b$ that satisfies
\begin{itemize}
\item[(A)] the symbol $b$ is compactly supported w.r.t. $E, J'$,
$t$ and $H$, and $2\pi$-periodic w.r.t. $\theta$.
\end{itemize}

We then recover $W_h(a) = \la w_h, \tilde{b}\ra$ with $\tilde{b}$ and $a$ linked by~\eqref{e:btilde}.

\begin{remark}\label{r:trunc}
Note that the bracket \eqref{e:braper} can also be written as
follows
\begin{equation}\label{e:braper1}\la w_h, b\ra=\la \scrV\scrU u_h, \Op_h(  \chi_0(\theta)b(s, \theta,E, J', H, ht)) \scrV\scrU u_h\ra_{L^2(\R_s\times \R_\theta\times \R_H)}\end{equation}
for any $\chi_0 \in C^\infty_c(\R)$ satisfying
$\sum_{k\in\Z}\chi_0(\theta+2\pi k)\equiv 1$ on $\R$. Indeed, we
have \begin{equation}
\label{e:thetasort}
\Op_h(  \chi_0(\theta)b) =  \chi_0(\theta)\Op_h( b)
\end{equation} (because $\Op$ denotes the standard quantization) and we
write for any $2\pi$-periodic function $f \in L^1_{\loc}(\R)$,
$\int_\R \chi_0(\theta) f(\theta) d \theta=   \int_\T f(\theta) d
\theta$.
Because of~\eqref{e:thetasort}, we may also take $\chi_0=\mathds{1}_{(0, 2\pi)}$ when needed.
\end{remark}

 \subsection{Second microlocalization \label{s:secondmicro}}
 We now introduce two auxiliary distributions which describe more precisely how $w_h  $ concentrates on the set
 $$  \left\{(s,\theta, E,J, H, t) \in \Phi^{-1}(\ovl{\ID}\times (\R^2\setminus\{0\})))\times \R^2, \text{ such that }  -\frac{J}{\sqrt{2H}} = \sin \alpha_0\right\}
 $$
(which intersection with $\{E=\sqrt{2H}\}$ is equal to $\cI_{\alpha_0} \cap \{E=\sqrt{2H}\}$).

For this, we define an appropriate class of symbols depending on
an additional variable $\eta$, which later in the calculations
will be identified with $\frac{J'}h = \frac{J+\sqrt{2H}\sin\alpha_0}h$.

\begin{definition}
\begin{itemize}
\item We denote by $\mathcal{S}$ the class of smooth functions $b(s,
\theta, E, J',\eta, H, t)  $ on $\R^7$, supported away from
$\{E=0\}$ and that satisfy condition {\rm (A)} in the variables $(
s, \theta, E, J',H,t)$, and, in addition,

\item[(E)] $b$ is homogeneous of degree zero at infinity in
$\eta\in \R$. That is, there exist $R_{0}>0$ and
$b_{{\rm{hom}}}\in C^{\infty}\left( \R^4 \times \{-1, +1\} \times \R^2 \right)$ with%
\[
b(s, \theta, E, J',\eta,H, t)  =b_{{\rm{hom}}}\left(   s, \theta,
E, J',\frac{\eta }{\left\vert \eta\right\vert }, H, t\right)
\text{,\quad for }\left\vert
\eta\right\vert >R_{0}\text{ and }(s, \theta, E, J',   H, t)  \in   \R^6 .%
\]

\item We denote by $\cS^\sigma$ those symbols $b \in \cS$ that
satisfy conditions {\rm (B)} and {\rm (C)}  (for all $H, t$):

\item[(B)] $ b \left(\cos\alpha,\theta, E,  J'\right)=
b \left(-\cos\alpha,\theta + \pi+2\alpha , E, J'\right)$\\
for all $\theta, E, J'$, and for
$\alpha=-\arcsin\left(\frac{J'-\sqrt{2H}\sin\alpha_0}E\right)$.

\item[(C)] $\partial_s^k b \left(\cos\alpha,\theta, E,  J'\right)=
\partial_s^kb \left(-\cos\alpha,\theta + \pi+2\alpha , E, J'\right)$\\
for all $k$, for all $\theta, E, J'$, and for
$\alpha=-\arcsin\left(\frac{J'-\sqrt{2H}\sin\alpha_0}E\right)$.

\item We denote by $\cS_{\alpha_0}^\sigma$ those symbols $b \in
\cS^\sigma$ satisfying the invariance condition {\rm (D)}:
\item[(D)]
$[(\alpha_0-\alpha)\partial_\theta+\cos\alpha\partial_s]b(s, \theta, E, J')=0$\\
for all $s, \theta, E, J'$, and for
$\alpha=-\arcsin\left(\frac{J'-\sqrt{2H}\sin\alpha_0}E\right)$.

\end{itemize}
\end{definition}

Let $\chi\in C_{c}^{\infty}\left(  \mathbb{R}\right)  $ be a
nonnegative cut-off function that is identically equal to one near
the origin and let $R>0$. For
$b\in\mathcal{S} $, we define%
\begin{multline*}
\left\langle w_{h,R}^{\alpha_0}   ,b\right\rangle := \\
 \left\la \scrV\scrU u_h,  \Op_h\left( \left(  1-\chi\left(  \frac{ J'}{Rh}\right)\right) \chi_0(\theta)b(s, \theta,E, J', \frac{ J'}{h}, H, ht)\right)\scrV\scrU u_h\right\ra_{L^2(\R_s\times \R_\theta\times \R_H)}  ,
\end{multline*}
and%
\begin{equation}\left\langle w_{\alpha_0, h,R}    ,b\right\rangle :=
 \left\la \scrV\scrU u_h,  \Op_h\left(  \chi\left(  \frac{ J'}{Rh} \right) \chi_0(\theta)b(s, \theta,E, J', \frac{ J'}{h}, H, ht)\right)\scrV\scrU u_h\right\ra_{L^2(\R_s\times \R_\theta\times \R_H)}    \label{subL},
\end{equation}

The Calder\'{o}n-Vaillancourt theorem~\cite{CV:71} ensures that both $
w_{h,R}^{\alpha_0}$ and $  w_{\alpha_0,h,R}$ are bounded in
$\mathcal{S}^{\prime}$.
After possibly extracting subsequences, we
have the existence of a limit: for every $b\in \mathcal{S} $,%
\[
  \left\langle {\mu}^{\alpha_0
}   ,b\right\rangle :=\lim_{R\rightarrow\infty}%
\lim_{h\rightarrow0^{+}}   \left\langle w_{h,R}^{\alpha_0}
,b\right\rangle ,
\]
and%
\begin{equation}
   \left\langle {\mu}_{\alpha_0
}   ,b\right\rangle  :=\lim_{R\rightarrow\infty}%
\lim_{h\rightarrow0^{+}}   \left\langle
 w_{\alpha_0,h,R}   ,b\right\rangle . \label{doublelim}%
\end{equation}

Positivity properties are described in the next proposition.

\begin{proposition}
\label{thm 1st2micro}
\begin{itemize}
\item[(i)] The distribution ${\mu}^{\alpha_0}$ is a nonnegative
Radon measure. In addition, ${\mu}^{\alpha_0} $ is nonnegative,
$0$-homogeneous and supported at infinity in the variable $\eta$
($i.e.$, it vanishes when paired with a compactly supported
function). As a consequence, ${\mu}^{\alpha_0 }  $ may be
identified
 with a nonnegative measure on
$\R^4\times\{-1, +1\}\times\R_t\times \R_H $.

\item[(ii)]
The projection of ${\mu}_{\alpha_0 } $ on $\R^4_{s, \theta, E,
J'}\times\R^2_{H, t} $ (that is, $\int_\R {\mu}_{\alpha_0
}(d\eta)$) is a nonnegative measure, carried on $\{J'=0\}$.
\end{itemize}
 \end{proposition}
Moreover, both ${\mu}_{\alpha_0}  $ and ${\mu}^{\alpha_0}$ are carried by
the set $\{E=\sqrt{2H}\}$, as can be seen from \eqref{e:final} and
\eqref{e:final2}.{ Note also that the argument of Proposition~\ref{p:regt} proves that ${\mu}^{\alpha_0}$ enjoys $L^\infty$ regularity in the time variable.}

Proposition~\ref{thm 1st2micro} (i) is proved at the beginning of Section~\ref{s:strucprop}, whereas (ii) shall be a consequence of Section~\ref{s:sy}.

\begin{remark}\label{r:mms}If $a=a(z, \xi, t, H)$ is a function on $\ovl{\ID}\times \R^2\times \R^2$, let us define
\begin{eqnarray*}
m^{\alpha_0}\left(  a \right)  &:=&
{\mu}^{\alpha_0}\left(b \mathds{1}_{J'=0})\right)   , \\
 m_{\alpha_0}\left(   a \right)
&:=&{\mu}_{\alpha_0}\left(   b \right)   .
\end{eqnarray*}
where $b(s, \theta, E, J', H, t)=a\circ\Phi(s, \theta, E,
J'-\sin\alpha_0\sqrt{2H}, -t, H)$ (a function that does not depend
on the additional variable $\eta$). Then we have
\begin{equation}
\mu   \rceil_{\cI_{\alpha_0}}%
=m^{\alpha_0}  +m_{\alpha_0}   .
\label{decomposition}%
\end{equation}
Thus, understanding $\mu   \rceil_{\cI_{\alpha_0}}$ amounts to
understanding both $m^{\alpha_0}$ and $m_{\alpha_0}$, which we
shall do by understanding the structure of $\mu^{\alpha_0}$ and
$\mu_{\alpha_0}$.
\end{remark}

The following proposition states that both distributions ${\mu}_{\alpha_0}$ and ${\mu}^{\alpha_0}$ are invariant under the billiard flow, as $\mu$.

\begin{proposition}
\label{Lemma Inv}The distributions ${\mu}_{\alpha_0}  $ and ${\mu}^{\alpha_0}$ enjoy the following property:%
\[
 \la{\mu}_{\alpha_0}  ,   E\,\partial_s b \ra=0  ,\quad \la{\mu}^{\alpha_0} ,  E\,\partial_s b\ra
= 0%
\]
for every $b\in\mathcal{S}^\sigma$.
\end{proposition}

\begin{proof}
 We use as a ``black-box'' the technical calculations developed in Appendix~\ref{coord-polar}. The main point of these calculations is to understand how an operator of the form
 $$\scrU^*\Op_h\left(  P(s, \theta,E, J, t, hH)\right)\scrU $$ preserves or modifies the Dirichlet boundary condition, according to the properties of $P$ (the technical difficulty is that our new coordinates $(s, \theta,E, J)$, well-adapted to the dynamics, are not adapted to express the Dirichlet boundary condition).

In the proof of Proposition \ref{Lemma Inv} for ${\mu}_{\alpha_0}
$, we consider the function $P$
 \begin{equation}\label{e:bP}P(s, \theta,E, J, t, H)=b\left(s, \theta, E, J+\sin\alpha_0\sqrt{2H}, \frac{J+\sin\alpha_0\sqrt{2H}}h, H, -t\right)
  \chi\left(  \frac{ J+\sin\alpha_0\sqrt{2H}}{Rh} \right).
  \end{equation}
To prove the result for ${\mu}^{\alpha_0}   $, the argument is the
same with the function
 \begin{multline}\label{e:tildeP}
\tilde P(s, \theta,E, J, t, H)=  \\
b\left(s, \theta, E, J+\sin\alpha_0\sqrt{2H},
\frac{J+\sin\alpha_0\sqrt{2H}}h, H, -t\right)
 (1- \chi)\left(  \frac{ J+\sin\alpha_0\sqrt{2H}}{Rh} \right).
  \end{multline}

 \begin{enumerate}
 \item The operator $\scrU^*\Op_h\left(  P(s, \theta,E, J, t, hH)\right)\scrU$ is expressed as a pseudodifferential operator $\cA_E(P)$ on $\R^2$ (modulo a small remainder) in polar coordinates $$z=(-r\sin u,  r\cos u).$$ Note that the polar coordinates are the ones adapted to our boundary problem, since  the boundary is given by the equation $r=1$.

Thus we have
$$\lim \la u_h, \scrU^*\Op_h\left(  P(s, \theta,E, J, t, hH)\right)\scrU u_h\ra = \lim \la u_h, \cA_E(P)  u_h\ra.$$

 \item We then introduce a pseudodifferential operator $\cA_H(P)$, having the property that the symbols of $\cA_E(P)$ and $\cA_H(P)$ coincide on $\{|\xi|^2=2H\}$.
More precisely, we are able to prove (Lemma \ref{l:EH})
$$
  \lim \la u_h, \cA_E(P)  u_h\ra = \lim
\la u_h, \cA_H(P)  u_h\ra.
$$

\item The explicit expression of $\cA_H(P)$ reads
\begin{equation}\label{e:ABterms}A(r, u, \sqrt{2hD_t}, hD_u, t)  + B(r, u, \sqrt{2hD_t},h D_u, t)  \circ hD_r
\end{equation}
modulo terms of order $O(h)$, where $z=(-r\sin u,  r\cos u)$ is
the decomposition in polar coordinates. The functions $A, B, C, D$
are expressed explicitly in terms of $P$ in Proposition~\ref{p:summary}. If
$P$ satisfies the symmetry condition (B), then $B\equiv 0$ for
$r=1$.

 \item Finally, we show in Proposition \ref{p:commutH} that
 $$\lim_{h\To 0} \la u_h,\cA_H(E\, \partial_s P)u_h\ra=\lim_{h\To 0} \left\la u_h,\left[-\frac{ih\Delta}2, \cA_H( P)\right]u_h\right\ra$$
 where $\Delta$ is the laplacian on $\R^2$.
 On the other hand, if $P$ depends on $t$, we have
$$ [\partial_t, \cA_H( P)]=\cA_H( \partial_t P).$$

 \end{enumerate}

 The proof of Proposition \ref{Lemma Inv} now goes as follows
 (with $b$ and $P$ related by \eqref{e:bP}):
 \begin{align}
 \nonumber\la{\mu}_{\alpha_0}  ,   E\, \partial_s b \ra
& = \lim_{R\rightarrow\infty}%
\lim_{h\rightarrow0^{+}}   \left\langle
 w_{\alpha_0,h,R}   ,b\right\rangle\\
\nonumber&  =\lim_{R\rightarrow\infty}%
\lim_{h\rightarrow0^{+}} \\
& \quad \left\la \scrV\scrU u_h,  \Op_h\left(     E\, \partial_s P(s, \theta,E, J'-\sin\alpha_0\sqrt{2H}, H, -ht)\right) \scrV\scrU u_h\right\ra_{L^2(\R_s\times \T_\theta\times \R_H)} \\
\nonumber
&  =\lim_{R\rightarrow\infty}%
\lim_{h\rightarrow0^{+}}  \left\la \scrU u_h,  \Op_h\left(     E\,
\partial_s P(s, \theta,E, J,t, hH)\right) \scrU
u_h\right\ra_{L^2(\R_s\times \T_\theta\times \R_t)}
\\
\nonumber & =\lim_{R\rightarrow\infty}%
\lim_{h\rightarrow0^{+}} \left\la u_h,  \cA_H(E\, \partial_s  P)u_h\right\ra \\
\nonumber &   =  \lim_{R\rightarrow\infty}%
\lim_{h\rightarrow0^{+}}  \left\la  u_h, \left[-\frac{ih}2\Delta
  ,  \cA_H(   P)\right]u_h\right\ra\\
 \nonumber  &   =  \lim_{R\rightarrow\infty}%
\lim_{h\rightarrow0^{+}}  \left\la  u_h, \left[-\frac{ih}2\Delta
  +ihV-h\partial_t,  \cA_H(   P)\right]u_h\right\ra\\
 &   = \lim_{R\rightarrow\infty}%
\lim_{h\rightarrow0^{+}}  \left\la  \frac{ih}2 \frac{\partial
u_h}{\partial n}\otimes \delta_{\partial \ID},   \cA_H(  P)  u_h
\right\ra
   -\lim  \left\la u_h,   \cA_H( P)  \frac{ih}2 \frac{\partial u_h}{\partial n}\otimes \delta_{\partial \ID}\right\ra \label{e:lastline}
 \end{align}
The last line comes from the fact that $u_h$, extended to $\R^2$
by the value $0$ outside $\ID$, satisfies
\begin{equation}\left(-\frac{ih}2\Delta
  +ihV-h\partial_t\right)u_h= \frac{ih}2 \frac{\partial u_h}{\partial n}\otimes \delta_{\partial \ID}\label{e:schrodelta}
  \end{equation}
where $\Delta$ is the laplacian on $\R^2$.

We use now the explicit expression \eqref{e:ABterms} of
$\cA_H(P)$, modulo terms that vanish at the limit. Using the fact
that $u_h$ satisfies Dirichlet boundary conditions, and the fact
that $B$ vanishes for $r=1$ if $P$ satisfies the symmetry
condition (B), we see that the last line \eqref{e:lastline}
vanishes.

Note that only the limit $h\To 0$ was actually used, so that the
result holds even before taking the limit $R\To +\infty$.
\end{proof}

\begin{remark}\label{r:boundary}More generally, let $b(s, \theta, E, J, \eta, H, t)$ be a smooth function on $\R\times \R/2\pi\Z \times \R^5$ with bounded deratives, and compactly supported w.r.t.  $s,  E, J, H, t$.
Let ${\mathbf P}(s, \theta, E, J, t, H)= b(s, \theta, E, J
+\sin\alpha_0 \sqrt{2H}, \frac{J +\sin\alpha_0 \sqrt{2H}}h, H,
-t)$. Then, the same proof yields,without using the symmetry
condition {\rm (B)},  the formula
 \begin{multline} \label{e:boundaryterm}
 \lim_h  \left\la \scrV\scrU u_h,  \Op_h\left(     E\, \partial_s b(s, \theta,E, J',\frac{J'}h, H, ht)\right) \scrV\scrU u_h\right\ra_{L^2(\R_s\times \T_\theta\times \R_H)}\\
= \lim_h  \left\la \scrU u_h,  \Op_h\left(     E\, \partial_s
{\mathbf P}(s, \theta,E, J,t, hH)\right) \scrU
u_h\right\ra_{L^2(\R_s\times \T_\theta\times \R_t)}
  \\ =-\lim_{ h}  \left\la  h \frac{\partial u_h}{\partial n} ,  \mathbf{B}(1, u, \sqrt{2hD_t}, hD_u, t)   h\frac{\partial u_h}{\partial n} \right\ra_{L^2(\partial\ID\times \R)}
    \end{multline}
where $ \mathbf{B}$ is the function associated to ${\mathbf P}$ by
the formulas of Proposition~\ref{p:summary}. Again, if ${\mathbf
P}$ satisfies {\rm (B)}, the operator $\mathbf{B}(1, u,
\sqrt{2hD_t}, hD_u, t)$ vanishes.

This formula, relating the semiclassical measures of boundary data
to the semiclassical measures of interior data, is analogous to
formula \eqref{e:section} but is expressed in a different set of
coordinates.

Applying \eqref{e:boundaryterm} to $\frac{s}E{\mathbf P}$ instead
of  ${\mathbf P}$ (that is, $\frac{s}E b$ instead of $b$) has the
following consequence that will be used later
 \begin{multline} \label{e:boundaryterm2}
 \lim_h  \left\la \scrU u_h,  \Op_h\left( ({\mathbf P}+ s\partial_s {\mathbf P})(s, \theta,E, J,t, hH)\right) \scrU u_h\right\ra_{L^2(\R_s\times \T_\theta\times \R_t)}\\
 =-\lim_{ h}  \left\la  h \frac{\partial u_h}{\partial n} , (E^{-2} \mathbf{P})^\sigma(1, u, \sqrt{2hD_t}, hD_u, t)   h\frac{\partial u_h}{\partial n} \right\ra_{L^2(\partial\ID\times \R)}
 \end{multline}
 with the notation \eqref{e:sym}.
\end{remark}

The following result states that both $\mu^{\alpha_0}$ and
$\mu_{\alpha_0}$ have some extra regularity (for two different
reasons).

\begin{theorem}
\label{Thm Properties}
\begin{itemize}
\item[(i)]  The measure $\mu^{\alpha_0}$ satisfies the invariance property:%
\begin{equation}
\la{\mu}^{\alpha_0}  , \partial _{\theta} b \ra =0 , \qquad \text{for every $b$ in $\cS_{\alpha_0}^\sigma$.}
\label{e:2minv}
\end{equation}%

\item[(ii)] The distribution $\mu_{\alpha_0}$ is concentrated on $\{J' =0\}$ and its projection onto the variables $(s, \theta)$ is a nonnegative absolutely continuous measure.
\end{itemize}
 \end{theorem}

Section~\ref{s:strucprop} is devoted to the study of the properties of $\mu^{\alpha_0}$ and gives the proofs of Proposition~\ref{thm 1st2micro} (i) and Theorem~\ref{Thm Properties} (i).
The study of the structure of $\mu_{\alpha_0}$ is performed in Section~\ref{s:sy} using the notion of
second-microlocal measures. This structure will imply \eqref{e:structnualpha} in Theorem \ref{t:precise} (iii). In particular, we prove at the end of Section~\ref{s:sy} that it yields Theorem~\ref{Thm Properties}~(ii).

\subsection{Structure and propagation of ${\mu}^{\alpha_0}$}
\label{s:strucprop}
In this section, we prove Proposition \ref{thm 1st2micro}~(i) and the invariance property given by Theorem \ref{Thm Properties}~(i).

The positivity of ${\mu}^{\alpha_0}$ can be deduced following
the lines of \cite{FermanianGerardCroisements} Section~2.1, or those of
the proof of Theorem 1 in \cite{GerardMDM91}; or also Corollary 27
in \cite{AnantharamanMaciaTore}. The argument will not be
reproduced here. Given $b\in\mathcal{S}$ there exists $R_{0}>0$
and $b_{{\rm{hom}}}\in C_{c}^{\infty}\left(
 \IR^4\times \{-1, +1\}\times \R^2\right)  $
such that
\[
b\left(s, \theta, E, J', \eta, H, t   \right)
=b_{{\rm{hom}}}\left( s, \theta, E, J',
\frac{\eta}{\vert\eta\vert}, H, t \right)  ,\quad\text{for
}\left\vert \eta\right\vert \geq R_{0}.
\]
Clearly, for $R$ large enough, the value $\left\langle
w_{h,R}^{\alpha_0 }   ,b\right\rangle $ only depends on
$b_{{\rm{hom}}}$. Therefore, the limiting distribution
${\mu}^{\alpha_0}  $ can be viewed as an element of the dual space
of $C_{c}^{\infty }\left( \IR^4\times \{-1, +1\}\times \R^2\right)
$.
Its positivity implies that it is a measure, which proves Proposition
\ref{thm 1st2micro}~(i).

\bigskip
We now assume that $b\in \cS^\sigma_{\alpha_0}$ and prove the
invariance property Theorem \ref{Thm Properties} (i).

Let $b\in  \cS^\sigma_{\alpha_0}$, and define $\tilde P$ as in
formula \eqref{e:tildeP}.
 Because of property (D) in the definition of the class $\mathcal{S}^\sigma_{\alpha_0}$, we have:%
\[ \partial_\theta \tilde P\left(s, \theta, \sqrt{2H}, J, t, H \right) = -\frac{\cos\alpha}{\alpha_0-\alpha}\partial_s \tilde P\left(s, \theta,  \sqrt{2H}, J, H, t \right)
\]
where $\alpha= -\arcsin\left(\frac{J}{ \sqrt{2H}}\right) $. The
crucial point in what follows is that
$\left|\frac{\cos\alpha}{\alpha_0-\alpha}\right|\leq \frac{C}{hR}$
on the support of $\tilde P\left(s, \theta, \sqrt{2H}, J, t, H
\right)$.

Recall that by definition
$$
 \la{\mu}^{\alpha_0}  ,   \partial_\theta b \ra= \lim_{R\rightarrow\infty}%
\lim_{h\rightarrow0^{+}}   \left\langle
 w_{h,R}^{\alpha_0}   ,\partial_\theta b\right\rangle .$$
Let us first fix $R$ and study the limit $h\To 0$. Arguing as in
the proof of Proposition \ref{Lemma Inv}, we have
\begin{align}\nonumber
 \lim_{h\rightarrow0^{+}}   \left\langle
 w_{h,R}^{\alpha_0}   ,\partial_\theta b\right\rangle
&=\lim  \left\la \scrU u_h,  \Op_h\left(   \partial_\theta \tilde P(s, \theta,E, J, t,hH) \right) \scrU u_h\right\ra_{L^2(\R_s\times \T_\theta\times \R_H)} \\
\nonumber &= \lim  \left\la u_h,  \cA_E\left( \partial_\theta \tilde P\right)u_h\right\ra  \\
\nonumber &= \lim  \left\la u_h,  \cA_H\left( \partial_\theta \tilde P\right)u_h\right\ra\\
\nonumber &= \lim  \left\la u_h,  \cA_H\left( -\frac{\cos\alpha}{(\alpha_0-\alpha)}   \partial_s \tilde P\right)u_h\right\ra\\
  \nonumber &=  \lim  \left\la  u_h, \left[-\frac{ih}2\Delta
  +ihV-h\partial_t,  \cA_H\left( -\frac{\cos\alpha}{E(\alpha_0-\alpha)}    \tilde P\right)\right]u_h\right\ra +O(R^{-1})\\
\label{e:lastline2}  &= \lim  \left\la  \frac{ih}2 \frac{\partial
u_h}{\partial n}\otimes \delta_{\partial \ID},  \cA_H\left(
-\frac{\cos\alpha}{E(\alpha_0-\alpha)}   \tilde P\right)  u_h
\right\ra \\ \nonumber
  & \qquad -\lim  \left\la u_h,   \cA_H\left( -\frac{\cos\alpha}{E(\alpha_0-\alpha)}   \tilde P\right) \frac{ih}2 \frac{\partial u_h}{\partial n}\otimes \delta_{\partial \ID}\right\ra +O(R^{-1})
 \end{align}
 where we used again \eqref{e:schrodelta}.

 But $\cA_H\left( -\frac{\cos\alpha}{E(\alpha_0-\alpha)}    \tilde P\right)$ equals (modulo terms which only add an error $O(R^{-1})$ to the whole calculation)
 \begin{multline*}-\tilde A(r, u, \sqrt{2hD_t}, hD_u, t)
 \frac{\cos\alpha(hD_u, hD_t)}{\sqrt{2hD_t}(\alpha_0-\alpha(hD_u, hD_t))}\\- \tilde B(r, u, \sqrt{2hD_t}, hD_u, t)  \frac{\cos\alpha(hD_u, hD_t)}{\sqrt{2hD_t}(\alpha_0-\alpha(hD_u, hD_t))}\circ hD_r
  \end{multline*}
 where $z=(-r\sin u, r\cos u)$ is the decomposition in polar coordinates, and $\tilde A, \tilde B$ are the functions associated to $\tilde P$ by the formulas of Proposition~\ref{p:summary}.

 If $\tilde P$ satisfies (B) then $\tilde B\equiv 0$ for $r=1$.
 Since $u_h$ satisfies Dirichlet boundary conditions, we see that the last terms in \eqref{e:lastline2} vanish.

 To conclude the proof, we take $R\To +\infty$ after taking $h\To 0$, so that the terms estimated as $O(R^{-1})$ vanish.

 \subsection{Second microlocal structure of ${\mu}_{\alpha_0}$\label{s:sy}}
If $\cH$ is a Hilbert space, we shall denote by $\mathcal{L}\left(  \cH\right)  $, $\mathcal{K}\left(
\cH\right)  $ and $\mathcal{L}^{1}\left( \cH\right)  $ the spaces
of bounded, compact and trace class operators on $\cH$. It is well
known that $\mathcal{L}^{1}\left( \cH\right)  $ is the dual of
$\mathcal{K}\left( \cH\right)  $. A measure on a polish space $T$,
taking values in $\mathcal{L}^{1}\left(  H\right)  $, is defined
as a bounded linear functional $\rho$ from $C_{c}\left(  T\right)
$ to $\mathcal{L}^{1}\left( H\right)  $; $\rho$ is said to be
nonnegative if, for every nonnegative $b\in C_{c}\left(  T\right)  $,
$\rho\left(  b\right)  $ is a nonnegative hermitian operator. The set
of such measures is denoted by $\mathcal{M}_{+}\left(
T;\mathcal{L}^{1}\left(  H\right)  \right)  $; they can be
identified in a natural way to nonnegative linear functionals on
$C_{c}\left(  T;\mathcal{K} \left(  H\right)  \right)  $.
Background and further details on operator-valued measures may be
found for instance in \cite{GerardMDM91}.

For each $\omega\in \R/2\pi\Z$, let us define $\cH_\omega$, the
space of functions $f$ on $\R$ satisfying
$f(\theta+2\pi)=f(\theta)e^{i\omega}$ and that are
square-integrable on $(0, 2\pi)$.

We shall denote by $\mathcal{K}^{2\pi}  $ the space of operators
on $L^2(\R)$ whose kernel $K$ satisfies $K(\theta+2\pi ,
\theta'+2\pi )=K(\theta , \theta')$ and that define compact
operators on each $\cH_\omega$. Each Hilbert space $\cH_\omega$ is
isometric to $L^2(0, 2\pi)$ (just by restricting functions to $(0,
2\pi)$), and in this identification the kernel of $K$ acting on
$\cH_\omega$ is given by \begin{equation}
\label{e:period}K_\omega(\theta, \theta'):=\mathds{1}_{(0,
2\pi)}(\theta)\mathds{1}_{(0,
2\pi)}(\theta')\sum_{n\in\Z}K(\theta, \theta'+2\pi
n)e^{in\omega}.\end{equation} The idea of the Floquet-Bloch theory
is that it is completely equivalent to know $K(\theta , \theta')$
and to know $K_\omega(\theta, \theta')$ for almost all $\omega$,
by decomposing
$$L^2(\R)=\int_{\oplus} \cH_\omega d\omega.$$
 Besides, $K$ is a nonnegative (resp. bounded) operator if and only if $K_\omega$ is nonnegative (resp. bounded) for a.e. $\omega$.

An example of an operator in $\mathcal{K}^{2\pi}  $ is
\begin{equation}\label{e:Kb}K_{b, h, R} (s, E, H, t)=b (s, \theta, E, hD_\theta, D_\theta, H, t)\chi(D_\theta/R)
\end{equation}
for $b\in \cS$ and fixed $R, h$, $ t, H, s, E$. Note that, as
$h\To 0$, we have $K_{b, h, R} (s, E, H, t)= K_{b, 0, R} (s, E, H,
t)+O_R(h)$.

If $b$ satisfies the symmetry condition (B), note that the
operator $K_{b, 0, R} (s, E, H, t)$ has the property
\begin{equation}K(\cos\alpha,E, H, t  )= R_{ \pi+2\alpha}^{-1}\circ K(- \cos\alpha, E, H, t)\circ R_{ \pi+2\alpha}\label{e:tangentialoperator}\end{equation}
where $R$ is a translation operator on $L^2\left(
\IR_{\theta}\right)$: $R_{\alpha}f(\theta)=f(\theta-\alpha)$ and
where
$\alpha=\arcsin\left(\frac{\sin\alpha_0\sqrt{2H}}{E}\right).$ In
particular,
\begin{equation}K(\cos\alpha_0,\sqrt{2H}, H, t  )= R_{ \pi+2\alpha_0}^{-1}\circ K(- \cos\alpha_0, \sqrt{2H}, H, t)\circ R_{ \pi+2\alpha_0}\label{e:tangentialoperator0}\end{equation}

\begin{remark} The fact that the orbits of the billiard flow are periodic on $\cI_{\alpha_0}$ ($\alpha_0\in \pi\IQ$) is reflected in the fact that the function
$s\mapsto K(s,\sqrt{2H}, H, t  )$ is periodic, if $K$ satisfies
\eqref{e:tangentialoperator0}.
\end{remark}

\bigskip

For $K\in C_{c}^{\infty}\left(\R/2\pi\Z\times
\R_s\times\R_E\times\R_H\times\R_t ;  \mathcal{K} \left(
L^2(0, 2\pi)\right) \right)$, let us define:%
\begin{multline}
\label{e:nh}\left\langle n_{h}^{\alpha_0} ,K\right\rangle =(2\pi
h)^{-2}\int_{\omega= 0}^{2\pi}\sum_{H, \sqrt{2H}\sin\alpha_0/h
\equiv \omega (2\pi)}
\frac{h}{\sin\alpha_0}\sqrt{\frac{H}2}\int_{s, s', E, H', t}\\
\left\langle \chi_0 \scrV\scrU u_{h}(s', H'),  K\left(\omega, s,
E, H', ht \right) {\chi_0} \scrV\scrU u_{h}(s, H)\right\rangle
_{L^{2}\left(
 0, 2\pi \right)  }\\e^{iE(s'-s)/h} e^{it(H'-H)/h}  dsdEds' dH' dt
\end{multline}
where $\chi_0$ is $\mathds{1}_{(0, 2\pi)}$ as in Remark
\ref{r:trunc}. This is also
$$\la {\chi_0} \scrV\scrU u_{h}, \cK {\chi_0} \scrV\scrU u_{h}\ra_{L^2 (\R_s\times \R_H, L^2(0, 2\pi))}$$
where $\cK$ is the pseudodifferential operator with
operator-valued symbol:
\begin{equation}\label{e:K}\int_{\omega= 0}^{2\pi}\sum_{H, \sqrt{2H}\sin\alpha_0/h \equiv \omega (2\pi)} \frac{h}{\sin\alpha_0}\sqrt{\frac{H}2}    K\left(\omega, s, E, H', ht \right).
\end{equation}

\begin{remark}As noted earlier, it is equivalent (by the relation \eqref{e:period}) to consider a family $K(\omega)$ of kernels on $(0, 2\pi)^2$ and a kernel $K$ on $\R^2$ satisfying $K(\theta, \theta')=K(\theta+2\pi, \theta'+2\pi)$. With this identification in mind, formula \eqref{e:nh} amounts to
\begin{multline}\label{e:nh1}
(2\pi h)^{-2} \int_{s, s', E,H, H', t}  \left\langle \chi_0
\scrV\scrU u_{h}(s', H'),  K\left( s, E, H', ht \right) \scrV\scrU
u_{h}(s, H)\right\rangle _{L^{2}\left(
 \R \right)  }\\e^{iE(s'-s)/h} e^{it(H'-H)/h}  dsdEds' dH dH' dt\\
 =\left\langle \chi_0 \scrV\scrU u_{h},
 K\left( s, h D_s, H, h^2D_t \right) \scrV\scrU u_{h}\right\rangle_{L^2 (\R_s\times \R_H, L^2( \R_\theta))}
 \end{multline}

The motivation for rewriting \eqref{e:nh1} in the apparently more
complicated form \eqref{e:nh} is that it will be more convenient
to use the compact operators $K(\omega)$ on each $\cH_\omega$ than
the non-compact operator $K$ on $L^2(\R)$.
\end{remark}

The relevance of definition \eqref{e:nh} for us is that we have
the relation
\begin{eqnarray}\label{e:relevance}\la w_{\alpha_0, h, R}, b\ra &=& \left\langle n_{h}^{\alpha_0}
,K_{b, h, R}\right\rangle
\nonumber \\
&= & \left\langle n_{h}^{\alpha_0}
,K_{b, 0, R}\right\rangle +O_R(h)
\end{eqnarray}
where $K_{b, h, R}$ was defined in \eqref{e:Kb}.

\begin{proposition}
\label{p:weakstarlimit} Suppose $\left(  u^0_{h}\right)  $ is
bounded in $L^{2}\left(  \ID\right)  $. Then, modulo taking
subsequences, the following convergence takes place:
\begin{equation}
\label{e:rholambda} \lim_{h\rightarrow0^{+}}\left\langle
n_{h}^{\alpha_0} ,K\right\rangle=\int_0^{2\pi}\int_{
\R_s\times\R_E\times\R_H\times\R_t }   \Tr \left\{ K\left(\omega,
s, E, H, t\right) {\rho}_{\alpha_0}\left(d\omega,  ds, dE, dH,
dt\right) \right\} ,
\end{equation}
for every $K\in C_{c}^{\infty}\left(\R/2\pi\Z \times
\R_s\times\R_E\times\R_H\times\R_t ;  \mathcal{K} \left( L^2(0,
2\pi)\right) \right)$. In other words, $ {\rho}_{\alpha_0} $ is
the limit of $n_{h}^{\alpha_0}  $ in the weak-$\ast$ topology of
$$ \cD^{\prime }\left( \R/2\pi\Z \times \R_s\times\R_E\times
\R_H\times \R_t, \mathcal{L}^{1}\left(
 L^2(0, 2\pi) \right)  \right)   .$$

In fact, ${\rho}_{\alpha_0}$ is a nonnegative, $\mathcal{L}^{1}\left(
 L^2(0, 2\pi) \right)$-valued measures on $ \R/2\pi\Z\times \R_s\times\R_E\times \R_H\times \R_t$.

In addition, ${\rho}_{\alpha_0}$ is supported in $\{s\in
[-\cos\alpha_0, \cos\alpha_0], E=\sqrt{2H}\}$.

\end{proposition}

\begin{proof}
Note that ${\chi_0} \scrV\scrU u_{h}(s, H) $ is bounded in $L^2
(\R_s\times \R_H, L^2(0, 2\pi)) $. The Calder\'{o}n-Vaillancourt
theorem~\cite{CV:71} gives that the operators $\cK$ with symbols of the form
\eqref{e:K} are uniformly bounded with respect to $h$. Therefore,
the linear map
\begin{equation*}
L_h: K\mapsto \int_{\mathbb{R}}\la n_h^{\alpha_0},K  \ra
\end{equation*}%
is uniformly bounded as ~$h\To 0$. As a consequence, for any~$K$, up to
extraction of a
subsequence, it has a limit~$l(K)$. \\
Considering a countable dense subset of~$%
C_{c}^{\infty}\left(\R/2\pi\Z \times
\R_s\times\R_E\times\R_H\times\R_t ;  \mathcal{K} \left( L^2(0,
2\pi)\right) \right) $, and using a diagonal
extraction process, one finds a sequence~$\left( h_{n}\right) $ tending to~$%
0 $ as~$n$ goes to~$+\infty $ such that for any ~$K\in
C_{c}^{\infty}\left(\R/2\pi\Z \times
\R_s\times\R_E\times\R_H\times\R_t ;  \mathcal{K} \left( L^2(0,
2\pi)\right) \right)$, the sequence~$L_{h_{n}}(K)$ has a limit
as~$n$ goes
to~$+\infty $.\\
 The limit is a linear form on $C_{c}^{\infty}\left(\R/2\pi\Z \times \R_s\times\R_E\times\R_H\times\R_t ;  \mathcal{K} \left(
L^2(0, 2\pi)\right) \right)$, characterized by an element
$\rho_{\alpha_0}$ of the dual space ~$ \cD^{\prime }\left(
\R/2\pi\Z \times \R_s\times\R_E\times \R_H\times \R_t,
\mathcal{L}^{1}\left(
 L^2(0, 2\pi) \right)  \right)   .$

The positivity of the limit is standard. Note that it is
immediately seen in the expression~\eqref{e:nh1}.
\end{proof}

Comparing with \eqref{e:relevance}, we obtain
\begin{corollary}For every $b\in\cS$,
 \begin{multline*}
\int b(s, \theta,E, J, \eta, H, t)\mu_{\alpha_0}(ds, d\theta, dJ, d\eta, dE, dH, dt)\\
=\Tr_{L ^{2}(0, 2\pi)}\int K_{b, 0, \infty}(s, E, H, t)_{\omega}\,
\,\rho_{\alpha_0 } (d\omega,d s, dE, dH, dt)
\end{multline*}
\end{corollary}
Remember that $K_{b, 0, R}(s, E, H, t)=b(s, \theta, 0, D_\theta,
H, t)\chi(D_\theta/R)$, so that $K_{b, 0, \infty}(s, E, H, t)=b(s,
\theta, 0, D_\theta, H, t)$.

\begin{corollary}
If $b$ does not depend on $\eta$
then the above identity can be rewritten as:%
\begin{multline*}
\int b(s, \theta,E, J, H, t)\mu_{\alpha_0}(ds, d\theta, dJ, d\eta, dE, dH, dt)\\
=\Tr_{L ^{2}(0, 2\pi)}\int  m_b(s, E ,H, t)\, \,\rho_{\alpha_0 }
(d\omega,d s, dE, dH, dt)
\end{multline*}
 where $m_b(s, E, H, t)$ is the multiplication operator by $b(s, \theta, E, 0, H, t)$ acting on $L ^{2}(0, 2\pi)$.
\end{corollary}
Note that $\int b(s, \theta,E, J, H, t)\mu_{\alpha_0}(ds, d\theta, dJ, d\eta, dE, dH, dt)\geq 0$ if $b$ does not depend on $\eta$ and $b\geq 0$. Thus the projection on $\mu_{\alpha_0}$ on the variables $(s, \theta,E, J, H, t)$ defines a nonnegative measure.

We finish this section by explaining why this implies that the projection of $\mu_{\alpha_0}$ on the variables $(s, \theta)$
is absolutely continuous.  If $b\in\mathcal{S}^\sigma$ does not depend  on $\eta$, Proposition \ref{Lemma Inv} implies that
\begin{multline*}
\int b(s, \theta,E, J, H, t)\mu_{\alpha_0}(ds, d\theta, dJ, d\eta, dE, dH, dt)\\
= \int \la b\ra_{\alpha_0}( \theta,E, J, H, t)\mu_{\alpha_0}(ds, d\theta, dJ, d\eta, dE, dH, dt).
\end{multline*}
We know from Section  \ref{s:trace} that $\mu=\mu_{sc}$ does not charge the set $S$. Since $\mu_{\alpha_0}\leq \mu$ by \eqref{decomposition}, the measure $\mu_{\alpha_0}$ does not charge the set $\{s=\pm \cos\alpha_0\}$, and the previous equality actually holds for all $b\in\mathcal{S}$. If $b$ does not depend on $(E, J, H, t)$, we get the formula
\begin{equation*}
\int b(s, \theta)\mu_{\alpha_0}(ds, d\theta, dJ, d\eta, dE, dH, dt)
=\Tr_{L ^{2}(0, 2\pi)}\int  m_{\la b\ra_{\alpha_0}}\, \,\rho_{\alpha_0 }
(d\omega,d s, dE, dH, dt).
\end{equation*}
This formula, defined a priori for continuous $b$, extends to $b\in L^\infty$. If $b$ vanishes for Lebesgue-almost every $(s, \theta)$,
the multiplication operator $m_{\la b\ra_{\alpha_0}}$ vanishes on $L ^{2}(0, 2\pi)$, and
$$\int b(s, \theta)\mu_{\alpha_0}(ds, d\theta, dJ, d\eta, dE, dH, dt)=0,$$ which proves the absolute continuity.

\subsection{Propagation law for $\rho_{\alpha_0}$\label{s:propagation}}

We now show that the operator-valued measure $\rho_{\alpha_0}$
constructed in the previous section possesses some invariance
properties. Below, the notation $\la V\ra_{\alpha_0}$ stands short
for the function $\la V\ra_{\alpha_0}\circ\Phi (s, \theta, E,
-E\sin\alpha_0, t)$, a function that actually does not depend on
$s$ and is $2\pi$-periodic in $\theta$.

\begin{proposition}\label{p:invop}
(i) If $K$ satisfies \eqref{e:tangentialoperator0}, we have
\begin{equation}\int    \Tr\, E\,\partial_s K\left(\omega, s, E, H, t\right)
{\rho}_{\alpha_0}\left(d\omega,  ds, dE, dH,
dt\right)=0\label{e:invarianceop}\end{equation}

(ii) If in addition $K(s, \sqrt{2H}, H, t)$ does not depend on
$s$, we have
\begin{multline}\int    \Tr\left( - \cos^2\alpha_0\,\partial_t K+i\left[-\frac{\partial_\theta^2}2 +  \cos^2\alpha_0\la V\ra_{\alpha_0}, K\right]_\omega\right)
\left(\omega, s, E, H, t\right) \\
{\rho}_{\alpha_0}\left(d\omega,
ds, dE, dH, dt\right)  =0\label {e:propop}\end{multline}
 where $\left[-\frac{\partial_\theta^2}2 +  \cos^2\alpha_0\la V\ra_{\alpha_0}, K\right]_\omega$ means that we are considering $\partial_\theta^2$
 acting on $\cH_\omega$ (in other words, $L^2(0, 2\pi)$ with Floquet-periodic boundary condition ($f(\theta+2\pi)=f(\theta)e^{i\omega}$).
 \end{proposition}

The proof of this key proposition is postponed to the end of this Section. Let us first draw some of its consequences in view of Theorem~\ref{t:precise}.

\begin{remark}
\label{rhobar} Proposition~\ref{p:invop} (ii) implies the
following. Take $K=a(\omega, E, H, t) Id_{L^2(0,2\pi)}$ with $a$ a
scalar continous function, then

\begin{equation*}  \Tr\left(\int \partial_t a
\left(\omega,E, H, t\right) {\rho}_{\alpha_0} \left(d\omega,  ds,
dE, dH, dt\right)  \right) =0 .
\end{equation*}
Therefore, the image of ${\rho}_{\alpha_0}$ by the projection on
$\R_s$, $$\overline{\rho}_{\alpha_0}(d\omega ,d E, dH,d t) :=\int
{\rho}_{\alpha_0}(d\omega,  ds,d E, dH, dt)$$ is such that
$\Tr(\overline{\rho}_{\alpha_0})$ does not depend on~$t$.
\end{remark}

\begin{remark}
\label{radnikoPG}
The Radon-Nikodym theorem~\cite[Appendix]{GerardMDM91} implies that the operator valued measure $\overline{\rho}_{\alpha_0}$ 
 can also be written as $\overline{\rho}_{\alpha_0} = \sigma_{\alpha_0} \ell_{\alpha_0}$ where $\ell_{\alpha_0} = \Tr(\overline{\rho}_{\alpha_0})$ is a nonnegative scalar measure on $\R/2\pi\Z\times  \R_E \times \R_H$, and
$$
\sigma_{\alpha_0}  : (\R/2\pi\Z)_\omega \times  \R_E \times \R_H
\times\R_t \to \mathcal{L}^1_+ \big(L^2(0,2\pi)\big) ,
$$
is an integrable function with respect to $\ell_{\alpha_0}$,
taking values in the set of nonnegative trace-class operators on
$L^2(0,2\pi)$. Note that $\Tr(\sigma_{\alpha_0}) = 1$.
\end{remark}

\begin{corollary}
\label{corsigmaell}
Let $\overline{\rho}_{\alpha_0}$ as in Remark~\ref{rhobar} 
and let ${\ell}_{\alpha_0}$ and ${\sigma}_{\alpha_0}$ as in
Remark~\ref{radnikoPG}. Then for $\ell_{\alpha_0}$-almost every
$(\omega, E,H)$, we have
$$-
\cos^2\alpha_0\partial_t
{\sigma}_{\alpha_0}+i\left[-\frac{\partial_\theta^2}2 +
\cos^2\alpha_0\la V\ra_{\alpha_0}, {\sigma}_{\alpha_0}
\right]_\omega = 0
$$
in $\mathcal{D}'\left(\R_t ; \mathcal{L}^1_+ \big(L^2(0,2\pi)\big)
\right)$.

Therefore, for $\ell_{\alpha_0}$-almost every $(\omega, E,H)$,
$\sigma_{\alpha_0}$ coincides with a continuous function in
$$C^0\left(\R_t ; \mathcal{L}^1_+ \big(L^2(0,2\pi)\big) \right)$$
and
$$
\sigma_{\alpha_0}(\omega, E,H , t) = U_{\alpha_0 , \omega}(t)
\sigma_{\alpha_0}(\omega, E,H , 0)U_{\alpha_0 , \omega}^*(t) .
$$
\end{corollary}
where $U_{\alpha_0,
\omega}(t)$ is the unitary propagator of the equation
\[
-\cos^2\alpha_0 D_t v(t,\theta)+\left(-\frac{1}{2}\partial_\theta^2+
\cos^2\alpha_0\la
V\ra_{\alpha_0}\circ\Phi\right) v(t,
\theta)=0
\]

\begin{proof}
We first rewrite~\eqref{e:propop} for $s$-independent operators
$K$ as
\begin{equation*}
\int    \Tr \left\{\left( - \cos^2\alpha_0\,\partial_t
K+i\left[-\frac{\partial_\theta^2}2 +  \cos^2\alpha_0\la
V\ra_{\alpha_0}, K\right]_\omega\right) \sigma_{\alpha_0} \right\}
\ell_{\alpha_0} \left(d\omega,  dE, dH\right) dt =0 .
\end{equation*}
Therefore, we have
$$
\int \Tr \left\{ K \left( -\cos^2\alpha_0\,\partial_t
\sigma_{\alpha_0} +i\left[-\frac{\partial_\theta^2}2 +
\cos^2\alpha_0\la V\ra_{\alpha_0}, \sigma_{\alpha_0}
\right]_\omega\right) \right\} \ell_{\alpha_0} \left(d\omega,  dE,
dH\right) dt =0,
$$
which concludes the proof of Corollary~\ref{corsigmaell}.

\end{proof}

To conclude this section, let us now prove its main result.

\begin{proof}[Proof of Propostion~\ref{p:invop}]
As was already mentioned, it is equivalent to consider a family of
kernels depending on $\omega$, $K(\omega, s, E, H, t) (\theta,
\theta')$ defined for $(\theta, \theta')\in (0, 2\pi)^2$, and a
kernel $K( s, E, H, t) (\theta, \theta')$ defined for $(\theta,
\theta')\in \R^2$ and satisfying $K( s, E, H, t) (\theta,
\theta')=K( s, E, H, t) (\theta+2\pi, \theta'+2\pi)$. The link
between both representations is the formula
$$K(\omega, s, E, H, t) (\theta, \theta')=\sum_{n\in\Z}K( s, E, H, t) (\theta, \theta'+2n\pi)e^{in\omega}.$$

By a density argument, it is enough to treat the case where $K(s,
E, H, t)$ is smooth in $(s, E, H, t)$ and is a pseudodifferential
operator on $L^2(\R)$. By this, we mean that there is a $b_0(s,
\theta, E, \eta, H, t)\in C_c^\infty(\R\times
\R/2\pi\Z\times\R^4)$ such that $K(s, E, H, t)=b_0(s, \theta,  E,
D_\theta, H, t)$.
As $\rho_{\alpha_0}$ is supported by $\{E =\sqrt{2H}\}$, we may further assume that $K$ satisfies \eqref{e:tangentialoperator} instead of \eqref{e:tangentialoperator0}.

If $K$ satisfies \eqref{e:tangentialoperator},
then we have $ b_0 \left(\cos\alpha,\theta,E, \eta, H, t\right)= b_0
\left(-\cos\alpha,\theta + \pi+2\alpha , E, \eta, H, t\right)$ for
$\alpha=\arcsin\left(\frac{\sqrt{2H}\sin\alpha_0}{E}\right)$. We
can extend $b_0$ to a function $b(s, \theta,E,  J', \eta, H,
t)\in C_c^\infty(\R\times \R/2\pi\Z\times\R^5)$ such that, for
$J'=0$, we have $b(s,  \theta,E,0, \eta, H, t)= b_0(s,  \theta,
E,\eta, H, t)$, and such that $b$ satisfies the symmetry condition
(B) with $\sin\alpha=-\frac{J'-\sqrt{2H}\sin\alpha_0}{E}$. We are
now back to our previous notation. The proof  Proposition
\ref{p:invop} (i) goes exactly along the lines  of the proof of
Proposition \ref{Lemma Inv} (see Remark \ref{r:boundary}).

\bigskip
Let us now focus on the proof of \eqref{e:propop}.

If $K(s, \sqrt{2H}, H, t)$ does not depend on $s$, then $b_0(s,
\sqrt{2H}, \theta, \eta, H, t)$ does not depend on $s$, and we can
impose that the function $b$ constructed above satisfy equation
(D).

Letting $\eta=\frac{J'}{h}$, we note that, for $\eta$ in the
(compact) support of $b(s, \sqrt{2H}, \theta, J', \eta, H, t)$, we
have
$$\alpha-\alpha_0\sim \frac{-h\eta}{\sqrt{2H}\cos\alpha_0}(1+O(h))$$
so that
\begin{equation}\label{e:dl}\frac{-\eta\cos\alpha}{\sqrt{2H}(\alpha-\alpha_0)}\sim \frac{\cos^2\alpha_0}h (1+O(h)).\end{equation}

We set
\begin{multline*}
\mathcal{Q}_0 := \int    \Tr\left( - \cos^2\alpha_0\,\partial_t
K+i\left[-\frac{\partial_\theta^2}2 +  \cos^2\alpha_0\la
V\ra_{\alpha_0}, K\right]_\omega\right) \left(\omega, s, E, H,
t\right)\\
{\rho}_{\alpha_0}\left(d\omega,  ds, dE, dH, dt\right)  ,
\end{multline*}
so that proving \eqref{e:propop} amounts to showing that $\mathcal{Q}_0 = 0$.

First note that
\begin{align}
\mathcal{Q}_0
& =\int    \Tr\left( - \cos^2\alpha_0\,\partial_t
K+i\left[-\frac{\partial_\theta^2}2 +  \cos^2\alpha_0\la
V\ra_{\alpha_0}, K\right]_\omega\right) \left(\omega, s,
\sqrt{2H}, H, t\right) \nonumber \\
 & \qquad \qquad \qquad \qquad\qquad \qquad\qquad \qquad \qquad \qquad \qquad \qquad \qquad \qquad
 {\rho}_{\alpha_0}\left(d\omega,  ds, dE,
dH, dt\right)
 \nonumber  \\
 &= \int    \Tr\left( - \cos^2\alpha_0\,\partial_t K+i\left[-\frac{\partial_\theta^2}2 +  \cos^2\alpha_0  V , K\right]_\omega\right)
\left(\omega, s, \sqrt{2H}, H, t\right)
\nonumber \\
& \qquad \qquad \qquad \qquad\qquad \qquad\qquad \qquad \qquad \qquad \qquad \qquad \qquad \qquad
{\rho}_{\alpha_0}\left(d\omega,  ds, dE, dH, dt\right)
\label{e:moyenn}
  \end{align}
since $\rho_{\alpha_0}$ is carried by $E=\sqrt{2H}$. With a slight
abuse of notation we denoted by $ V = V(s, E, t)$ the operator of
multiplication by $V\circ \Phi (s, \theta, E, -\sin\alpha_0 E, t)$
acting on $L^2(0, 2\pi)$. Note that it does not depend on
$\omega$. It satisfies the condition \eqref{e:tangentialoperator0}
since the function $V\circ\Phi$ satisfies the symmetry condition
(B) (since $V$ is only a function of $z$ in the old coordinates).
In \eqref{e:moyenn} we used the fact that $K(s,
\sqrt{2H}, H, t)$ does not depend on $s$, and the result of
Proposition \ref{p:invop} (i), to replace $\la V\ra_{\alpha_0}$ by
$V$.

Now, by definition of $\rho_{\alpha_0}$, we have
\begin{align*}
 \mathcal{Q}_0 & =\lim \left\la n_h^{\alpha_0}, - \cos^2\alpha_0\,\partial_t K+i\left[-\frac{\partial_\theta^2}2 +  \cos^2\alpha_0  V , K\right]_\omega\right\ra \nonumber\\
& = \lim \bigg\langle \chi_0 \scrV\scrU u_{h}, \nonumber\\
& \quad
 \left(- \cos^2\alpha_0\,\partial_t K+i\left[-\frac{\partial_\theta^2}2 +  \cos^2\alpha_0  V , K\right]\right)\left( s, h D_s, H, h^2D_t \right) \scrV\scrU u_{h}\bigg\rangle_{L^2 (\R_s\times \R_H, L^2( \R_\theta))}
 \end{align*}
Using the fact that $K(s, E, H, t)=b_0(s, \theta, E, D_\theta, H,
t)$ and the commutator calculus rule \eqref{e:weylcomm} for the
standard quantization, we obtain
\begin{align*}
 \mathcal{Q}_0 & =
 \lim_{h\rightarrow0^{+}}
i\left\langle \chi_0 \scrV\scrU u_{h},
 \left( \left[ \cos^2\alpha_0  V , K\right]\right)\left( s, h D_s, H, h^2D_t \right) \scrV\scrU u_{h}\right\rangle_{L^2 (\R_s\times \R_H, L^2( \R_\theta))} \nonumber\\
& \quad +
 \left\la \chi_0 \scrV\scrU u_h,  \Op_h\left((  \eta \partial_\theta- i\frac{\partial_\theta^2}2 -\cos^2\alpha_0\partial_t) b(s, \theta,E, J', \frac{ J'}{h}, H, ht) \right)   \scrV\scrU u_h\right\ra_{L^2(\R_s\times \R_\theta\times \R_H)}
\end{align*}
using the notation $\eta=\frac{J'}h=\frac{J+\sin\alpha_0\sqrt{2H}}{h}$.

We now set
\begin{multline*}
\mathcal{Q}_1 :=\lim_{h\rightarrow0^{+}}    \left\la \chi_0 \scrV\scrU u_h,  \Op_h\left( (  \eta \partial_\theta -\cos^2\alpha_0\partial_t) b(s, \theta,E, J', \frac{ J'}{h}, H, ht) \right) \scrV\scrU u_h\right\ra_{L^2(\R_s\times \R_\theta\times \R_H)}  \\
+i\left\langle \chi_0 \scrV\scrU u_{h},
 \left( \left[ \cos^2\alpha_0  V , K\right]\right)\left( s, h D_s, H, h^2D_t \right) \scrV\scrU u_{h}\right\rangle_{L^2 (\R_s\times \R_H, L^2( \R_\theta))} ,
\end{multline*}
so that we have
\begin{align}
 \mathcal{Q}_0 & = \mathcal{Q}_1 +
 \lim_{h\rightarrow0^{+}}
 \left\la \chi_0 \scrV\scrU u_h,  \Op_h\left(- i\frac{\partial_\theta^2}2  b(s, \theta,E, J', \frac{ J'}{h}, H, ht) \right)   \scrV\scrU u_h\right\ra_{L^2(\R_s\times \R_\theta\times \R_H)}
\label{e:rho}
\end{align}

Let us for the moment focus on the term $\mathcal{Q}_1$, involving only
derivatives of order $1$ of $b$.
As in Remark \ref{r:boundary}, we let ${\mathbf
P}(s, \theta, E, J, t, H)= b(s, \theta, E, J +\sin\alpha_0
\sqrt{2H}, \frac{J +\sin\alpha_0 \sqrt{2H}}h, H, -t)$. Since $b$
is compactly supported in the fifth variable, this is also, modulo
$O(h)$,
$${\mathbf P}(s, \theta, E, J, t, H)= b(s, \theta, E, 0, \frac{J +\sin\alpha_0 \sqrt{2H}}h, H, -t).$$

Still using the notation $\eta=\frac{J'}h=\frac{J+\sin\alpha_0\sqrt{2H}}{h}$, we have
\begin{align*}
\mathcal{Q}_1 &=\lim
 \left\la \scrU u_h,  \Op_h\left( (  \eta \partial_\theta -\cos^2\alpha_0\partial_t) {\mathbf P}(s, \theta, E, J, t, hH) \right) \scrU u_h\right\ra_{L^2(\R_s\times \T_\theta\times \R_t)}
  \\
& \quad \qquad +i\left\langle \chi_0 \scrV\scrU u_{h},
 \left( \left[ \cos^2\alpha_0  V , K\right]\right)\left( s, h D_s, H, h^2D_t \right) \scrV\scrU u_{h}\right\rangle_{L^2 (\R_s\times \R_H, L^2( \R_\theta))} \\
 & = \lim  \left\la u_h,  \cA_H\left( (  \eta \partial_\theta
-\cos^2\alpha_0\partial_t)
{\mathbf P}\right)u_h\right\ra_{L^2(\R^2\times\R^2\times \R_t)}  \\
& \quad \qquad  +i \left\la u_h, \left[ \cos^2\alpha_0  V , \cA_H\left( {\mathbf P}\right)\right]u_h\right\ra_{L^2(\R^2\times\R^2\times \R_t)}  \\
& = \lim  \left\la u_h,  \cA_H\left(
-\frac{\eta\cos\alpha}{E(\alpha_0-\alpha)}   (E\,
\partial_s-h\partial_t) {\mathbf P}\right)u_h\right\ra
+i \left\la u_h, \left[ \cos^2\alpha_0  V , \cA_H\left( {\mathbf
P}\right)\right]u_h\right\ra .
\end{align*}
Using that $b$ (and thus also ${\mathbf P}$) satisfies equation (D), together with \eqref{e:dl}, we obtain
\begin{equation}
\label{e:order1}
 \mathcal{Q}_1 =\lim \frac{\cos^2\alpha_0}h \left\la u_h,  \cA_H\left(   (E\, \partial_s-h\partial_t) {\mathbf P}\right)u_h\right\ra + i\left\la u_h, \left[ \cos^2\alpha_0  V , \cA_H\left( {\mathbf P}\right)\right]u_h\right\ra
 \end{equation}

 Finally, we use again the Schr\"odinger equation \eqref{e:schrodelta} satisfied by $u_h$ extended to $\R^2$, and rewrite the last line as
 \begin{align}\label{e:ouf}
  \mathcal{Q}_1
  & = \lim \frac{\cos^2\alpha_0}h \left\la  u_h, \left[-\frac{ih}2\Delta +ihV
  -h\partial_t,  \cA_H\left(    {\mathbf P}\right)\right]u_h\right\ra \nonumber \\
 &=\lim - \frac{\cos^2\alpha_0}h \left\la  \frac{ih}2 \frac{\partial u_h}{\partial n}\otimes \delta_{\partial \ID},  \cA_H\left(   {\mathbf P}\right)  u_h \right\ra +\lim - \frac{\cos^2\alpha_0}h \left\la u_h,   \cA_H\left(   {\mathbf P}\right) \frac{ih}2 \frac{\partial u_h}{\partial n}\otimes \delta_{\partial \ID}\right\ra .
 \end{align}

Here we need the knowledge of $\cA_H\left( {\mathbf P}\right)$
modulo $O(h^2)$ (because of the factor $\frac{\cos^2\alpha_0}h$
that appears in the previous expression). Our calculations of Proposition~\ref{p:summary} give us the expression
 \begin{multline*}\cA_H\left( {\mathbf P}\right)= {\mathbf A}(r, u, \sqrt{2hD_t}, hD_u, t)+{\mathbf B}(r, u, \sqrt{2hD_t}, hD_u, t)\circ hD_r\\ +ih{\mathbf C}(r, u, \sqrt{2hD_t}, hD_u, t)+ih{\mathbf D}(r, u, \sqrt{2hD_t}, hD_u, t)\circ hD_r
  \end{multline*}
 if $z=(-r\sin u, r\cos u)$ is the decomposition in polar coordinates and ${\mathbf{A, B, C, D}}$ are the functions associated to ${\mathbf P}$ by the formulas of Proposition~\ref{p:summary}.

 The terms ${\mathbf{A, C}}$ give a vanishing contribution in formula \eqref{e:ouf} because they are radial operators and $u_h$ satisfies a Dirichlet boundary condition. The term
 ${\mathbf B}$ gives a vanishing condition if $b$ (and hence ${\mathbf P}$) satisfy the symmetry condition (B): in that case we have ${\mathbf B}(1, u, \sqrt{2hD_t}, hD_u, t)=0$.
 So there just remains to look at the term ${\mathbf D}(1, u, \sqrt{2hD_t}, hD_u, t)$.

Look at formula \eqref{e:D} defining the function ${\mathbf D}$.
Remember that ${\mathbf P}(s, \theta, \sqrt{2H}, J, t, H)$ is
supported where $J+\sqrt{2H}\sin\alpha_0= O(h)$, so that we have
$\partial_s {\mathbf P}=O(h)$; also note that, on the set
$\{J=-\sin\alpha_0 E\}$, the boundary equation $r=1$ amounts to
$s=\pm\cos\alpha_0$, $\cos\theta_1(r, J, E)= \pm\cos\alpha_0$, so
that $s\cos\theta_1(r, J, E)= \cos^2\alpha_0$ in formulas
\eqref{e:D} and the following lines. We see that the function
${\mathbf D}(1, u,\sqrt{2H}, J, t)$ coincides, modulo $O(h)$, with
$\frac{1}{2H\cos^2\alpha_0} {\mathbf P}^\sigma (1, u, \sqrt{2H},
J, t)$, so that
\begin{align*}
\mathcal{Q}_1 & = - \lim  \left\la  \frac{h}2 \frac{\partial u_h}{\partial n}\otimes \delta_{\partial \ID}, {\mathbf D}(1, u, \sqrt{2hD_t}, hD_u, t)\circ hD_r u_h \right\ra \\
 & \quad  + h\lim  \left\la u_h, {\mathbf D}(1, u, \sqrt{2hD_t}, hD_u, t)\circ hD_r \frac{h}2 \frac{\partial u_h}{\partial n}\otimes \delta_{\partial \ID}\right\ra\\
 &  = - \frac{1}{2\cos^2\alpha_0}
 \lim  \left\la  h \frac{\partial u_h}{\partial n} , (E^{-2} \partial_2^2 {\mathbf P}^\sigma)(1, u, \sqrt{2hD_t},  -\sqrt{2hD_t}\alpha_0, ht, hD_t)h\frac{\partial u_h}{\partial n} \right\ra  .
\end{align*}
Hence, we obtain
\begin{align*}
\mathcal{Q}_1 & =
- \lim  \frac{\cos^2\alpha_0}h  ih  \left\la  \frac{h}2 \frac{\partial u_h}{\partial n}\otimes \delta_{\partial \ID}, {\mathbf D}(1, u, \sqrt{2hD_t}, hD_u, t)\circ hD_r u_h \right\ra  \\
& = -\frac{ i}2
 \lim  \left\la  h \frac{\partial u_h}{\partial n} , (E^{-2} \partial_2^2 {\mathbf P})(1, u, \sqrt{2hD_t},  -\sqrt{2hD_t}\alpha_0, ht, hD_t)h\frac{\partial u_h}{\partial n} \right\ra   .
 \end{align*}
Using Remark \ref{r:boundary}, this limit expressed in terms of
boundary data can also be expressed in terms of the interior, and
we see that it equals
 \begin{align*}
\mathcal{Q}_1 &  = \frac{ i}2\lim  \left\la   \scrU u_h,  \Op_h\left( \partial_2^2 {\mathbf P}(s, \theta,E, J', \frac{ J'}{h}, H, ht) \right)  \scrU u_h\right\ra_{L^2(\R_s\times \T_\theta\times \R_t)}\\
& =\frac{ i}2 \lim  \left\la   \scrV\scrU u_h,  \Op_h\left(
\partial_2^2 b(s, \theta,E, J', \frac{ J'}{h}, H, ht) \right)
\scrV\scrU u_h\right\ra_{L^2(\R_s\times \T_\theta\times \R_H)} .
 \end{align*}

Finally coming back to \eqref{e:rho}, this yields $\mathcal{Q}_0 = 0$, that is, identity \eqref{e:propop}. This concludes the proof of Proposition~\ref{p:invop}.
\end{proof}

\section{End of the semiclassical construction: proof of Theorem~\ref{t:precise}\label{s:pause}}
In this section, we first prove Proposition~\ref{p:barmu}, and then conclude the proof of Theorem~\ref{t:precise}.

\subsection{A proof of Proposition~\ref{p:barmu}}
\label{s:proofbarmu} For $a$ a smooth compactly supported function
on $\R^4$, we show that
\begin{equation*}
\lim_{h\To 0} \la u_h, \Op_h(\partial_t a(|\xi|^2,
x\xi_y-y\xi_x, t, hH)u_h\ra_{L^2(\R^2\times \R)} =0  .
\end{equation*}
This limit is the same as
$$\lim_{h\To 0} \la u_h, \Op_h(\partial_t a(2H, x\xi_y-y\xi_x, t, hH)u_h\ra_{L^2(\R^2\times \R)}$$
which is
\begin{equation*}\lim_{h\to 0}\la u_h, [\partial_t, \Op_h( a(2H, x\xi_y-y\xi_x, t, hH)]u_h\ra_{L^2(\R^2\times \R)}=
\lim_{h\to 0}\la u_h, [\partial_t, a(2hD_t, h D_u, t, h^2
D_t)]u_h\ra
\end{equation*}
where $z=(-r\sin u, r\cos u)$ is the decomposition of $z=(x, y)$
into polar coordinates. Because of the equation satisfied by
$u_h$, this is also (with $\Delta_D$ the Dirichlet laplacian)
\begin{equation*}
\lim_{h\To 0}\left\la u_h, \left[-i\frac{\Delta_D}{2}+iV, a(2hD_t,
h D_u, t, h^2 D_t)\right]u_h\right\ra
\end{equation*}
{Note that $a(2hD_t, h D_u, t, h^2 D_t)$ actually defines an operator on $L^2(\ID)$ as it is tangential to $\d \ID$.}
This limit vanishes, because $\Delta_D$ commutes with
$a(2hD_t, h D_u, t, h^2 D_t)$ and because
$$ \left[V, a(2hD_t, h D_u, t, h^2 D_t)\right]=O(h).$$
This concludes the proof of Proposition~\ref{p:barmu}.

\subsection{End of the proof of Theorem~\ref{t:precise}}

We are now in a position to prove Theorem \ref{t:precise}.

The measure $\ell_{\alpha_0}$ and the function $\sigma_{\alpha_0}$
of Theorem \ref{t:precise} (iii) are the ones appearing in Remark
\ref{radnikoPG}. The object called $\nu_{Leb}$ in Theorem
\ref{t:precise} (ii) is defined as
$$\nu_{Leb}=\mu_{sc}\rceil_{\alpha\not\in \pi\IQ}+\sum_{\alpha_0\in \pi\IQ} m^{\alpha_0}$$
where $m^{\alpha_0}$ was defined in Remark \ref{r:mms}. For
$\alpha\not\in \pi \IQ$, we must have
$$\mu_{sc}\rceil_{\cI_\alpha}(t)=\int c_1(t, E, J) \lambda_{E, J} \, d\nu_1(E, J)$$
for some nonnegative measure $\nu_1$ (carried by $\{J=-\sin\alpha
E\}$) and some measurable function $c_1(t, E, J)$. But, because
the image of $\mu_{sc}$ under the map $M: (z, \xi)\mapsto (E, J)$
does not depend on $t$ (see below), the function $c_1(t, E, J)$
actually does not depend on $t$.

The two invariance properties Proposition \ref{Lemma Inv}
(invariance w.r.t. $s$) and Theorem \ref{Thm Properties} (i)
(invariance w.r.t. $\theta$) also imply that $m^{\alpha_0}$ is of
the form
$$m^{\alpha_0}(t)=\int c_2(t, E, J) \lambda_{E, J} \, d\nu_2(E, J)$$
for some nonnegative measure $\nu_2$ (carried by $\{J=-\sin\alpha_0
E\}$).

\bigskip
 We now prove that the function $c_2(t, E, J)$ actually does not depend on $t$.
 For this, we remark that the same proof as that of Proposition~\ref{p:barmu} above applies if we replace
$$\Op_h\left(\partial_t a(|\xi|^2, x\xi_y-y\xi_x, t, hH)\right)$$ in the
first line by
$$\scrU^* \Op_h\left(\partial_t a(E^2, J, t, hH)(1-\chi)\left(\frac{J+\sin\alpha_0\sqrt{2H}}{hR}\right)  \right)\scrU$$
in the limits $h\To 0$ followed by $R\To +\infty$.

Using the notation of Remark \ref{r:mms}, this shows that the
image of $m^{\alpha_0}$ under the map $M$ is independent of $t$.
Since we already know that $m^{\alpha_0}(t)$ is of the form $\int
c_2(t, E, J) \lambda_{E, J} d\nu_2(E, J)$ for some nonnegative
measure $\nu_2$, we conclude that $c_2(t, E, J)$ actually does not
depend on $t$.

The proof of Theorem \ref{t:precise} is now complete.

\section{The microlocal construction: sketch of the proof of Theorem \ref{t:preciseml}\label{s:H0}}
Herein we use the definitions and notation introduced in Sections~\ref{s:versus} and~\ref{s:mlst}.

Let $(u_n^0)$ be a sequence of initial data, normalized in $L^2$,
and as above denote $u_n(z, t)=U_V(t)u_n^0(z)$ or in short
$u_n=U_V u_n^0$.

\subsection{Structure of $\mu_c$\label{s:muc}}

Let $a\in\cS_0$. Recall that $\la\mu_c, a\ra$ is defined as the
limit (after extraction of subsequences) as $n\To +\infty$
followed by $R\To +\infty$ of
\begin{equation}\label{e:wc}\left\langle W_{c, n, R}   ,a\right\rangle :=
 \left\la u_n,  \Op_1\left( \chi\left(  \frac{|\xi|^2+ |H|}{R^2}\right)  a( z, \xi, t, H  )\right)u_n\right\ra_{L^2(\R^2 \times \IR)}  .
\end{equation}
On the support of  $\chi\left(  \frac{|\xi|^2+ |H|}{R^2}\right)
a( z, \xi, t, H  )$ note that $|H|\leq K R^2$ if $\chi$ is
supported in $[-K, K]$.

Let $g$ be a smooth compactly supported function on $\R$, taking
the value $1$ on $[-K, K]$. For $R>0$ define the operator $P_R :
L^2(\ID)\mapsto L^2(\ID)$ by
$$P_R u= g\left(\frac{D_t}{R^2}\right) U_V(t) u\rceil_{t=0}.$$
By the results of Appendix~\ref{s:osc} (in particular Remark
\ref{r:osc}), we know that $P_R$ has the following properties:
\begin{itemize}
\item For fixed $R$, $P_R$ is compact, and we have $\norm{\nabla
P_R u}_{L^2(\ID)}\leq C R\norm{u}_{L^2(\ID)}$ for some constant
$C$ (that depends on the function $g$, but not on $R$). \item For
any $u$, we have $P_R u\To u$ in $L^2(\ID)$ as $R\To +\infty$.
\item $\la\mu_c, a\ra$ is the limit as $n\To +\infty$ followed by
$R\To +\infty$ of
\begin{equation} \label{e:wcp}
 \left\la u_n,  \Op_1\left( \chi\left(  \frac{|\xi|^2+ |H|}{R^2}\right)  a( z, \xi, t, H  )\right)U_VP_R u_n^0\right\ra_{L^2(\R^2 \times \IR)}  .
\end{equation}
\end{itemize}

Let $K_t(z, z')$ be the kernel of $U_V(t)$. Let $B_R$ be the
bounded operator on $L^2(\ID)$ with kernel
{
\begin{multline}B_R(z_1, z_2)=\frac{1}{(2\pi)^{3}}\int \overline{K_t(z, z_1) }\mathds{1}_{\ID}(z)a(z, \xi, t, H)\chi\left(  \frac{|\xi|^2+ |H|}{R^2}\right) \\
e^{i\xi(z-z')}e^{iH(t-t')} \mathds{1}_{\ID}(z')K_{t'}(z', z_2) dz dz' dt dt' dH d\xi\label{e:BR}
\end{multline}
so that \eqref{e:wc} is $\la u_n^0, B_R u_n^0\ra_{L^2(\ID)}$ and
\eqref{e:wcp} is $\la u_n^0, B_R  P_Ru_n^0\ra_{L^2(\ID)}$.
}

Call $\rho_0$ a weak-$\ast$ limit of the sequence of trace class
operators $|u_n^0\ra\la u_n^0|$ on $L^2(\ID)$. Then for fixed $R$
$\left\langle W_{c, n, R}   ,a\right\rangle $ converges to
$\Tr_{L^2(\ID)}(B_R P_R\rho_0)$. Letting now $R\To +\infty$ we
find the expression
\begin{equation}\la\mu_c, a\ra=\Tr_{L^2(\ID)}(B_\infty \rho_0)\label{e:Binf}\end{equation}
where $B_\infty$ is defined as in \eqref{e:BR} with $\chi\left(
\frac{|\xi|^2+ |H|}{\infty}\right)=1$. The expression of
$B_\infty$ is simpler when $a$ does not depend on $H$, in this
case we have
{
$$B_\infty =\int  U_V(t)^{*} \mathds{1}_{\ID}a(z, D_z, t) \mathds{1}_{\ID}U_V(t)dt.$$

In particular if $a$ does not depend on $H$ we have
$$\la\mu_c, a\ra=\Tr_{L^2(\ID)}\left(\int  U_V(t)^{*} \mathds{1}_{\ID}a(z, D_z, t) \mathds{1}_{\ID} U_V(t) \rho_0 dt\right)$$
and more generally we write (in a somewhat fuzzy notation)
$$\la\mu_c, a\ra=\Tr_{L^2(\ID)}\left(\int  U_V(t)^{*} \mathds{1}_{\ID}a(x, D_z, t, D_t)\mathds{1}_{\ID} U_V(t)\rho_0 dt \right)$$ to mean the well-defined expression \eqref{e:Binf}.
}
\subsection{Structure of $\mu^\infty$ and proof of Theorem \ref{t:preciseml}}

Let $\eta=\frac{\xi}{\sqrt{2H}}$. On the support of $\mu^\infty$,
$\eta$ has norm $1$. To any pair $(z, \eta)\in \ID\times \IS^1$ we
now associate $j=x\eta_y- y\eta_x$ and $\alpha=-\arcsin j$ which
is the angle that the billiard ray issued from $(z, \eta)$ makes
with the inner normal when it bounces on the boundary of the disk.
Exactly as in
Lemma \ref{Lemma decomposition} we decompose $\mu^\infty$ as a sum of nonnegative measures:%
\begin{equation}
\mu^\infty=\mu^\infty\rceil_{ \alpha\not\in \pi\IQ}+\sum_{r\in
\IQ\cap[-1/2, 1/2]}\mu^\infty\rceil_{ \alpha=r\pi}.
\label{dec2}%
\end{equation}

The invariance \eqref{e:transport2} implies that
$\mu^\infty\rceil_{ \alpha \not\in \pi\IQ}$ is of the form
$\int_{E>0, |J|\leq E, \alpha \not\in \pi\IQ} \lambda_{E, J}
d\bar\mu^\infty(E, J)$ The fact that $\bar\mu^\infty$ does not
depend on $t$ is the microlocal version of Proposition \ref{p:barmu}.

 We now fix $r_0\in \IQ\cap(-1/2, 1/2)$, write $\alpha_0 = r_0\pi$ and wish to study $\mu^\infty\rceil_{ \alpha=\alpha_0}$. We define
\begin{align*}
&\left\langle \mu_{\alpha_0}^{\infty}   ,a\right\rangle \\
&:=\lim_{R'}\lim_{R\rightarrow\infty} \lim_{n\rightarrow +\infty} \\
& \left\la u_n,  \Op_1\left( \left(  1-\chi\left(  \frac{|\xi|^2+ H}{R^2 }\right)\right)  a(z, \xi, t, H )\chi\left(\frac{J+\sin\alpha_0\sqrt{2H}}{R'}\right)\right) u_n\right\ra_{L^2(\R^2 \times \IR)}
\\
&= \lim_{R'}\lim_R\lim_n  \\
&\left\la \scrU  u_n,  \Op_1\left(  \left(  1-\chi\left(  \frac{H }{R^2}\right)\right)a\circ\Phi(s, \theta, \sqrt{2H}, J, t, H )\chi\left(\frac{J+\sin\alpha_0\sqrt{2H}}{R'}\right)\right) \scrU  u_n\right\ra_{L^2(\R\times \IT\times \IR)}
\end{align*}
and%
\begin{align*}
&\left\langle \mu^{\infty, \alpha_0}
,a\right\rangle\\
&  :=
\lim_{R'}\lim_{R\rightarrow\infty} \lim_{n\rightarrow +\infty} \\
& \left\la u_n,  \Op_1\left( \left(
1-\chi\left(  \frac{|\xi|^2+ H}{R^2 }\right)\right)  a(z, \xi, t,
H
)\left(1-\chi\left(\frac{J+\sin\alpha_0\sqrt{2H}}{R'}\right)\right)\right)
u_n\right\ra_{L^2(\R^2 \times \IR)}
  \\
  &=\lim_{R'}\lim_R\lim_n   \\
  & \left\la \scrU  u_n,  \Op_1\left(  \left(  1-\chi\left(  \frac{H }{R^2}\right)\right)a\circ\Phi(s, \theta, \sqrt{2H}, J, t, H )\left(1-\chi\left(\frac{J+\sin\alpha_0\sqrt{2H}}{R'}\right)\right)\right) \scrU  u_n\right\ra .
\end{align*}

The following theorem is proven in essentially the same way as in
the semiclassical case:
\begin{theorem}
\label{th:structure2micromicro} (i)$\mu^{\infty, \alpha_0} $ is
invariant under rotations (that is, under the flow of $P_1$). It
is a multiple of the Lebesgue measure
 $\lambda_{(1, -\sin \alpha_0)}$. That is, $\mu^{\infty, \alpha_0}= c(\alpha_0) \lambda_{(1, -\sin \alpha_0)}$ with $c(\alpha_0)\geq 0$ independent of $t$.

(ii) for every $\alpha_0\in \pi\IQ\cap(-\pi/2, \pi/2)$, we can
build from the sequence of initial conditions $(u_n)$ a
nonnegative measure $\sigma_{\alpha_0}(d\omega,  ds, dE,  dH, dt)$
(carried by $\{H=E^2/2\}$) on $\R/2\pi\Z\times \R_s\times
\IS^2_{E, H}\times\R_t$, taking values in the trace-class
operators on $L ^{2}(0, 2\pi)$, so that $\mu_{\alpha_0}^\infty$ is
the measure carried by the set $\{j=-\sin\alpha_0
\}\cap\{H=E^2/2\}$ such that
\begin{multline*}
\int a_{{\rm{hom}}}(z, \xi, t)\mu_{\alpha_0}^\infty(dz, d\xi, dt,
dH) \\
=\Tr_{L ^{2}(0, 2\pi)}\left( \int  m_{ a_{{\rm{hom}}} \circ
\Phi }(s, \cdot, 1, -\sin\alpha_0 , t) \,\sigma_{\alpha_0 }
(d\omega, d s, dE, dH, dt)\,\right).
\end{multline*}

 If in addition $a$ is symmetric w.r.t. the boundary, we have
\begin{multline*}
\int a_{{\rm{hom}}}(z, \xi)\mu_{\alpha_0}^\infty(dz, d\xi, t, dH)\\
=\Tr_{L ^{2}(0, 2\pi)}\left( \int {U_{{\alpha_0,
\omega}}(t)}^{\ast} m_{\left\langle
a_{{\rm{hom}}}\right\rangle_{\alpha_0}\circ \Phi }(\cdot, 1,
-\sin\alpha_0 )\,U_{ {\alpha_0, \omega}}(t)\,\sigma_{\alpha_0 }
(d\omega,  d s, dE, dH, 0)\,\right)  .
\end{multline*}

\end{theorem}

The decomposition forlmula of Theorem \ref{t:preciseml}~(i) now holds with
\begin{itemize}
\item the distribution $\mu_c\in \cS_0'$ described in Section~\ref{s:muc};
\item the measure $\mu_{Leb}$ given by
$$\mu_{Leb}=\mu^\infty\rceil_{ \alpha\not\in \pi\IQ} + \sum_{\alpha_0\in \pi\IQ\cap(-\pi/2, \pi/2)}\mu^{\infty, \alpha_0};$$
\item for $\alpha_0\in  \pi\IQ\cap(-\pi/2, \pi/2)$ the measure $\mu_{\alpha_0}$ given by
 $$\mu_{\alpha_0}(t)=\mu_{\alpha_0}^\infty(t);$$
\item for $\alpha_0=\pm \pi/2$ the measure $\mu_{\alpha_0}$ given by
$$\mu_{\alpha_0}=\mu^\infty\rceil_{ \alpha=\alpha_0}.$$
\end{itemize}
Theorem~\ref{th:structure2micromicro} then implies Theorem \ref{t:preciseml}.

\section{Proof of Theorems~\ref{t:obs-i} and~\ref{t:obs-b}: Observability inequalities\label{s:obs}}
In this section, we prove Theorems~\ref{t:obs-i} and~\ref{t:obs-b}
using the microlocal version of our results. We could have chosen
to do it with semiclassical measures as well. However, since there
is no natural frequency-scale, it would have required to perform a
dyadic decomposition in frequency (see for instance~\cite{Leb:92,
AnantharamanMaciaTore}). Note that the idea of proving
observability inequalities using microlocal defect measures is due
to Lebeau~\cite{Leb:96}.

\subsection{Unique continuation for microlocal measures}
\label{sec:mdm=0}
The goal of this section is to prove a unique continuation result for microlocal measures $\mu_{ml}$ associated to solutions of the Schr\"odinger equation~\eqref{e:S}. According to Theorem~\ref{t:preciseml}, such a measure decomposes as
$$\mu_{ml}=\mu^\infty+\mu_c ,$$
that we shall study independently.

In order to state the result for $\mu^\infty$, we introduce the following notation.
For $z \in \partial \ID$, we define
$$
S^+_z =\{\xi \in \R^2, \xi\cdot z > 0\}, \qquad \overline{S}^+_z =\{\xi \in \R^2, \xi\cdot z \geq 0\}.
$$
The set $S^+$ defined in Section~\ref{s:billiard} is $S^+ =\bigcup_{z \in \partial \ID} S^+_z$ and
$$
 \bigcup_{z \in \partial \ID} \overline{S}^+_z = \left\{ \Phi\left(  \left(1-(J/E)^2\right)^\frac12,\theta , E, J \right),
E>0, |J|\leq E , \theta \in \IR/2\pi\Z \right\} .
$$

The following two lemmas are respectively useful for the proof of internal and boundary observability.
\begin{lemma}
\label{l:muinfty=0} Fix $T>0$.
Take $b \in \mathcal{S}^0$ independent of $(t,H)$ and assume that
\begin{equation}
\label{hypS0+}
\text{ there exists } z_0 \in \partial \ID \text{ such that } b>0 \text{ in a neighbourhood of } \overline{S}^+_{z_0}.
\end{equation}

Then $\int_0^T \int_{\R^2 \times \IS^2_{H, \xi}}b^2_{\rm{hom}}(z,
\xi )\mu^\infty(dz, dH, d\xi,t)dt = 0 $ implies $\mu^\infty =0$ on
$\R_t \times \R^2_z \times \IS^2_{H, \xi}$.
\end{lemma}

\begin{lemma}
\label{l:muinfty=0bo}
Take any nonempty set $\Gamma \subset \d \ID$ and $T>0$.
Then $\mu_{ml}^\d = 0$ on $T^* ((0,T)\times \Gamma)$
implies $\mu^\infty =0$ on $\R_t \times \R^2_z \times \IS^2_{H, \xi}$.
\end{lemma}

The proof of these lemmas relies on the properties of $\mu^\infty$ together with a unique continuation result for the one dimensional Schr\"odinger flows $U_{\alpha_0 , \omega}(t)$ on $L^2(0,2\pi)$ from any nonempty open set $(0,T) \times \Omega$, where $\Omega \subset (0,2\pi)$. Such unique continuation property holds as soon as $\la V \ra_{\alpha_0} \in L^{\infty}((0,T)\times (0,2\pi))$, for instance as a consequence of \cite[Appendix~B]{Laurent} (see also the references therein).

Concerning $\mu_c$ we have the following result.

\begin{lemma}
\label{l:muc=0} Let $\Omega \subset \ID$ be a nonempty open set.
Assume that the unique continuation property~\eqref{UCP} holds.
Then we have
$$
\la \mu_c , \mathds{1}_{(0,T)\times \Omega} \ra =0\Longrightarrow \mu_c = 0.
$$
\end{lemma}

The unique continuation property \eqref{UCP} is for instance known to hold (in any time $T>0$ and for any nonempty open set $\Omega$) if $V$ is analytic in $(t,z)$ as a consequence of the Holmgren theorem (as stated by H\" ormander~\cite[Theorem~5.3.1]{Hormander:LPDO}). If $V = V(z)$ is smooth and does not depend on $t$, it is proved in the next section. Note that this last result can be extended to the case where $V$ is continuous outside a set of zero measure zero, see~\cite{AnantharamanMaciaTore}.

\begin{remark}
Note that the analogues of the unique continuation results of Lemmas~\ref{l:muinfty=0},~\ref{l:muinfty=0bo} and~\ref{l:muc=0} also hold for semiclassical measures. We chose not to state them here for the sake of brevity.
\end{remark}

 \begin{proof}[Proof of Lemma~\ref{l:muinfty=0}]
We decompose $\mu^\infty$ as in Theorem~\ref{t:preciseml}
$$\mu^\infty(t,\cdot)=\mu_{Leb}+\sum_{\alpha_0\in \pi\IQ\cap[-\pi/2, \pi/2]}\mu_{ \alpha_0}(t,\cdot).$$
As every term in this sum is a non-negative measure, the assumption on $\mu^\infty$ implies
\begin{align}
&\int_0^T \int_{\R^2 \times \IS^2}b^2_{\rm{hom}}(z, \xi )\mu_{Leb}(dz, dH, d\xi,t)dt = 0 , \label{muleb0}\\
&\int_0^T \int_{\R^2 \times \IS^2}b^2_{\rm{hom}}(z, \xi )\mu_{ \alpha_0}(dz, dH, d\xi,t)dt = 0,\label{mualph0}
\end{align}
for all $\alpha_0\in \pi\IQ\cap[-\pi/2, \pi/2]$.

Still according to Theorem~\ref{t:preciseml}, $\mu_{Leb}$ is of the form $\int_{E>0, |J|\leq E} \lambda_{E, J}
d\mu'(E, J)$ for some nonnegative measure $\mu'$ on $\R P^1$. Together with~\eqref{muleb0}, this reads
$$
0 = \int_{E>0, |J|\leq E} \int_{T(E,J)}b^2_{\rm{hom}} \circ \Phi (s, \theta, E, J) \lambda_{E, J}(ds, d\theta)
\mu'(dE, dJ) .
$$
Recall (see Section~\ref{s:billiard}) that $\lambda_{E, J}(ds, d\theta)  = c(E,J)ds  d\theta$ where $c(E,J) = \left(\int_{T(E,J)}ds d\theta \right)^{-1}>0$, so that we have
$$
0 = \int_{E>0, |J|\leq E} \left(\int_{T(E,J)}b^2_{\rm{hom}} \circ \Phi (s, \theta, E, J) ds d\theta\right)
c(E,J)
\mu'(dE, dJ) .
$$
Now, for any $(E,J)$ such that $E>0, |J|\leq E$, there exists $\theta \in \IS^1$ (depending only on $J/E$), such that
$$
\Phi\left( \left(1-(J/E)^2\right)^\frac12,\theta , E, J \right) \in \overline{S}_{z_0}^+ .
$$
Assumption~\eqref{hypS0+} then implies that $\int_{T(E,J)}b^2_{\rm{hom}} \circ \Phi (s, \theta, E, J) ds d\theta >0$ for any $(E,J)$. As a consequence, $\mu_{Leb}$ vanishes identically.

\medskip
Let us now consider $\alpha_0 = \pm \pi/2$. The rotation invariance given by Theorem~\ref{t:preciseml} together with Assumption~\eqref{hypS0+} imply that $\mu_{\pm\pi/2}$ vanish.

\medskip
Let us now consider $\alpha_0\in \pi\IQ\cap(-\pi/2, \pi/2)$. The measure $\mu_{\alpha_0}$ is supported by $\mathcal{I}_{\alpha_0}$ and invariant by the billiard flow, so that
$$
\int_0^T \int_{\R^2 \times \IS^2}b^2_{\rm{hom}}\,\mu_{ \alpha_0}(dz, dH, d\xi,t)dt =
\int_0^T \int_{\R^2 \times \IS^2}\la b^2_{\rm{hom}}\ra_{\alpha_0}\, \mu_{ \alpha_0}(dz, dH, d\xi,t)dt
$$
Using Theorem~\ref{t:preciseml} with~\eqref{mualph0}, we obtain
\begin{align*}
0 = \int \Tr_{L ^{2}(0, 2\pi)}\left( B_{\alpha_0}\,\sigma_{\alpha_0
}\,\right) d \ell_{\alpha_0} dt
, \quad \text{with }B_{\alpha_0} := m_{ \la b^2_{\rm{hom}}\ra_{\alpha_0}}^{\alpha_0} .
\end{align*}
According to Corollary~\ref{corsigmaell}, this yields
\begin{align*}
0= \int \Tr_{L ^{2}(0, 2\pi)}\left( B_{\alpha_0}U_{\alpha_0 , \omega}(t) \sigma_{\alpha_0}(\omega, E,H , 0)U_{\alpha_0 , \omega}^*(t)\right) d \ell_{\alpha_0} dt .
\end{align*}
Since the integrand is non-negative, we have for $\ell_{\alpha_0}$-almost every $(\omega, E,H)$,
\begin{align}
\label{inttrace}
0 = \int_0^T \Tr_{L ^{2}(0, 2\pi)}\left( B_{\alpha_0}U_{\alpha_0 , \omega}(t) \sigma_{\alpha_0}(\omega, E,H , 0)U_{\alpha_0 , \omega}^*(t)\right) dt .
\end{align}
For $\ell_{\alpha_0}$-almost every $(\omega, E,H)$, $\sigma_{\alpha_0}(\omega, E,H , 0)$ is a non-negative trace-class operator. We can decompose it as a sum of of orthogonal projectors on its eigenfunctions:
$$
\sigma_{\alpha_0}(\cdot , 0) = \sum_{k \in \N} \lambda_k |\varphi_k \ra \la \varphi_k| , \quad \text{with} \quad \lambda_k \geq 0, \quad \sum_{k \in \N} \lambda_k = 1 , \quad \la \varphi_k| \varphi_j \ra_{L^2(0,2\pi)} = \delta_{kj} .
$$
Note that $\lambda_k, \varphi_k$ depend on $(\omega, E,H)$. Now Equation~\eqref{inttrace} is equivalent to
having, for all $k \in \N$, such that $\lambda_k>0$,
\begin{align}
\label{eqbeforUC}
0= \int_0^T \int_0^{2\pi}  \la b^2_{\rm{hom}}\ra_{\alpha_0} \circ \Phi(s, \theta,  E, - E \sin\alpha_0) \left|U_{\alpha_0 , \omega}(t) √ä\varphi_k\right|^2(\theta)d\theta dt .
\end{align}
As above, there exists $\theta \in \IS^1$ depending only on $\alpha_0$, such that
$$
\Phi\left( \cos\alpha_0,\theta , E, - E\sin\alpha_0 \right) \in \overline{S}_{z_0}^+ .
$$
Hence, $\la b^2_{\rm{hom}}\ra_{\alpha_0}>0$ in a neighborhood of
this $\theta$. Then \eqref{eqbeforUC} implies that $U_{\alpha_0 ,
\omega}(t)\varphi_k$ vanishes in a nonempty open subset of
$(0,T)\times (0, 2\pi)$. One dimensional unique continuation (see
e.g. \cite[Appendix~B]{Laurent} and the references therein) then
implies that $\varphi_k=0$. Therefore, $\sigma_{\alpha_0}(\omega,
E,H , 0)$ vanishes $\ell_{\alpha_0}$-almost everywhere, which
yields $\mu_{\alpha_0}=0$.

This finally proves that $\mu^\infty = 0$ and concludes the proof of the lemma.

\end{proof}

\begin{proof}[Proof of Lemma~\ref{l:muinfty=0bo}]
Let us fix $z_0 \in \Gamma$ and prove that $\mu^\infty$ vanishes
in a neighborhood of $\overline{S}_{z_0}^+$. The result shall then
follow from Lemma~\ref{l:muinfty=0}. First, according
to~\eqref{e:mudmuint}, the assumption implies that
$\mu^\infty\rceil_{T^*((0,T) \times\Gamma)}$ vanishes. Second, as
a consequence of the assumption together with~\eqref{e:mudmuS}, we
have $\mu^S_{ml}\rceil_{(z,\xi) \in S^+, z \in \Gamma} = 0$.
Coming back to the definition of the measure $\mu^S_{ml}$ in
Section~\ref{s:projbou}, this implies that $\mu^\infty$ vanishes
on all trajectories of the billiard flow touching the boundary on
$\{(z,\xi) \in S^+, z \in \Gamma\}$. In particular, this yields
$\mu^\infty\rceil_{\overline{S_{z_0}^+}}$ and the result follows
from Lemma~\ref{l:muinfty=0}.
\end{proof}

\begin{proof}[Proof of Lemma~\ref{l:muc=0}]
 Theorem~\ref{t:preciseml} (and Theorem~\ref{t:muc}) together with $\la \mu_c , \mathds{1}_{(0,T)\times \Omega} \ra =0$ imply that
\begin{align*}
0
=  \Tr_{L ^{2}(\ID)} \left(
\int_0^T U_{V}(t)^* m_{\mathds{1}_{\Omega}}U_{V}(t)\rho_0 dt\right) ,
\end{align*}
where $m_{\mathds{1}_{\Omega}}$ is the multiplication operator in $L^2(\ID)$ by the function $\mathds{1}_{\Omega}$.
As in the proof of Lemma~\ref{l:muinfty=0}, this implies that for any eigenfunction $\varphi$ of $\rho_0$, we have
\begin{align*}
0 = \int_0^T \int_{\Omega} \left|U_{V}(t) \varphi \right|^2(z)dz dt .
\end{align*}
The unique continuation property~\eqref{UCP} then implies $\varphi =0$. This proves that $\rho_0=0$ which concludes the proof of the lemma.
\end{proof}

\subsection{Interior observability inequality: proof of Theorem~\ref{t:obs-i}}
\subsubsection{Unique continuation implies observability}
\label{sec:UCPO}
In this section, we prove the observability inequality~\eqref{e:oi} assuming that~\eqref{UCP} holds.
Instead of proving~\eqref{e:oi} for any open set $\Omega \subset \ID$ containing a neighbourhood in $\ID$ of a point of $\partial \ID$, we prove the equivalent statement: for any function $b \in C^0(\R^2)$ (also considered as a function in $C^0(\overline{\ID})$) which is positive on a nonempty open subset of $\partial \ID$, for any $T>0$, there exists $C>0$ such that the following inequality holds:
\begin{equation}
\left\Vert u^{0}\right\Vert _{L^{2}\left(  \ID\right)  }^{2}\leq
C\int_{0}^{T}\left\Vert b(z) U_{V}(t)  u^{0}\right\Vert _{L^{2}\left(
\ID\right)  }^{2}dt .
\label{e:oi:b}
\end{equation}
Note that under these conditions on $\Omega$ and $b$, inequalities \eqref{e:oi:relax} and \eqref{e:oi:b} are equivalent.

We proceed by contradiction and suppose that the observability
inequality~\eqref{e:oi:b} is not
satisfied. Thus, there exists a sequence $(u_n^0)_{n \in \N}$ in $L^2(\ID)$ such that
\begin{align}
  & \| u^{0}_n \| _{L^2( \ID) } = 1  , \label{eq: energy = 1} \\
  & \int_{0}^{T}\left\Vert b(z)  U_{V}(t)  u^0_n \right\Vert _{L^{2}\left(
\ID\right)  }^{2}dt \to 0 ,
 \quad n \to \infty ,\label{eq: converge zero on omega} 
\end{align}
We write $u_n(t) = U_{V}(t)  u^0_n$ the associated solution of \eqref{e:S}-\eqref{e:Dirichlet}. As in Section~\ref{sec:intro}, we extend $u_n$ to $\R^2$ by zero outside $\ID$ (and still use the notation $u_n$ for its extension).

After having extracted a subsequence, we associate to $(u_n)$ a microlocal measure
$$
\mu_{ml} = \mu^{\infty} + \mu_c
$$ as in  Theorem~\ref{t:preciseml}.
Equation~\eqref{eq: converge zero on omega} implies that
$$
\int_0^T \int_{\R^2 \times \IS^2}b^2(z)\mu^\infty(dz, dH, d\xi,t)dt = 0 , \qquad
\la \mu_c ,\mathds{1}_{(0,T)} \otimes b^2 \ra= 0 .
$$

Lemmas~\ref{l:muinfty=0} and~\ref{l:muc=0} imply that $\mu^\infty = 0$ and $\mu_c = 0$ respectively. However, equation~\eqref{eq: energy = 1} implies that
$$
\la \mu_{ml}, \mathds{1}_{(0,T)} \otimes 1 \ra= T.
$$

This yields a contradiction and concludes the proof. Note that \eqref{UCP} has only be used to apply Lemma~\ref{l:muc=0} in order to get rid of the term $\mu_c$.

\subsubsection{Observability for time independent potentials}
\label{s:obstimeindep}
 The structure of the proof in this setting is classical~\cite{BLR:92,Leb:92}. In a first step, we prove the following weakened observability inequality:
\begin{equation}
\left\Vert u^0 \right\Vert _{L^{2}\left(  \ID\right)  }^{2}\leq
C\int_{0}^{T}\left\Vert U_{V}(t)  u^0 \right\Vert _{L^{2}\left(
\Omega\right)  }^{2}dt + C \left\Vert u^0 \right\Vert _{H^{-1}\left(  \ID\right)  }^{2} , \label{e:oi:relax}%
\end{equation}
In a second step, we conclude the proof of Theorem \ref{t:obs-i} using a unique continuation property for eigenfunctions of the elliptic operator $-\Delta_D + V$.

The first step is similar to Section~\ref{sec:UCPO}. We consider a sequence of initial data $(u_n^0)$ contradicting~\eqref{e:oi:relax}. It satisfies~\eqref{eq: energy = 1}, \eqref{eq: converge zero on omega}, together with
\begin{align}
  \| u^0_n \| _{H^{-1}( \ID) } \to 0 ,
  \quad n \to \infty . \label{eq: norm -1 to zero}
\end{align}
As before, we consider the associated microlocal measure $\mu_{ml} = \mu^{\infty} + \mu_c$. Note now that~\eqref{eq: norm -1 to zero} implies that $\mu_c =0$. The rest of the proof is completely similar.

\medskip
We now prove that \eqref{e:oi:relax} implies the observability inequality~\eqref{e:oi}: this step is by now classical~\cite{BLR:92,Leb:92} but we include it for the sake of completeness. We proceed again by contradiction and suppose that the inequality
\begin{equation}
\label{eq:ioH-1}
\| u^0 \| _{H^{-1}( \ID) } \leq  C \int_{0}^{T}\left\Vert b(z)  U_{V}(t)  u^0 \right\Vert _{L^{2}\left(
\Omega \right)  }^{2}dt
\end{equation}
is not satisfied. Then, there exists a sequence $(u^0_n)_{n \in \N}$ in $L^2(\ID)$ such that
\begin{equation}
\label{equnicite}
   \| u^0_n \| _{H^{-1}( \ID) } = 1  , \qquad
  \int_{0}^{T}\left\Vert b(z)  U_{V}(t)  u^0_n \right\Vert _{L^{2}\left(
\ID\right)  }^{2}dt \to 0 ,
  \quad n \to \infty .
\end{equation}
Inequality~\eqref{e:oi:b} implies that $u^0_n$ is bounded in $L^2(\ID)$, so that, after having extracted a subsequence, we have $u^n_0 \rightharpoonup u^0$ in $L^2(\ID)$ and $u^0_n \to u^0$ in $H^{-1}(\ID)$. We deduce from~\eqref{equnicite} that
$$
   \| u^0 \| _{H^{-1}( \ID) } = 1  , \qquad
   U_{V}(t)  u^0 = 0 \text{ on }\{b^2>0\} \text{ for all } t \in (0,T) .
$$
The weak limit $u^0$ belongs to the set
$$
\mathcal{N} = \{f \in L^2(\ID) ,  U_{V}(t)  f = 0 \text{ on }\{b^2>0\} \text{ for all } t \in (0,T)\}.
$$
Then, by linearity, $\mathcal{N}$ is a closed vector subspace of $L^2(\ID)$. Inequality~\eqref{e:oi:relax} proves that $\mathcal{N}$ is finite dimensional and the time independence of $V$ implies that it is a subspace of $H^2(\ID)\cap H^1_0(\ID)$, stable by the action of the operator $-\Delta_D + V$. If not reduced to $\{0\}$, the space $\mathcal{N}$ hence contains an eigenfunction of $-\Delta_D + V$, vanishing on $\{b^2>0\}$. A classical uniqueness result for elliptic operators then implies that this does not occur. This yields $\mathcal{N}= \{0\}$ and thus $u^0=0$, which contradicts $\| u^0 \| _{H^{-1}( \ID) } = 1$.

\subsection{Boundary observability inequality: proof of Theorem~\ref{t:obs-b}}

We proceed as in the previous section: in a first step, we prove the following weakened observability inequality:
\begin{lemma}
\label{l:obs-brelax}
For all $T>0$, there exists $C>0$ such that for all $u^0 \in H^1_0(\ID)$, we have
\begin{equation}
\left\Vert u^0 \right\Vert _{H^{1}\left(  \ID\right)  }^{2}\leq
C\int_{0}^{T}\left\Vert \d_n( U_{V}(t)  u^0 ) \right\Vert _{L^{2}\left(
\Gamma\right)  }^{2}dt + C \left\Vert u^0 \right\Vert _{L^2\left(  \ID\right)  }^{2} , \label{e:oi:relaxb}%
\end{equation}
\end{lemma}

With this lemma, we now conclude the proof of the observability inequality~\eqref{e:ob}. We proceed by contradiction and suppose that the inequality
\begin{equation}
\label{eq:ioH-2}
\| u^0 \| _{L^2( \ID) } \leq  \int_{0}^{T}\left\Vert \d_n( U_{V}(t)  u^0) \right\Vert _{L^{2}(\Gamma)  }^{2}dt
\end{equation}
is not satisfied. Then, there exists a sequence $(u_n^0)_{n \in \N}$ in $L^2(\ID)$ such that
\begin{equation}
\label{equniciteb}
   \| u_n^0 \| _{L^2( \ID) } = 1  , \qquad
\int_{0}^{T}\left\Vert \d_n( U_{V}(t)  u^0_n) \right\Vert _{L^{2}(\Gamma)  }^{2}dt \to 0 ,
  \quad n \to \infty .
\end{equation}
Then,~\eqref{e:oi:relaxb} implies that $u^0_n$ is bounded in $H^1(\ID)$, so that, after having extracted a subsequence, we have $u_n^0 \rightharpoonup u^0$ in $H^1_0(\ID)$ and $u_n^0 \to u^0$ in $L^2(\ID)$. We deduce from~\eqref{equniciteb} that
$$
   \| u^{0} \| _{L^2( \ID) } = 1  , \qquad
  \d_n( U_{V}(t)  u^{0} )= 0 \text{ on } \Gamma \text{ for all } t \in (0,T) .
$$
From here, we discuss the two cases with different uniqueness arguments. In the case $V(t,z) =V(z)$, the proof of $u=0$ follows exactly Section~\ref{s:obstimeindep} (using unique continuation from the boundary for elliptic operators).  The same conclusion holds if we assume~\eqref{UCPbord}. This contradicts $\| u^0 \| _{L^2( \ID) } = 1$, and proves~\eqref{eq:ioH-2}. Then,~\eqref{eq:ioH-2} and~\eqref{e:oi:relaxb} imply the sought observability inequality~\eqref{e:ob}.

\begin{proof}[Proof of Lemma~\ref{l:obs-brelax}]
Assume that~\eqref{e:oi:relaxb} is not satisfied. Then, there exists a sequence $u_n^0$ such that
\begin{align}
&\label{e:unbornH1} \| u_n^0 \| _{H^1( \ID)  }^{2} =  1 ,\\
& \label{e:unto0L2} \| u_n^0 \| _{L^2( \ID)  }^{2}  \to 0 , \\
&\label{e:dnunto0} \int_0^T\|  \d_n( u_n(t) ) \|_{L^2(\Gamma)}^2 dt \to 0 ,
\end{align}
where, as usual, $u_n(t) = U_V(t) u_n^0$. Let us now fix $\chi_T \in C^\infty(\R)$ such that $\chi_T = 1$ in a neighbourhood of $[0,T]$, $\psi \in C^\infty(\R)$, such that $\psi = 0$ on $(-\infty ,1]$ and $\psi = 1$ on $[2,+\infty)$ and set
 $$
w_n = B(D_t) \chi_T(t) u_n , \qquad B(D_t)= \Op_1(\psi(H)\sqrt{2H} ) .
$$
\begin{lemma}
\label{l:Dtun}
We set $A(D_t)= \Op_1\left(\frac{\psi(H)}{\sqrt{2H}}\right)$. For any $R>0$ and $\varepsilon >0$, we have
\begin{align}
& \label{e:Schrodbor} \left\|\left( D_t + \frac12 \Delta - V\right)w_n \right\|_{L^2((0,T) \times \ID)} \to 0 \\
& \label{e:regto0} \left\| A(D_t) w_n \right\|_{L^2((-R, R) \times \ID)} \to 0 , \\
& \label{e:wL2borinf} T/2 + o_{\varepsilon} (1) \leq\left\| w_n \right\|_{L^2((-\varepsilon, T+\varepsilon) \times \ID)}^2 \\
& \label{e:wL2borsup} \left\| w_n \right\|_{L^2((-R, R) \times \ID)}^2 \leq R + \varepsilon + o_{R,\varepsilon} (1),\\
& \label{e:trabor} \left\| \d_n( A(D_t)w_n ) \right\|_{L^2((-R,R) \times \d\ID)} \leq C , \\
& \label{e:wnormalL2bor} \|  \d_n( A(D_t)w_n ) \|_{L^2((\varepsilon, T-\varepsilon) \times\Gamma)} \to 0 .
\end{align}
\end{lemma}

Now, as $(w_n)$ forms a bounded sequence of $L^2_{\loc}(\R \times\ID)$, we associate to a subsequence a microlocal measure $\mu_{ml}=\mu^\infty+\mu_c$ as in Section~\ref{s:mlst}. {According to~\eqref{e:Schrodbor} and Remark~\ref{r:solo1} $\mu_{ml}$ satisfies the conclusions of Theorem~\ref{t:preciseml}.}
The measure $\mu_c$ vanishes as a consequence of~\eqref{e:regto0}. According to~\eqref{e:trabor}, the sequence $\d_n( A(D_t)w_n )$ is bounded in $L^2((0,T) \times\ID)$, so we may again extract another subsequence and associate a microlocal measure $\mu_{ml}^\d$ as in Section~\ref{s:micromesboundary}. According to~\eqref{e:wnormalL2bor},  $\mu_{ml}^\d = 0$ on $T^* ((\varepsilon,T-\varepsilon)\times \Gamma)$. As a consequence of Lemma~\eqref{l:muinfty=0bo},  $\mu^\infty $ vanishes identically on $\R_t \times \R^2_z \times \IS^2_{H, \xi}$.  Thus, $w_n$ converges to zero in  $L^2_{\loc}(\R \times\ID)$, which is contradiction with~\eqref{e:wL2borinf}. This concludes the proof of Lemma~\ref{l:obs-brelax}.
\end{proof}

\begin{proof}[Proof of Lemma~\ref{l:Dtun}]
Take $\tilde{\chi} \in C^\infty_c(\R)$ such that $\tilde{\chi}= 1$ on $(0,T)$ and $\chi_T=1$ on $\supp(\tilde{\chi} )$. Using that $\left( D_t + \frac12 \Delta - V\right)u_n =0$, we have
\begin{align*}
 \left\|\left( D_t + \frac12 \Delta - V\right)w_n \right\|_{L^2(0,T \times \ID)}
 & \leq  \left\| \tilde{\chi} \left( D_t + \frac12 \Delta - V\right)w_n \right\|_{L^2(\R \times \ID)} \\
  & \leq  \left\| \tilde{\chi} B( D_t ) \chi_T' u_n \right\|_{L^2(\R \times \ID)}  + \left\| \tilde{\chi} [V , B( D_t ) ]\chi_T u_n \right\|_{L^2(\R \times \ID)} \\
    & \leq  C \left\|  \chi_T' u_n \right\|_{L^2(\R \times \ID)}  + C\left\|\chi_T u_n \right\|_{L^2(\R \times \ID)}
     \leq C\left\| u_n^0 \right\|_{L^2( \ID)} \to 0 ,
\end{align*}
as $\tilde{\chi} =0$ on $\supp(\chi_T')$ and $[V , B( D_t ) ]$ is bounded on $L^2(\R \times \ID)$. This proves~\eqref{e:Schrodbor}.

Let us now take $\tilde{\chi} \in C^\infty_c(\R)$ such that $\tilde{\chi}= 1$ on $(-R,R)$. We have
\begin{align*}
\left\| A(D_t) w_n \right\|_{L^2((-R, R) \times \ID)} & \leq  \left\| \tilde{\chi} \Op_1(\psi^2(H)) \chi_T u_n \right\|_{L^2(\R \times \ID)} \\
& \leq  \left\| \chi_T u_n \right\|_{L^2(\R \times \ID)}\to 0 ,
\end{align*}
which proves~\eqref{e:regto0}.

Let us fix now $\check{\chi}\in C^\infty_c(-\varepsilon, T+\varepsilon) $ such that $\check{\chi} = 1$ in a neighbourhood of $[0,T]$, and compute 
\begin{align*}
\left\| w_n \right\|_{L^2((-\varepsilon, T+\varepsilon) \times \ID)}^2
&  \geq \left\| \check{\chi} w_n \right\|_{L^2(\R\times \ID)}^2
= \la B(D_t) \check{\chi}^2 B(D_t) \chi_T u_n , \chi_T u_n \ra_{L^2(\R\times \ID)} \nonumber  \\
& \geq  \la  \check{\chi}^2 \Op_1(\psi^2(H)) D_t (\chi_T u_n ), \chi_T u_n \ra_{L^2(\R\times \ID)} \nonumber \\
& \quad +\la [B(D_t) , \check{\chi}^2] B(D_t) \chi_T u_n , \chi_T u_n \ra_{L^2(\R\times \ID)} ,
\end{align*}
where $\left| \la [B(D_t) , \check{\chi}^2] B(D_t) \chi_T u_n , \chi_T u_n \ra_{L^2(\R\times \ID)} \right| \leq \left\| \chi_T u_n \right\|_{L^2(\R \times \ID)}\to 0$. As a consequence, we have
\begin{align}
\label{numerobis1}
\left\| w_n \right\|_{L^2((-\varepsilon, T+\varepsilon) \times \ID)}^2
&  \geq  \la  \check{\chi}^2 \Op_1(\psi^2(H)) D_t( \chi_T u_n ), \chi_T u_n \ra_{L^2(\R\times \ID)} + o(1) .
\end{align}
On the other hand, we have
\begin{align}
\label{numerobis2}
\la \check{\chi}^2 D_t( \chi_T u_n ), \chi_T u_n \ra_{L^2(\R\times \ID)}
& = \la \check{\chi}^2 \Op_1(\psi^2(H)) D_t( \chi_T u_n ), \chi_T u_n \ra_{L^2(\R\times \ID)} \nonumber \\
& \quad + \la \check{\chi}^2 \Op_1(\psi^2(-H)) D_t( \chi_T u_n ), \chi_T u_n \ra_{L^2(\R\times \ID)}  + o(1) ,
\end{align}
where
\begin{align*}
\la \check{\chi}^2 \Op_1(\psi^2(-H)) D_t( \chi_T u_n ), \chi_T u_n \ra_{L^2(\R\times \ID)}
= \la \check{\chi}^2 \Op_1(\psi^2(-H)H) ( \chi_T u_n ), \chi_T u_n \ra_{L^2(\R\times \ID)}
\end{align*}
Since $\check{\chi}^2 \psi^2(-H)H \leq 0$ the sharp G{\aa}rding inequality then provides
\begin{align*}
\la \check{\chi}^2 \Op_1(\psi^2(-H)) D_t( \chi_T u_n ), \chi_T u_n \ra_{L^2(\R\times \ID)}
\leq \| \chi_T u_n \|_{L^2(\R\times \ID)}^2  =o(1)
\end{align*}
This, combined with~\eqref{numerobis1} and~\eqref{numerobis2} now yields
\begin{align*}
\left\| w_n \right\|_{L^2((-\varepsilon, T+\varepsilon) \times \ID)}^2
&  \geq  \la  \check{\chi}^2  D_t( \chi_T u_n ), \chi_T u_n \ra_{L^2(\R\times \ID)} + o(1)\\
& \geq  \la  \check{\chi}^2 \chi_T  D_t u_n , \chi_T u_n \ra_{L^2(\R\times \ID)} + o(1)\\
& \geq  \la  \check{\chi}^2  \chi_T  (-\frac12 \Delta + V) u_n , \chi_T u_n \ra_{L^2(\R\times \ID)} + o(1)\\
& \geq  \frac12 \la  \check{\chi}^2 \chi_T \nabla u_n , \chi_T \nabla u_n \ra_{L^2(\R\times \ID)} + o(1)\\
& \geq  \frac{T}{2} \| \nabla u_n^0\|_{L^2(\ID)}^2 + o(1) =  \frac{T}{2} + o(1)\\
\end{align*}
This concludes the proof of ~\eqref{e:wL2borinf}. The proof of~\eqref{e:wL2borsup} follows the same arguments.

Finally, according to~\eqref{e:unbornH1} and the hidden regularity result of Proposition~\ref{p:regboundary} the sequence $\d_n( u_n )$ is bounded in $L^2((-R,R) \times\ID)$. Moreover, we have
\begin{align*}
\left\| \d_n( A(D_t)w_n ) \right\|_{L^2((-R,R) \times \d\ID)}
& \leq \left\| \tilde{\chi} \Op_1(\psi^2(H)) \chi_T \d_n(u_n ) \right\|_{L^2(\R\times \d\ID)} \\
& \leq \left\| \chi_T \d_n(u_n ) \right\|_{L^2(\R \times \d\ID)} \leq C.
\end{align*}
This proves \eqref{e:trabor}. The proof of~\eqref{e:wnormalL2bor} comes from a similar computation combined with~\eqref{e:dnunto0}.

\end{proof}

\appendix

 \section{From action-angle coordinates to polar coordinates}
 \label{coord-polar}
 Here we develop the (painful) calculations leading to the definitions of the operators $\cA_E(P)$ and $\cA_H(P)$ used as a black-box in the paper. The point is that our ``action-angle'' coordinates $(s, \theta, E, J)$, well adapted to integrate the dynamics of the billiard flow, are not so convenient to express the Dirichlet boundary condition ($v(z)=0$ for $|z|=1$). Actually the best coordinates in which to write the boundary condition are the polar coordinates (which below will be written as $(x=-r\sin u, y=r\cos u)$) since the boundary is simply expressed as the set $\{r=1\}$.

Let $P(s, \theta, E, J)$ be a function expressed in the new coordinates and let $\scrU $ be the Fourier integral operator defined in \eqref{e:U}. The technical calculations done below are aimed at understanding how  $\scrU^* \Op_h(P) \scrU$ acts in polar coordinates; in particular, under which conditions on the symbol $P$ the boundary condition is preserved by $\scrU^* \Op_h(P) \scrU$.

 For our purposes we need to understand the operator $\scrU^* \Op_h(P) \scrU$ modulo $O(h^2)$. Ideally we would like to separate it into a ``tangential part'' (involving only angular derivation $\frac{\partial}{\partial u}$) and a ``radial part'' involving the radial derivative $\frac{\partial}{\partial r}$ in a simple way. Below we calculate the action of the operator $\scrU^* \Op_h(P) \scrU$ on a plane wave $e_\xi(z):=e^{i\frac{(\xi_x x+\xi_y y)}h}$  (where we use $z=(x, y)$,  $\xi=(\xi_x, \xi_y)$ and $|\xi|^2= \xi_x^2 + \xi_y^2$) and apply the method of stationary phase. The length of the calculation comes from the fact that we explicitly need the term of order $h$ in the expansion.

  Let $P(s, \theta, E, J)$ be a smooth function (possibly depending on $h$), with support away from $\{E=0\}$ and inside $\{|J|<E\}$. We assume that it satisfies $\norm{\partial_s^{\alpha}\partial_\theta^{\beta}\partial_E^{\gamma}
  \partial_J^{\delta}P}_\infty \leq C_{\alpha, \beta, \gamma,\delta} h^{-\gamma-\delta}$ for all integers $\alpha, \beta, \gamma,\delta$. The function $P$ may also depend on the time variable $t$ and its dual $H$, but here we omit them from the notation since they are transparent in the calculation.
 Using the definition \eqref{e:U} and unfolding all the integrals, we write
 \begin{align*}& \scrU^* \Op_h(P) \scrU e_\xi (x, y)\\
&= (2\pi h)^{-5}\int P\left( s,\theta , E', j\right) e^{\frac{ij(\theta-\theta')}h}e^{\frac{iE'(s-s')}h}
 e^{-i\frac{S(x', y', \theta', s', E)}h}  e^{i\frac{(\xi_x x'+\xi_y y')}h}   e^{i\frac{S(x, y, \theta, s, E'')}h}\\ & a(E)a(E'')  d\theta ds dE''  dx' dy' dE d\theta'ds'dE' dj\\
 &= (2\pi h)^{-3}\int P\left( s,\theta , E', j\right)e^{\frac{ij(\theta-\theta_0)}h}e^{\frac{iE'(s-s')}h}
 e^{i\frac{ s' |\xi|}h}     e^{i\frac{S(x, y, \theta, s, E'')}h} \frac{a(|\xi|)}{|\xi|}a(E'')  d\theta ds dE''    ds'dE' dj\\
 &= (2\pi h)^{-2}\int P\left( s,\theta , |\xi|, j\right)e^{\frac{ij(\theta-\theta_0)}h}
 e^{i\frac{ s |\xi|}h}     e^{i\frac{S(x, y, \theta, s, E'')}h} \frac{a(|\xi|)}{|\xi|}a(E'')  d\theta  dE''    dsdj\\
  &= (2\pi h)^{-1}\int ( P\left( s(x, y, \theta),\theta , |\xi|, j\right)a(|\xi|)-ih\partial_s P\left( s(x, y, \theta),\theta , |\xi|, j\right)a'(|\xi|)
   ) \\&e^{\frac{ij(\theta-\theta_0)}h}
      e^{i\frac{-|\xi|\sin\theta x +|\xi|\cos\theta y}h}  \frac{a(|\xi|)}{|\xi|}d\theta    dj\qquad +\frac{O(h^2)}{\inf_{P(s, \theta, E, J)\not=0}|E|^2},
      \end{align*}
  where {$a(E) = \sqrt{E}$}, $(s_0, \theta_0, |\xi|, j_0) =\Phi^{-1}(x, y,\xi_x, \xi_y)$ and $s(x, y, \theta)=-x\sin\theta+y\cos\theta.$ By standard estimates on pseudodifferential operators, the remainder term will correspond to an estimate in the $L^2_{\comp}\To L^2_{\loc}$ norm of operators.

       Let $(x, y)=(-r\sin u, r\cos u)$, so that $r=\sqrt{x^2+y^2}$, $u=\arccos y/r$, and $s(x, y, \theta)=r\cos(\theta-u)$.
      We are left with
 \begin{multline*} (2\pi h)^{-1} \frac{a(|\xi|)}{|\xi|}\int ( P\left(r\cos(\theta-u), \theta,|\xi|, j\right)a(|\xi|)\\-ih \partial_s P\left(r\cos(\theta-u), \theta,|\xi|, j\right)a'(|\xi|))e^{\frac{ij(\theta-\theta_0)}h}
    e^{i\frac{|\xi|r\cos(\theta-u)}h}d\theta    dj .
     \end{multline*}

Now we apply stationary phase w.r.t. $\theta$ (while $j$ is kept fixed, since our symbols can be rapidly oscillating in $j$).
We start with the $P$ term. The $ih\partial_s P$-term can be treated exactly the same way. Fixing $j$ and looking at the $\theta$-integral, we let
\begin{eqnarray*}I&=&(2\pi h)^{-1/2}\int_{0}^{2\pi}  P\left(r\cos(\theta-u), \theta,|\xi|, j\right) e^{\frac{ij(\theta-\theta_0)}h}
    e^{i\frac{|\xi|r\cos(\theta-u)}h}d\theta \\
&=&(2\pi h)^{-1/2}\int_{u-\pi/2}^{u+\pi/2}\cdots +    (2\pi h)^{-1/2}\int_{u+\pi/2}^{u+3\pi/2}\cdots
    .
    \end{eqnarray*}

    The phase in $I$ has $2$ critical points $\theta=u+\theta_1,u+ \theta_2$, where $\theta_i$ are the solutions of $j- |\xi|r\sin\theta=0$. Since we are assuming that $P(s, \theta, E, j)$ is supported in $\{|j|< E\}$, these two solutions are distinct for $r$ close to $1$, and correspond to non-degenerate stationary points (in all that follows we consider that $r$ is close to $1$ since this calculation only serves to understand $\scrU^* \Op_h(P) \scrU$ near the boundary of the disk). We will denote by $\theta_1(r, E, j), \theta_2(r, E, j)$ the solutions of $j-Er\sin\theta=0$. To fix ideas, $\theta_1$ will be the one with $\cos\theta_1>0$ and $\theta_2$ the one with $\cos\theta_2<0$ (that is, $\theta_1\in (-\pi/2, \pi/2), \theta_2\in (\pi/2, 3\pi/2)$). Below, $E$ will always take the value $E=|\xi|$.

According to the method of stationary phase, the asymptotic expansion of the integral $I$ may be obtained, modulo $O(h^2)$, by replacing $P$ by its Taylor expansion at order $2$ at each critical point:
\begin{multline}\label{e:dl1}P(r\cos(\theta-u), \theta)\sim P(r\cos\theta_i, u+\theta_i)+(\theta-u-\theta_i)\frac{d}{d\theta}P(r\cos\theta_i, u+\theta_i)\\
+(\theta-u-\theta_i)^2\frac12\frac{d^2}{d\theta^2}P(r\cos\theta_i, u+\theta_i).
\end{multline}
(we momentarily drop the $j$ and $E$ variables from the notation since they are fixed in the upcoming calculation).

\begin{remark}\label{r:partialderiv}
Denoting by $\partial_1=\partial_s, \partial_2=\partial_\theta$ (to avoid possible confusion), we have
\begin{equation*}
\frac{d}{d\theta}P( r\cos\theta, u+\theta)=\partial_2 P ( r\cos\theta, u+\theta)-r\sin\theta\,\partial_1 P ( r\cos\theta, u+\theta) ,
\end{equation*}
and
\begin{multline*}
\frac{d^2}{d\theta^2}P( r\cos\theta, u+\theta)=\partial^2_2 P ( r\cos\theta, u+\theta)-r\cos\theta\,\partial_1 P ( r\cos\theta, u+\theta)\\-r\sin\theta\,\partial_2\partial_1 P ( r\cos\theta, u+\theta)+r^2\sin^2\theta\,\partial^2_1 P ( r\cos\theta, u+\theta) .
\end{multline*}
 \end{remark}

 To use integration by parts, it is convenient to rewrite \eqref{e:dl1} using $j-  Er\sin(\theta-u)$ (the derivative of the phase) instead of $(\theta-u-\theta_i)$ (where $\theta_i$ stands short for $\theta_i(r, E, j)$). Starting with
 \begin{eqnarray}j-  Er\sin(\theta-u)\sim Er \left[-\cos\theta_i(\theta-u-\theta_i)+\frac{\sin\theta_i}2 (\theta-u-\theta_i)^2\right],
 \end{eqnarray}
 equation \eqref{e:dl1} can be rewritten as
 \begin{multline}\label{e:dl2}P(r\cos(\theta-u), \theta)\sim
 P(r\cos\theta_i, u+\theta_i)-\frac{\left(j-  Er\sin(\theta-u)\right)}{Er\cos\theta_i}\frac{d}{d\theta}P(r\cos\theta_i, u+\theta_i)\\
+\frac{\left(j-  Er\sin(\theta-u)\right)^2}{(Er\cos\theta_i)^2}
 \left[\frac{\sin\theta_i}{2\cos\theta_i}\frac{d}{d\theta}P(r\cos\theta_i, u+\theta_i)+\frac12\frac{d^2}{d\theta^2}P(r\cos\theta_i, u+\theta_i)\right].
\end{multline}

\begin{remark} Important remark about symmetry. We keep denoting $\theta_i$ for $\theta_i(r, E, j)$. We first note that $\theta_2=\pi-\theta_1$, $\cos\theta_1=-\cos\theta_2, \sin\theta_1=\sin\theta_2$.

Moreover, if $P$ satisfies the symmetry condition (B), we have for $r=1$ (restoring in our notation the dependence of $P$ on the full set of variables)
$$P(\cos\theta_1, u+\theta_1,  E, j)=P(\cos\theta_2, u+\theta_2,  E, j).$$
And similarly for all partial derivatives of $P$ if we assume the stronger symmetry condition (C).
\end{remark}
Here we don't necessarily want to assume that $P$ is symmetric; but, motivated by the previous remark, we introduce the following notation:
\begin{eqnarray}  \label{e:sym}
  P^\sigma(r, \theta, E, j)
&:=&\frac{P(r\cos\theta_1, \theta+\theta_1, E, j)+ P(-r\cos\theta_1, \theta+\pi-\theta_1, E, j)}{2}\nonumber\\
&=& \frac{P(r\cos\theta_1, \theta+\theta_1, E, j)+ P(r\cos\theta_2, \theta+\theta_2, E, j)}{2}
\end{eqnarray}
\begin{eqnarray}\label{e:antisym}
  P^\alpha(r, \theta, E, j)
&:=&\frac{P(r\cos\theta_1, \theta+\theta_1, E, j)- P(-r\cos\theta_1, \theta+\pi-\theta_1, E, j)}{2} \nonumber \\
&=&\frac{P(r\cos\theta_1, \theta+\theta_1, E, j)- P(r\cos\theta_2, \theta+\theta_2, E, j)}{2}
\end{eqnarray}
for $\theta_1=\theta_1(r, E, j), \theta_2=\theta_2(r, E, j)$ defined previously,
so that
$$P(r\cos\theta_1, \theta+\theta_1, E, j)=  P^\sigma(r, \theta, E, j)+  P^\alpha(r, \theta, E, j),$$
$$P(r\cos\theta_2, \theta+\theta_2, E, j) =  P^\sigma(r, \theta, E, j)-  P^\alpha(r, \theta, E, j).$$

 Modulo $O(h)$ the asymptotic expansion of the integral $I$ looks as follows:

\begin{multline}(2\pi h)^{-1/2}\int_{u-\pi/2}^{u+\pi/2}P(r\cos\theta_1, u+\theta_1,  E, j) e^{\frac{ij(\theta-\theta_0)}h}
    e^{i\frac{Er\cos(\theta-u)}h}d\theta
\\ +    (2\pi h)^{-1/2}\int_{u+\pi/2}^{u+3\pi/2}P(r\cos\theta_2, u+\theta_2,  E, j) e^{\frac{ij(\theta-\theta_0)}h}
    e^{i\frac{Er\cos(\theta-u)}h}d\theta\\
   =( 2\pi h)^{-1/2}\int_0^{2\pi}   P^\sigma(r, u, E, j) e^{\frac{ij(\theta-\theta_0)}h}
    e^{i\frac{Er\cos(\theta-u)}h}d\theta\\
    +( 2\pi h)^{-1/2}\int_0^{2\pi} \frac{  P^\alpha(r, u, E, j)}{E\cos\theta_1(r, E, j)}\left(E\cos(\theta-u)-ih\frac{1}{2r\cos^2\theta_1(r, E, j)}\right) e^{\frac{ij(\theta-\theta_0)}h}
    e^{i\frac{Er\cos(\theta-u)}h}d\theta
  \\
   =( 2\pi h)^{-1/2}\int_0^{2\pi}   P^\sigma(r, u, E, j) e^{\frac{ij(u-\theta)}h}
    e^{i\frac{Er\cos(\theta-\theta_0)}h}d\theta\\
    +( 2\pi h)^{-1/2}\int_0^{2\pi} \frac{  P^\alpha(r, u, E, j)}{E\cos\theta_1(r, E, j)}\left(E\cos(\theta-\theta_0)-ih\frac{1}{2r\cos^2\theta_1(r, E, j)}\right) e^{\frac{ij(u-\theta)}h}
    e^{i\frac{Er\cos(\theta-\theta_0)}h}d\theta 
    \label{e:h0}\end{multline}
We note that $e^{i\frac{Er\cos(\theta-\theta_0)}h}=e^{i\frac{(\xi_xx'+\xi_y y')}h}$ if $(x', y')=(-r\sin \theta, r\cos \theta)$, and
$E\cos(\theta-\theta_0)
    e^{i\frac{Er\cos(\theta-\theta_0)}h} =hD_r     e^{i(\xi_xx'+\xi_y y')/h}$ where
 $D_r=\frac{1}i \partial_r$.

\bigskip

{\bf Other terms of order $h$.} Apart from the term of order $h$ arising in the last line of \eqref{e:h0}, other terms of order $h$ come from evaluation of the integrals
\begin{multline*}
(2\pi h)^{-1/2}\int_{u-\pi/2}^{u+\pi/2} \Bigg[ \frac{-\left(j-  Er\sin(\theta-u)\right)}{Er\cos\theta_i}\frac{d}{d\theta}P(u+\theta_1, r\cos\theta_1)\\
+\frac{\left(j-  Er\sin(\theta-u)\right)^2}{(Er\cos\theta_1)^2}
 \left(\frac{\sin\theta_1}{2\cos\theta_1}\frac{d}{d\theta}P(u+\theta_1, r\cos\theta_1)+\frac12\frac{d^2}{d\theta^2}P(u+\theta_1, r\cos\theta_1)\right)\Bigg] \\
 e^{\frac{ij(\theta-\theta_0)}h}
    e^{i\frac{Er\cos(\theta-u)}h}d\theta
\\ +    (2\pi h)^{-1/2}\int_{u+\pi/2}^{u+3\pi/2}
 \Bigg[\frac{-\left(j-  Er\sin(\theta-u)\right)}{Er\cos\theta_2}\frac{d}{d\theta}P(u+\theta_2, r\cos\theta_2)\\
+\frac{\left(j-  Er\sin(\theta-u)\right)^2}{(Er\cos\theta_2)^2}
 \left(\frac{\sin\theta_2}{2\cos\theta_2}\frac{d}{d\theta}P(u+\theta_2, r\cos\theta_2)+\frac12\frac{d^2}{d\theta^2}P(u+\theta_2, r\cos\theta_2)\right)\Bigg] \\
 e^{\frac{ij(\theta-\theta_0)}h}
    e^{i\frac{Er\cos(\theta-u)}h}d\theta
 \end{multline*}
 After integrations by parts w.r.t. $\theta$ (using the fact that $j-  Er\sin(\theta-u)$ is the derivative of the phase $j(\theta-\theta_0)+Er\cos(\theta-u)$), this becomes
 \begin{multline}
(2\pi h)^{-1/2} ih\int_{u-\pi/2}^{u+\pi/2}
\frac{1}{(Er\cos\theta_1)^2} Er\cos(\theta-u)
\\ \left[\frac{\sin\theta_1}{2\cos\theta_1}\frac{d}{d\theta}P(u+\theta_1, r\cos\theta_1)+\frac12\frac{d^2}{d\theta^2}P(u+\theta_1, r\cos\theta_1)\right]e^{\frac{ij(\theta-\theta_0)}h}
    e^{i\frac{Er\cos(\theta-u)}h}d\theta
\\ +    (2\pi h)^{-1/2}ih\int_{u+\pi/2}^{u+3\pi/2}
  \frac{1}{(Er\cos\theta_2)^2}Er\cos(\theta-u)
\\ \left[\frac{\sin\theta_2}{2\cos\theta_2}\frac{d}{d\theta}P(u+\theta_2, r\cos\theta_2)+\frac12\frac{d^2}{d\theta^2}P(u+\theta_2, r\cos\theta_2)\right]e^{\frac{ij(\theta-\theta_0)}h}
    e^{i\frac{Er\cos(\theta-u)}h}d\theta\\
 =   (2\pi h)^{-1/2} ih\int_{u-\pi/2}^{u+\pi/2}
\frac{1}{(Er\cos\theta_1)}  \\ \left[\frac{\sin\theta_1}{2\cos\theta_1}\frac{d}{d\theta}P(u+\theta_1, r\cos\theta_1)+\frac12\frac{d^2}{d\theta^2}P(u+\theta_1, r\cos\theta_1)\right]e^{\frac{ij(\theta-\theta_0)}h}
    e^{i\frac{Er\cos(\theta-u)}h}d\theta
\\ +    (2\pi h)^{-1/2}ih\int_{u+\pi/2}^{u+3\pi/2}
  \frac{1}{(Er\cos\theta_2)}
\\ \left[\frac{\sin\theta_2}{2\cos\theta_2}\frac{d}{d\theta}P(u+\theta_2, r\cos\theta_2)+\frac12\frac{d^2}{d\theta^2}P(u+\theta_2, r\cos\theta_2)\right]e^{\frac{ij(\theta-\theta_0)}h}
    e^{i\frac{Er\cos(\theta-u)}h}d\theta +O(h^2)
    \label{e:secondorder}
 \end{multline}
 Here $\frac{d}{d\theta}P(u+\theta_i, r\cos\theta_i)$ and $\frac{d^2}{d\theta^2}P(u+\theta_i, r\cos\theta_i)$ may be replaced by their expressions in terms of partial derivatives of $P$, as in Remark \ref{r:partialderiv}.

\bigskip
We summarize our calculations in the following proposition.

\begin{proposition}
\label{p:summary}
Modulo a term of order $\frac{O(h^2)}{\inf_{P(s, \theta, E, J)\not=0}|E|^2}$ in the $L^2_{\comp}\To L^2_{\loc}$-norm of operators, $\scrU^* \Op_h(P) \scrU $ acts as follows. For $\xi=(\xi_x, \xi_y)$, $E=|{\xi}|$ and $(x, y)=(-r\sin u, r\cos u)$, we have
\begin{multline*}
\scrU^* \Op_h(P) \scrU   e_\xi (x, y)= \frac1{2\pi h}\int A(r, u, E, j) e^{\frac{ij(u-\theta)}h}
     e_\xi (-r\sin\theta, r\cos\theta)d\theta dj \\+\frac1{2\pi h}\int B(r, u, E, j)  e^{\frac{ij(u-\theta)}h}
    hD_r e_\xi (-r\sin\theta, r\cos\theta)d\theta dj\\
    + \frac{ih}{2\pi h}\int C(r, u, E, j) e^{\frac{ij(u-\theta)}h}
     e_\xi (-r\sin\theta, r\cos\theta)d\theta dj \\+\frac{ih}{2\pi h}\int D(r, u, E, j)  e^{\frac{ij(u-\theta)}h}
    hD_r e_\xi (-r\sin\theta, r\cos\theta)d\theta dj
\end{multline*}
where
$$ A(r, u, E, j)= P^\sigma(r, u, E, j)$$
$$ B(r, u, E, j)= \frac{P^\alpha(r, u, E, j)}{E\cos\theta_1(r, E, j)}$$
$$ C(r, u, E, j)=- \frac{1}{2E}\partial_s P^\sigma(r, u, E, j) + c^\sigma(r, u, E, j) -\frac{1}{2r\cos^2\theta_1(r, E, j)}\frac{P^\alpha(r, u, E, j)}{E\cos\theta_1(r, E, j)}$$
\begin{equation}\label{e:D}D(r, u, E, j)=-\frac{1}{2E}\frac{\partial_s P^\alpha(r, u, E, j)}{E\cos\theta_1(r, E, j)}
+\frac{c^\alpha(r, u, E, j)}{E\cos\theta_1(r, E, j)}
\end{equation}
with the notation $P^\sigma, P^\alpha$ of \eqref{e:sym}, \eqref{e:antisym}, and
where, in addition
\begin{multline}c(s, \theta, E, j)
 =\frac{1}{(Es)}
 \Bigg[\frac{j}{2Es}
 \left(\partial_2 P ( s, \theta, E, j)-\frac{j}E \partial_1 P ( s, \theta, E, j)\right)
 \\+\frac12\left(\partial^2_2 P ( s, \theta, E, j)-s\partial_1 P ( s, \theta, E, j)-\frac{j}E\partial_2\partial_1 P ( s, \theta, E, j)+\frac{j^2}{E^2}\partial^2_1 P ( s, \theta, E, j)\right)\Bigg]
\end{multline}
is calculated so that $c(u+\theta_i, r\cos\theta_i, E, j) $ equals
$$ \frac{1}{(Er\cos\theta_i)}
 \left[\frac{\sin\theta_i}{2\cos\theta_i}\frac{d}{d\theta}P(u+\theta_i, r\cos\theta_i)+\frac12\frac{d^2}{d\theta^2}P(u+\theta_i, r\cos\theta_i)\right]$$
(the expression appearing in the last line of \eqref{e:secondorder}).

Note that $A, B, C, D$ are real-valued functions if $P$ is.

\end{proposition}

\section{Commutators\label{s:commutator}}
In the following formal calculations, it will be convenient to introduce the following notation.

${\cA}_E(P)$ is the operator whose action on $e_\xi$ at $(-r\sin u, r\cos u)$ is defined by
\begin{multline*}
  \frac1{2\pi h}\int A(r, u, E, j) e^{\frac{ij(u-\theta)}h}
    e_\xi (-r\sin\theta, r\cos\theta)d\theta dj \\+\frac1{2\pi h}\int B(r, u, E, j)  e^{\frac{ij(u-\theta)}h}
     hD_r e_\xi (-r\sin\theta, r\cos\theta)d\theta dj\\
    + \frac{ih}{2\pi h}\int C(r, u, E, j) e^{\frac{ij(u-\theta)}h}
    e_\xi (-r\sin\theta, r\cos\theta)d\theta dj \\+\frac{ih}{2\pi h}\int D(r, u, E, j)  e^{\frac{ij(u-\theta)}h}
     hD_r e_\xi (-r\sin\theta, r\cos\theta)d\theta dj\\
    =: I_E(P)+II_E(P)+ih III_E(P)+ih IV_E(P)
\end{multline*}
where $E=|{\xi}|$. We have shown that $\cA_E(P)$ coincides with $\scrU^* \Op(P) \scrU$ modulo $\frac{O(h^2)}{\inf_{P(s, \theta, E, J)\not=0}|E|^2}$ in the $L^2_{\comp}\To L^2_{\loc}$-norm of operators.

We now define ${\cA}_H(P)$ as the operator whose action on $e_\xi$ at $(-r\sin u, r\cos u)$ is defined by
\begin{multline*}
  \frac1{2\pi h}\int A(r, u, \sqrt{2H}, j) e^{\frac{ij(u-\theta)}h}
    e_\xi (-r\sin\theta, r\cos\theta)d\theta dj \\+\frac1{2\pi h}\int B(r, u, \sqrt{2H}, j) e^{\frac{ij(u-\theta)}h}
    hD_r e_\xi (-r\sin\theta, r\cos\theta)d\theta dj\\
    + \frac{ih}{2\pi h}\int C(r, u, \sqrt{2H}, j) e^{\frac{ij(u-\theta)}h}
    e_\xi (-r\sin\theta, r\cos\theta)d\theta dj \\+\frac{ih}{2\pi h}\int D(r, u, \sqrt{2H}, j)  e^{\frac{ij(u-\theta)}h}
    hD_r e_\xi (-r\sin\theta, r\cos\theta)d\theta dj
    \\
    =: I_H(P)+II_H(P)+ih III_H(P)+ih IV_H(P)
\end{multline*}
In other words, in the definition of $\cA_E(P)$ we have replaced $E$
with $\sqrt{2H}$ in the symbols.
  For us, $\cA_H(P)$ is a very convenient operator since we have
  $$ I_H(P)=A(r, u, \sqrt{2hD_t}, hD_u), \qquad III_H(P)=C(r, u, \sqrt{2hD_t}, hD_u)$$
  (so that they do not involve any derivative w.r.t. $r$) and
  $$ II_H(P)=B(r, u, \sqrt{2hD_t}, hD_u)\circ hD_r, \quad IV_H(P)=D(r, u, \sqrt{2hD_t}, hD_u)\circ hD_r$$
  which is is only of degree $1$ w.r.t. $r$.

  We would like to use everywhere $\cA_H(P)$ instead of $\cA_E(P)$. This is possible thanks to the following lemma:
  \begin{lemma}\label{l:EH} If $u_h$ is a solution to the Schr\"odinger equation \eqref{e:S} satisfying in addition the assumptions of Remark \ref{r:H1}, then we have
 \begin{multline*}
  \la u_h, \cA_E(P)u_h\ra_{L^2(\R^2\times\R^2\times \R)}-\la u_h, \cA_H(P)u_h\ra_{L^2(\R^2\times\R^2\times \R)} \\
  = O\bigg(  h\norm{\partial_E P}_\infty (h\inf_{P(s, \theta, E, J)\not=0}|E|)^{-1/2-\eps}\bigg).
  \end{multline*}
  \end{lemma}
  For instance if $P$ has bounded derivatives and $\inf_{P(s, \theta, E, J)\not=0}|E|$ is bounded away from $0$ independently of $h$, the error above is $O(h^{1/2-\eps})$.

   The goal of this section is to calculate explicitly (in terms of $P$) the expression of the commutator $[\Delta, {\cA}_H(P)]$, where $\Delta$ is the laplacian on $\R^2$. This could, in principle, be done by brutal calculation, using the expression of the laplacian in polar coordinates ($\Delta_{r, u}=\frac{\partial^2}{\partial_r^2}+\frac{1}r \frac{\partial}{\partial r }+\frac1{r^2}\frac{\partial^2}{\partial_u^2}$). But this is too cumbersome and we try a less frontal approach. We want to use the fact that  $[\Delta, {\cA}_E(P)]$ is known (from the exact Egorov theorem, equation \eqref{e:Egorov} below) and to see how the calculus is modified when we replace ${\cA}_E(P)$ by ${\cA}_H(P)$.

   Recall from Lemma \ref{p:FIO} and formula \eqref{e:weylcomm} that we have the exact formula (without remainder term)

   \begin{equation}\label{e:Egorov}
\left[-\frac{ih\Delta}2, \scrU^* \Op_h(P) \scrU\right]= \scrU^*
\Op_h\left(E\partial_1 b-\frac{ih}2\partial_1^2 b\right)\scrU
 \end{equation}

    \subsection{Formal calculation of $[\Delta, \cA_E(P)]$}
    We use the expression of $\nabla$ in polar coordinates: $\nabla=( \partial_r ,  r^{-1}\partial_u)$ in the orthonormal frame $(e_r, e_u)$. We also use the formula $\Delta(fg)=f\Delta g +2\nabla f\cdot\nabla g+g\Delta f$.
    We obtain the following expression of $[\Delta, I_E(P)]$ applied to $e_\xi$ at $(x, y)=(-r\sin u, r\cos u)$:
    \begin{multline}(2\pi h)^{-1}\int_{u-\pi/2}^{u+\pi/2}\Delta_{r, u}A(r, u, E, j) e^{\frac{ij(\theta-\theta_0)}h}
    e^{i\frac{Er\cos(\theta-u)}h}d\theta dj\\
    +\frac{2i}{h}(2\pi h)^{-1}\int_{u-\pi/2}^{u+\pi/2} \partial_r A(r, u, E, j)E\cos(\theta-u) e^{\frac{ij(\theta-\theta_0)}h}
    e^{i\frac{Er\cos(\theta-u)}h}d\theta dj\\
   +\frac{2i}{h} (2\pi h)^{-1}\int_{u-\pi/2}^{u+\pi/2} r^{-2}\partial_u A(r, u, E, j)j e^{\frac{ij(\theta-\theta_0)}h}
    e^{i\frac{Er\cos(\theta-u)}h}d\theta dj
   \end{multline}
Note that the details of the calculations are actually not important, we only need to know ``what the calculations look like'' at a formal level (in particular, small errors of calculation are harmless).

Similarly, $[\Delta, II_E(P)]$ has the expression
 \begin{multline}(2\pi h)^{-1}\int_{u-\pi/2}^{u+\pi/2}\Delta_{r, u}B(r, u, E, j)E\cos(\theta-u)  e^{\frac{ij(\theta-\theta_0)}h}
    e^{i\frac{Er\cos(\theta-u)}h}d\theta dj\\
    +\frac{2i}{h}(2\pi h)^{-1}\int_{u-\pi/2}^{u+\pi/2} \partial_r B(r, u, E, j)(E\cos(\theta-u))^2 e^{\frac{ij(\theta-\theta_0)}h}
   e^{i\frac{Er\cos(\theta-u)}h}d\theta dj\\
   +\frac{2i}{h} (2\pi h)^{-1}\int_{u-\pi/2}^{u+\pi/2} r^{-2}\partial_u B(r, u, E, j)j E\cos(\theta-u)e^{\frac{ij(\theta-\theta_0)}h}
    e^{i\frac{Er\cos(\theta-u)}h}d\theta dj\\
    =(2\pi h)^{-1}\int_{u-\pi/2}^{u+\pi/2}\Delta_{r, u}B(r, u, E, j)E\cos(\theta-u)  e^{\frac{ij(\theta-\theta_0)}h}
    e^{i\frac{Er\cos(\theta-u)}h}d\theta dj\\
    +\frac{2i}{h}(2\pi h)^{-1}\int_{u-\pi/2}^{u+\pi/2} \partial_r B(r, u, E, j)[(E\cos(\theta_1))^2+ih\frac{\cos\theta_1}{(Er)^2}] e^{\frac{ij(\theta-\theta_0)}h}
   e^{i\frac{Er\cos(\theta-u)}h}d\theta dj\\
   +\frac{2i}{h} (2\pi h)^{-1}\int_{u-\pi/2}^{u+\pi/2} r^{-2}\partial_u B(r, u, E, j)j E\cos(\theta-u)e^{\frac{ij(\theta-\theta_0)}h}
    e^{i\frac{Er\cos(\theta-u)}h}d\theta dj +O(h^2)
    \end{multline}

   Similar calculations can be done for $[\Delta, III_E(P)]$ and $[\Delta, IV_E(P)]$. We do not need the explicit expressions, but need only to note that it gives a final expression of  $[-ih\Delta/2, \cA_E(P)]$ applied to $e_\xi$ at $(-r\sin u, r\cos u)$ in the form:
\begin{multline*}
  \frac1{2\pi h}\int K(r, u, E, j) e^{\frac{ij(u-\theta)}h}
    e^{i\frac{Er\cos(\theta-\theta_0)}h}d\theta dj \\+\frac1{2\pi h}\int L(r, u, E, j) E\cos(\theta-\theta_0)e^{\frac{ij(u-\theta)}h}
    e^{i\frac{Er\cos(\theta-\theta_0)}h}d\theta dj\\
    + \frac{ih}{2\pi h}\int M(r, u, E, j) e^{\frac{ij(u-\theta)}h}
    e^{i\frac{Er\cos(\theta-\theta_0)}h}d\theta dj \\+\frac{ih}{2\pi h}\int N(r, u, E, j) E\cos(\theta-\theta_0)e^{\frac{ij(u-\theta)}h}
    e^{i\frac{Er\cos(\theta-\theta_0)}h}d\theta dj\\
    + (2\pi h)^{-1}\int  \partial_r B(r, u, E, j)\left[(E\cos(\theta_1))^2+ih\frac{\cos\theta_1}{(Er)^2}\right] e^{\frac{ij(\theta-\theta_0)}h}
   e^{i\frac{Er\cos(\theta-u)}h}d\theta dj\\
   + (2\pi h)^{-1}\int ih \partial_r D(r, u, E, j)(E\cos(\theta_1))^2 e^{\frac{ij(\theta-\theta_0)}h}
   e^{i\frac{Er\cos(\theta-u)}h}d\theta dj
\end{multline*}
Note that the two last lines may obviously be incorporated into the previous terms; but we shall see later why it is convenient to keep them separate.

 The functions $K, L, M,N$ are partial differential operators applied to $A, B, C, D$, and could in principle be expressed explicitly in terms of $P$, but we actually do not need these expressions.

\subsection{Identification\label{s:identifications}}
We know from \eqref{e:Egorov} that $[-\frac{ih\Delta}2, \scrU^* \Op_h(P) \scrU]= \scrU^*
\Op_h(E\partial_1 P-\frac{ih}2\partial_1^2 P)\scrU= \cA_E(E\partial_1 P-\frac{ih}2\partial_1^2 P) +O(h^2)$.

Using the identification lemma \ref{l:identification} below, this leads directly to the identifications:
 $$ K(r, u, E, j) +\partial_r B(r, u, E, j)(E\cos(\theta_1))^2= A_{E\partial_1 P}$$
$$  L(r, u, E, j)=B_{E\partial_1 P}$$
$$ M(r, u, E, j)+  \partial_r B(r, u, E, j) \frac{\cos\theta_1}{(Er)^2}+  \partial_r D(r, u, E, j)(E\cos(\theta_1))^2= C_{E\partial_1 P}-\frac12
A_{\partial_1^2 P}$$
$$  N(r, u, E, j)=D_{E\partial_1 P}-\frac12 B_{\partial_1^2 P}$$
where $\theta_1=\theta_1(r, E, j)$ denotes as before the solution in $[-\pi/2, \pi/2)$ of $\sin\theta_1=j/Er$. On the right-hand sides, notation such as $A_{E\partial_1 P}$, $B_{E\partial_1 P}$ etc. means ``the functions $A, B$ etc. associated to $E\partial_1P$ by the formulas of Proposition~\ref{p:summary}''.

To justify these identifications we are using the following:

\begin{lemma} \label{l:identification}Let $A$ and $B$ be two {\em smooth} real-valued functions. Then the values of
\begin{multline}\label{e:identity}
  \frac1{2\pi h}\int A(r, u, E, j) e^{\frac{ij(u-\theta)}h}
    e^{i\frac{Er\cos(\theta-\theta_0)}h}d\theta dj \\+\frac1{2\pi h}\int B(r, u, E, j) \cos(\theta-\theta_0)e^{\frac{ij(u-\theta)}h}
    e^{i\frac{Er\cos(\theta-\theta_0)}h}d\theta dj
\end{multline}
for all $r, u, \theta_0, E$
determine $A$ and $B$ uniquely.
\end{lemma}
\begin{proof}Integrating \eqref{e:identity} along $e^{in\theta_0}d\theta_0$ ($\theta_0\in[0, 2\pi]$, $n$ an arbitrary integer) yields the value
\begin{multline}
   \int A(r, u, E, nh) e^{{in(u-\theta)}}
    e^{i\frac{Er\cos(\theta-\theta_0)}h}d\theta \\+\int B(r, u, E, nh) \cos(\theta-\theta_0)e^{{in(u-\theta)}}
    e^{i\frac{Er\cos(\theta-\theta_0)}h}d\theta
\end{multline}
If we take $n=n(h)$ a family of even integers growing like $1/h$, application of the method of stationary phase yields that this is (up to $O(h)$)
\begin{multline}\label{e:AB}2e^{inu}(2\pi h)^{1/2}[\sin^{1/2}\theta_1\, A(r, u, E, hn(h)\cos(-n\theta_1+Erh^{-1}\cos\theta_1+\pi/4) \\ +iB(r, u, E, hn(h))\sin(-n\theta_1+Erh^{-1}\cos\theta_1+\pi/4)] \end{multline}
where $\theta_1$ is the solution in $[-\pi/2, \pi/2)$ of $\sin\theta_1=\frac{h n(h)}{Er}$. If $A$ and $B$ are continuous and real-valued
then \eqref{e:AB} suffices to determine $A$ and $B$.
\end{proof}

 \subsection{Formal calculation of $[\Delta, \cA_H(P)]$}
 We want to use the previous identities to find the formal expression of $[\Delta, \cA_H(P)]$. Remember that $\cA_H(P)$ is the operator we want to use in all our proofs, because it comes naturally into a ``tangential'' part and a ``radial'' part of degree $1$.

 If we compare the formal calculations leading to the expressions of $[\Delta, \cA_E(P)]$ and $[\Delta, \cA_H(P)]$, we see that they are identical and thus $[-ih\Delta/2, \cA_H(P)]$ applied to $e_\xi$ at $(-r\sin u, r\cos u)$ has the form
 \begin{multline*}
  \frac1{2\pi h}\int K(r, u, \sqrt{2H}, j) e^{\frac{ij(u-\theta)}h}
    e^{i\frac{Er\cos(\theta-\theta_0)}h}d\theta dj \\+\frac1{2\pi h}\int L(r, u, \sqrt{2H}, j) E\cos(\theta-\theta_0)e^{\frac{ij(u-\theta)}h}
    e^{i\frac{Er\cos(\theta-\theta_0)}h}d\theta dj\\
    + \frac{ih}{2\pi h}\int M(r, u, \sqrt{2H}, j) e^{\frac{ij(u-\theta)}h}
    e^{i\frac{Er\cos(\theta-\theta_0)}h}d\theta dj \\+\frac{ih}{2\pi h}\int N(r, u, \sqrt{2H}, j) E\cos(\theta-\theta_0)e^{\frac{ij(u-\theta)}h}
    e^{i\frac{Er\cos(\theta-\theta_0)}h}d\theta dj\\
    + (2\pi h)^{-1}\int_{u-\pi/2}^{u+\pi/2} \partial_r B(r, u, \sqrt{2H}, j)\left[(E\cos(\theta_1))^2+ih\frac{\cos\theta_1}{(Er)^2}\right] e^{\frac{ij(\theta-\theta_0)}h}
   e^{i\frac{Er\cos(\theta-u)}h}d\theta dj\\
   + (2\pi h)^{-1}\int_{u-\pi/2}^{u+\pi/2}ih \partial_r D(r, u, \sqrt{2H}, j)(E\cos(\theta_1))^2 e^{\frac{ij(\theta-\theta_0)}h}
   e^{i\frac{Er\cos(\theta-u)}h}d\theta dj
\end{multline*}
Note that $\theta_1=\theta_1(r, E, j)$ and that the symbol in the last two lines still depends on $E$ (this is why we treat it separately). Everywhere else in the symbol, $E$ has been replaced by $\sqrt{2H}$. Note also that $(E\cos(\theta_1))^2=E^2-\frac{j^2}{r^2}$.

From this and from the identifications of Section~\ref{s:identifications}, we deduce the final formula

\begin{proposition}\label{p:commutH} There exists a function $R(r, u, E, \sqrt{2H}, j)$ such that
\begin{multline}[-ih\Delta/2, \cA_H(P)]
=\cA_H(E\partial_1 P)-\frac{ih}2\cA_H(\partial_1^2 P)+O(h^2)\\
+ \partial_r B(r, u, \sqrt{2hD_t}, hD_u)\circ (-h^2\Delta-2hD_t) + ih R(r, u, \sqrt{-h^2\Delta},  \sqrt{2hD_t}, hD_u)\circ (-h^2\Delta-2hD_t)
\end{multline}
\end{proposition}

\begin{proof}
Indeed, the identifications of Section~\ref{s:identifications} yield
\begin{multline}[-ih\Delta/2, \cA_H(P)]=
I_{\sqrt{2H}}(E\partial_1 P)+II_{\sqrt{2H}}( E\partial_1 P) \\ +
 ih\left(III_{\sqrt{2H}}(E\partial_1 P)-1/2
I_{\sqrt{2H}}(\partial_1^2 P)\right)\\
+ih\left( IV_{\sqrt{2H}}(E\partial_1 P) -1/2 II_{\sqrt{2H}}(\partial_1^2 P) \right)\\
+ \partial_r B(r, u, \sqrt{2hD_t}, hD_u)\circ (-h^2\Delta-2hD_t) + ih R(r, u, \sqrt{-h^2\Delta},  \sqrt{2hD_t}, hD_u)\circ (-h^2\Delta-2hD_t)
\end{multline}
where the function $R$ is defined by the identity $ R(r, u, E, \sqrt{2H}, j) (E^2-2H)$
is
\begin{multline*} R(r, u, E, \sqrt{2H}, j) (E^2-2H)=  \partial_r B(r, u, \sqrt{2H}, j)\left[ \frac{\cos\theta_1(r, E, j)}{(Er)^2}- \frac{\cos\theta_1(r, \sqrt{2H}, j)}{2Hr^2}\right]
\end{multline*}
We can apply a simple division lemma (actually the Taylor integral formula) to write
$$\frac{\cos\theta_1(r, E, j)}{(Er)^2}- \frac{\cos\theta_1(r, \sqrt{2H}, j)}{2Hr^2}= S(r, u, E, j, \sqrt{2H}) (E^2-2H),$$
and thus
$$R(r, u, E, \sqrt{2H}, j)=\partial_r B(r, u, \sqrt{2H}, j)S(r, u, E, j, \sqrt{2H}).$$

\end{proof}

\section{Regularity of boundary data and consequences\label{s:osc}}
 We recall the following classical ``hidden regularity'' of the boundary data of solutions of~\eqref{e:S}:

    \begin{proposition}
    \label{p:regboundary}
For every $T>0$ there exists a constant $C>0$ such that, for every $u^0 \in H_{0}^{1}\left(  \ID\right) $ and every $f \in L^1(0,T ;H^1_0(\ID))$, the solution $u \in C^0([0,T]; H^1_0(\ID))$ of
\begin{eqnarray*}
\frac{1}{i}\frac{\partial u}{\partial t}= \left(-\frac{1}{2}\Delta + V \right)u + f &, t\in\R, \quad z\in \ID,
\\
u\rceil_{\d \ID} = 0  && \\
u\rceil_{t=0} = u^0 &&
\end{eqnarray*}
satisfies
\begin{equation}
\left\Vert \partial_{n} u  \right\Vert_{L^{2}\left((0,T)  \times\partial\ID\right)  }\leq
C\left( \| \nabla u^0\|_{L^{2}\left(  \ID\right) }  + \| f\|_{L^1(0,T ;H^1(\ID))}  \right).\label{e:nder}%
\end{equation}
\end{proposition}
We refer to \cite[p.~284]{Leb:92} or~\cite[Lemma~2.1]{GerLeich93} for a proof.

\begin{remark}
\label{r:energyestimates}
Because we have
$$\frac{d}{dt}\left\la \left(-\frac{\Delta}{2}+V\right)  U_V(t)u^0,  U_V(t)u^0\right\ra=\left\la \partial_t V \,\,U_V(t)u^0,  U_V(t)u^0\right\ra,$$ we see that there exists $C$ (depending on $T$, $\norm{V}_\infty$ and $\norm{  \partial_t V}_\infty$ such that
\begin{align}\label{nabla}
 T \left\Vert \nabla u^0 \right\Vert _{L^{2}\left(  \ID\right)}^2 &\leq C  \int_0^T \norm{ U_V(t)u^0}^2_{L^{2}\left(  \ID\right)    }dt
 +
\int_0^T \left\Vert \nabla U_V(t)u^0 \right\Vert^2 _{L^{2}\left(  \ID\right)  }dt \\
\label{nabla2}
\int_0^T \left\Vert \nabla U_V(t)u^0 \right\Vert^2 _{L^{2}\left(  \ID\right)  }dt
& \leq
T \left\Vert \nabla u^0 \right\Vert^2_{L^{2}\left(  \ID\right)}dt+C  \int_0^T \norm{ U_V(t)u^0}^2_{L^{2}\left(  \ID\right)    }dt
\end{align}
hence
\begin{equation*}
 T \left\Vert \nabla u^0\right\Vert _{L^{2}\left(  \ID\right)}^2 -CT   \norm{ u^0}^2_{L^{2}\left(  \ID\right)}  \leq
\int_0^T \left\Vert \nabla U_V(t) u^0\right\Vert^2 _{L^{2}\left(  \ID\right)  }dt \leq T \left\Vert \nabla u^0\right\Vert^2_{L^{2}\left(  \ID\right)}dt+ CT   \norm{ u^0}^2_{L^{2}\left(  \ID\right)  }
\end{equation*}
\end{remark}

 Proposition \ref{p:regboundary} has the following consequences, which are used everywhere in the paper.

 \medskip

Let $A$ be an operator.
For
$T>0$ let $\chi_T(t)$ be a smooth compactly supported function on $\R$ such that $|\chi_T(t)|\leq 1$, taking the value $0$ outside $(0, T)$. Obviously, for every $u^0\in H_{0}%
^{1}\left(  \ID\right)  $, for every $s>0$, we can write%
\begin{multline}
\left\la \chi_T U_V(t)u^0, A \, \chi_T\left( \partial_{n}\left( U_V(t)u^0\right) \otimes \delta_{\partial \ID}\right)\right\ra_{L^{2}\left(\ID\times \R \right)  }\\
 \leq
 \norm{A}_{  L^{2}\left(  \left(  0,T\right) , H^{-s}(\IR^2)\right)\To L^{2}\left(  \left(  0,T\right)  \times\IR^2\right)}
\, \norm{U_V(t) u^0}_{L^{2}\left(  \left(  0,T\right)  \times\ID\right)  } \,\, \norm{ \partial_{n}U_V(t)u^0 \otimes \delta_{\partial \ID}}_{  L^{2}\left(  \left(  0,T\right) , H^{-s}(\IR^2)\right) }.
\end{multline}

For any $s>1/2$, the standard trace estimates (see for instance \cite[Chapter~2, Section~4]{ChazarainPiriou}) imply that
$$\norm{ \partial_{n}U_V(t)u^0 \otimes \delta_{\partial \ID}}_{  L^{2}\left(  \left(  0,T\right) , H^{-s}(\IR^2)\right) }
\leq
\left\Vert \partial_{n}\left( U_V(t)u^0\right)  \right\Vert
_{L^{2}\left(  \left(  0,T\right)  \times\partial\ID\right)  }$$
and by Proposition \ref{p:regboundary} this is bounded by $C\left\Vert \nabla u^0\right\Vert _{L^{2}\left(  \ID\right)  }$.

Let now be $\eps>0$ (in the context of this paper, think of $\eps$ as being $h$ or $R^{-1}$).
In the previous inequality we take $A=\Op_1(a(z, \eps\xi, t, \eps^2 H))=a(z, \eps D_z, t, \eps^2 D_t)$ where $a$ is smooth, compactly supported in all variables, and is such that $\chi_T(t)\equiv 1$ on the support of $a$. We note then that $\chi_T \, A  =A$ and $A \,\chi_T =A+O(\eps^{-\infty})$. Using the fact that
$$\norm{A}_{  L^{2}\left(  \left(  0,T\right) , H^{-s}(\IR^2)\right)\To L^{2}\left(  \left(  0,T\right)  \times\IR^2\right)}\leq C(T, a)\eps^{-s} ,$$ we obtain
\begin{equation} \label{e:pdo}
\frac{\eps^2}2\left\la \chi_T U_V(t)u^0, A \left( \partial_{n}\left( U_V(t)u^0\right) \otimes \delta_{\partial \ID}\right)\right\ra_{L^{2}\left( \IR^2 \times \R \right)  }\\
 \leq
C(T, a)  \eps^{1-s}
 \norm{u^0}_{L^{2}\left(  \ID\right)  }\left\Vert \eps \nabla u^0\right\Vert _{L^{2}\left(  \ID\right)  }.
\end{equation}
If we extend  $U_V(t)u^0$ to take the value $0$ outside  $\ID$, we have already noted that
$$\frac{\eps^2}2 \partial_{n}\left( U_V(t)u^0\right) \otimes \delta_{\partial \ID}=\eps^2\left(-\frac{\Delta}2+ V+i \partial_t\right)U_V(t)u^0$$
and thus equation \eqref{e:pdo} (always for  $A=\Op_1(a(z, \eps\xi, t, \eps^2 H)$) may be rewritten as
\begin{multline} \label{e:pdobis}
 \left\la \chi_T U_V(t) u^0, A \left(-\frac{\eps^2\Delta}2+\eps^2 V+i\eps^2\partial_t\right)U_V(t)u^0\right\ra_{L^{2}\left(\IR^2 \times \R \right)  }\\
 \leq
C(T, a)  \eps^{1-s}
 \norm{u^0}_{L^{2}\left(  \ID\right)  }\left\Vert \eps \nabla u^0\right\Vert _{L^{2}\left(  \ID\right)  }.
\end{multline}
As a consequence
\begin{multline} \label{e:pdobisbis}
 \left\la \chi_T U_V(t) u^0, \Op_1\left(a(z, \eps\xi, t, \eps^2 H)\left(\frac{\eps^2|\xi|^2}2-\eps^2H\right)\right)U_V(t)u^0\right\ra_{L^{2}\left( \IR^2 \times \R \right)  }\\
 \leq
C(T, a)  \eps^{1-s}
 \norm{u^0}_{L^{2}\left(  \ID\right)  }\left\Vert \eps \nabla u^0\right\Vert _{L^{2}\left(  \ID\right)  } +\eps^2\left|\int \chi_T(t)\left\la V U_V(t)u^0, U_V(t)u^0\right\ra_{L^{2}\left(  \ID\right)  } dt\right|.
\end{multline}
Finally, if $a$ is supported away from $\{H=0\}$ this implies
\begin{multline} \label{e:pdobisbis0}
 \left\la \chi_T U_V(t) u^0, \Op_1\left(a(z, \eps\xi, t, \eps^2 H)\left(\frac{|\xi|^2}{2H}-1\right)\right)U_V(t)u^0\right\ra_{L^{2}\left(\IR^2 \times \R \right)  }\\
 \leq
C(T, a)  \eps^{1-s}
 \norm{u^0}_{L^{2}\left(  \ID\right)  }\left(\left\Vert \eps \nabla u^0\right\Vert _{L^{2}\left(  \ID\right)  } + \norm{u^0}_{L^{2}\left(  \ID\right)  }\right).
\end{multline}

Since $a$ is compactly supported with respect to $H$, let us now show that we can actually replace $\left\Vert \eps \nabla u^0\right\Vert _{L^{2}\left(  \ID\right)  }$ by $\norm{u^0}_{L^{2}\left(  \ID\right)  }$ in the last estimate (up to a constant depending on the support of $a$). The argument is easy if $V$ does not depend on time, but requires some care for time-dependent $V$. To see that, introduce a compactly supported function $g$ on $\R$ such that $a(z, \xi, t, H) g(H)=a(z, \xi, t, H)$ (i.e. $g=1$ in a neighborhood of $\supp(a)$).

Let \begin{equation} \label{e:defww0} w(t)= g(\eps^2 D_t)U_V(t)u^0
, \quad\text{and} \quad  w^0 = w\rceil_{t=0}.
\end{equation}
If $V$ does not depend on time, we have
$$w=U_V(t) g\left(-\frac{\eps^2\Delta_D}{2}+\eps^2  V\right) u^0,
\quad\text{and} \quad  w^0 =
g\left(-\frac{\eps^2\Delta_D}{2}+\eps^2  V\right) u^0$$
 so we can replace $u^0$ by $\tilde u^0= g\left(-\frac{\eps^2\Delta}{2}+\eps^2  V\right) u^0$ in the previous argument, and it is obvious that $\left\Vert \eps \nabla \tilde u^0\right\Vert _{L^{2}\left(  \ID\right)  }\leq C(g) \norm{u^0}_{L^2(\ID)}.$

For general $V$, we write $[ g(\eps^2 D_t), V]=
O(\eps^2)\norm{\partial_t V}_\infty$, where the estimate of the remainder holds in the operator norm from $L^\infty_{\comp}(\R_t,
L^2(\ID))$ to $L^\infty_{\loc}(\R_t,
L^2(\ID))$. We have
$$\frac{1}i\frac{\partial}{\partial t}w= \left(-\frac{\Delta}{2}+V\right)w +[ g(\eps^2 D_t), V]w.$$
Thus
\begin{align}
&  \left\la \chi_T U_V(t) u^0, \Op_1\left(a(z, \eps\xi, t, \eps^2 H)\left(\frac{\eps^2|\xi|^2}2-\eps^2H\right)\right)U_V(t)u^0\right\ra_{L^{2}\left(\IR^2 \times \R \right)  }\\
& \qquad = \left\la \chi_T U_V(t) u^0, \Op_1\left(a(z, \eps\xi, t, \eps^2 H)\left(\frac{\eps^2|\xi|^2}2-\eps^2H\right)\right) w \right\ra_{L^{2}\left( \IR^2 \times \R \right)  } \nonumber \\
& \qquad =
 \left\la \chi_T U_V(t) u^0, \Op_1\left(a(z, \eps\xi, t, \eps^2 H)\left(\frac{\eps^2|\xi|^2}2-\eps^2H\right)\right)U_V(t)w^0\right\ra_{L^{2}\left(\IR^2 \times \R\right)  }
 \nonumber \\
 & \qquad  \quad  + O_T(\eps^2)\norm{\partial_t V}_\infty \norm{u^0}_{L^2(\ID)}^2 ,\nonumber
\end{align}
where $w^0$ is defined in~\ref{e:defww0}.

We know from the bilinear version of inequality \eqref{e:pdobisbis} that
\begin{multline} \label{e:polarized}
 \left\la \chi_T U_V(t) u^0, \Op_1\left(a(x, \eps\xi, t, \eps^2 H)\left(\frac{\eps^2|\xi|^2}2-\eps^2H\right)\right)U_V(t)w^0\right\ra_{L^{2}\left(\IR^2 \times \R \right)  }\\
 \leq
C(T, a)  \eps^{1-s}
 \norm{u^0}_{L^{2}\left(  \ID\right)  }\left\Vert \eps \nabla w^0\right\Vert _{L^{2}\left(  \ID\right)  } +\eps^2\left|\int \chi_T(t)\left\la V U_V(t)u^0, U_V(t)w^0\right\ra_{L^{2}\left(  \ID\right)  }dt\right|\\
 \leq C(T, a)  \eps^{1-s}
 \norm{u^0}_{L^{2}\left(  \ID\right)  }\left\Vert \eps \nabla w^0\right\Vert _{L^{2}\left(  \ID\right)  } +\eps^2 C T \norm{u^0}^2_{L^2(\ID)}.
\end{multline}
Now we can use \eqref{nabla}-\eqref{nabla2} in the form
\begin{multline*}
 T \left\Vert \nabla w^0\right\Vert _{L^{2}\left(  \ID\right)}^2 \leq C T (1+\eps^2)    \norm{u^0}^2_{L^{2}\left(  \ID\right)    } +
\int_0^T \left\Vert \nabla w(t)\right\Vert^2 _{L^{2}\left(  \ID\right)  }\\
 \leq C T(1+\eps^2)    \norm{u^0}^2_{L^{2}\left(  \ID\right)    } +
2\int_0^T \left\la \left(-\frac{\Delta}{2}+V\right)  w(t),   w(t)\right\ra_{L^{2}\left(  \ID\right)  }dt\\
 \leq C T(1+\eps^2)    \norm{u^0}^2_{L^{2}\left(  \ID\right)    } +
2\left|\int_0^T \left\la D_t  w(t),   w(t)\right\ra_{L^{2}\left(  \ID\right)  }dt\right|\\
 \leq C T(1+\eps^2)    \norm{u^0}^2_{L^{2}\left(  \ID\right)    } +
2\left|\int_0^T  \left\la D_t  g(\eps^2 D_t) u(t),   g(\eps^2 D_t) u(t)\right\ra_{L^{2}\left(  \ID\right)  }dt\right|\\
 \leq C T(1+\eps^2)    \norm{u^0}^2_{L^{2}\left(  \ID\right)    } + C T \eps^{-2}\norm{u^0}^2_{L^{2}\left(  \ID\right)    }
\end{multline*}
This finally shows that the term $\left\Vert \eps \nabla w^0\right\Vert _{L^{2}\left(  \ID\right)  }$ on the right-hand side of \eqref{e:polarized} is bounded by $C T \norm{u^0}_{L^2(\ID)}$. Finally, what we have shown is that, for $a(x, \xi, t, H)$ compactly supported, and supported away from $H=0$, we have
\begin{multline} \label{e:final}
 \left\la \chi_T U_V(t) u^0, \Op_1\left(a(z, \eps\xi, t, \eps^2 H)\left(\frac{|\xi|^2}{2H}-1\right)\right)U_V(t)u^0\right\ra_{L^{2}(\IR^2 \times \R )  }\\
 \leq
C(T, a)  \eps^{1-s}
 \norm{u^0}^2_{L^{2}\left(  \ID\right)  }
\end{multline}
for all $u^0\in L^2(\ID)$ (that is to say, we do not need to assume a priori that $u^0$ is $\eps$-oscillating).

Similarly, if $a$ is supported away from $\xi=0$ and compactly supported, we have
\begin{multline} \label{e:final2}
 \left\la \chi_T U_V(t) u^0, \Op_1\left(a(z, \eps\xi, t, \eps^2 H)\left(\frac{2H}{|\xi|^2}-1\right)\right)U_V(t)u^0\right\ra_{L^{2}\left(\IR^2 \times \R \right)  }\\
 \leq
C(T, a)  \eps^{1-s}
 \norm{u^0}^2_{L^{2}\left(  \ID\right)  }
\end{multline}

\begin{remark}\label{r:osc}Let us summarize the properties of the function $w^0=g(\eps^2 D_t)U_V(t)u^0\rceil_{t=0}$.
\begin{itemize}
\item For fixed $\eps>0$, the operator $u^0\mapsto w^0$ is compact on $L^2(\ID)$ since we have proved $\norm{\eps\nabla w^0}_{L^2(\ID)}\leq C \norm{u^0}_{L^2(\ID)}.$
\item By continuity of $t\mapsto U_V(t) u^0$, we have $w^0\Lim_{L^2(\ID)} u^0$ as $\eps\To 0$.
\item For compactly supported $a$ satisfying $ga=a$ we have
\begin{multline*} \left\la U_V(t) w^0, \Op_1\left(a(x, \eps\xi, t, \eps^2 H) \right)U_V(t)w^0\right\ra_{L^{2}\left(\IR^2 \times \R \right)  }\\=  \left\la U_V(t) u^0, \Op_1\left(a(x, \eps\xi, t, \eps^2 H) \right)U_V(t)u^0\right\ra_{L^{2}\left(\IR^2 \times \R \right)  } +o(1)_{\eps} \norm{u^0}^2_{L^2(\ID)}\end{multline*}
\end{itemize}
\end{remark}

Finally, we prove by dyadic decomposition a statement similar to \eqref{e:final} and \eqref{e:final2} for homogeneous functions:
\begin{proposition} Let $a(z, \xi, t, H)$ be a smooth function, compactly supported in $(z, t)$, and with the following homogeneity property :
there exist $R_{0}>0$ such that
\[
a(z, \xi, t, H)   =a(z, \lambda \xi, t, \lambda^2 H) \text{,\quad for } |\xi|^2 +| H| >R_{0} \text{ and } \lambda\geq 1.%
\]
Fix $s\in (1/2, 1)$. If $a$ vanishes in a neighbourhood of the set $\{|\xi|^2= 2H\}$ then for $R$ large enough
\begin{multline} \label{e:finalhom}
 \left\la U_V(t) u^0, \Op_1\left(a(z, \xi, t,  H)\left(1-\chi\left(\frac{|\xi|^2 +| H|}{R^2}\right)\right)\right) U_V(t)u^0\right\ra_{L^{2}\left(\IR^2 \times \R \right)  }\\
 \leq
C(T, a)  R^{s-1}
 \norm{u^0}^2_{L^{2}\left(  \ID\right)  }.
\end{multline}
\end{proposition}

\begin{proof}
To see that, decompose
\begin{multline*}
a(z, \xi, t,  H)\left(1-\chi\left(\frac{|\xi|^2 +| H|}{R^2}\right)\right)=
\\ \sum_{k=0}^\infty
a(z, 2^{-k}R^{-1}\xi, t, 2^{-2k}R^{-2}H)\left(\chi\left(\frac{|\xi|^2 +| H|}{2^{2(k+1)}R^2}\right)-\chi\left(\frac{|\xi|^2 +| H|}{ 2^{2k}R^2}\right)\right).
\end{multline*}
For each $k$ in the sum above, decompose further
\begin{multline*}
\chi\left(\frac{|\xi|^2 +| H|}{2^{2(k+1)}R^2}\right)-\chi\left(\frac{|\xi|^2 +| H|}{ 2^{2k}R^2}\right)\\=
\left(\chi\left(\frac{|\xi|^2 +| H|}{2^{2(k+1)}R^2}\right)-\chi\left(\frac{|\xi|^2 +| H|}{ 2^{2k}R^2}\right)\right) \left(\chi\left(\frac{ H}{ 2^{2k-1}R^2}\right) + (1-\chi)\left(\frac{ H}{ 2^{2k-1}R^2}\right)   \right)
\end{multline*}
and note that we must have $|\xi|^2\geq 2^{2k-1}R^2$ or $H\geq 2^{2k-1}R^2$ on the support of this function.

If $a$ vanishes in a neighbourhood of the set $\{|\xi|^2= 2H\}$, we can write
$$ a(z, \xi, t,  H)=b(z, \xi, t,  H)\left(\frac{2H}{|\xi|^2}-1\right)$$
where $|\xi|^2\geq 2^{2k-1}R^2$
and
$$a(z, \xi, t,  H)= b(z, \xi, t,  H)\left(\frac{|\xi|^2}{2H}-1\right)$$
where $H\geq 2^{2k-1}R^2$. Applying \eqref{e:final} and \eqref{e:final2}
for each $k$ (with $\eps=2^{-k}R^{-1}$), we finally obtain
\begin{multline}
 \left\la U_V(t) u^0, \Op_1\left(a(z, \xi, t, H)\left(1-\chi\left(\frac{|\xi|^2 +| H|}{R^2}\right)\right)\right) U_V(t)u^0\right\ra_{L^{2}(\IR^2 \times \R )  }\\
 \leq
C(T, a) \sum_{k=0}^{+\infty} R^{s-1}2^{k(s-1)}
 \norm{u^0}^2_{L^{2}\left(  \ID\right)  },
\end{multline}
which proves the proposition.
\end{proof}

\bigskip
To conclude this section, we give a proof of Lemma~\ref{l:regutemps}.

\begin{proof}[Proof of Lemma \ref{l:regutemps}]
The operator $A(D_t)\varphi$ is bounded on $L^2(\R \times \ID)$, so
\begin{align}
\label{e:estimAphi}
\|\nabla A(D_t)\varphi u\|_{L^2(\R \times \ID)}^2 & =  \la - \Delta A(D_t)\varphi u , A(D_t)\varphi u \ra_{L^2(\R \times \ID)} \nonumber  \\
& = \la A(D_t)\varphi (- \Delta) u , A(D_t)\varphi u \ra_{L^2(\R \times \ID)}  \nonumber\\
& \quad +
\la [- \Delta , A(D_t)\varphi ] u , A(D_t)\varphi u \ra_{L^2(\R \times \ID)} .
\end{align}
One the one hand, we have
\begin{align*}
\left|\la [- \Delta , A(D_t)\varphi ] u , A(D_t)\varphi u \ra_{L^2(\R \times \ID)} \right|
&= \left| - \la 2 \nabla\varphi \cdot \nabla  u+ u \Delta \varphi , A(D_t)^2 \varphi u \ra_{L^2(\R \times \ID)} \right| \\
&\leq 2 \left| \la u , div \left\{ \nabla\varphi  \left( A(D_t)^2 \varphi u \right) \right\} \ra_{L^2(\R \times \ID)} \right|
+ C \|\tilde{\varphi} u\|^2_{L^2(\R \times \ID)} \\
&\leq 2 \left| \la u , \nabla\varphi  \cdot \nabla \left( A(D_t)^2 \varphi u \right)  \ra_{L^2(\R \times \ID)} \right|
+ C \|\tilde{\varphi} u\|^2_{L^2(\R \times \ID)} \\
&\leq
\varepsilon \|\nabla A(D_t)\varphi u\|_{L^2(\R \times \ID)}^2
+ C(1+\varepsilon^{-1}) \|\tilde{\varphi} u\|^2_{L^2(\R \times \ID)}
\end{align*}
for some $\tilde{\varphi}$ equal to one on the support of $\varphi$, for all $\varepsilon >0$.

On the other hand, since $u$ solves~\eqref{e:S}, we have
\begin{align*}
\left|\la A(D_t)\varphi (- \Delta) u , A(D_t)\varphi u \ra_{L^2(\R \times \ID)} \right|
& = \left|\la A(D_t)\varphi (2 D_t - V ) u , A(D_t)\varphi u \ra_{L^2(\R \times \ID)} \right| \\
& \leq \left|\la A(D_t)\varphi (2 D_t - V ) u , A(D_t)\varphi u \ra_{L^2(\R \times \ID)} \right| \\
& \leq 2 \left|\la A(D_t)^2 \varphi D_t u , \varphi u \ra_{L^2(\R \times \ID)} \right|  + C \|\tilde{\varphi} u\|^2_{L^2(\R \times \ID)} \\
& \leq 2 \left|\la A(D_t)^2  D_t \varphi u , \varphi u \ra_{L^2(\R \times \ID)} \right|  +  C \|\tilde{\varphi} u\|^2_{L^2(\R \times \ID)} \\
& \leq  C \|\tilde{\varphi} u\|^2_{L^2(\R \times \ID)} ,
\end{align*}
since $A(D_t)^2  D_t = 1/2 \Op_1(\psi^2(H))$ is bounded. Collecting these estimates in~\eqref{e:estimAphi}, recalling that $\|\tilde{\varphi} u\|^2_{L^2(\R \times \ID)}\leq C\|u^0\|^2_{L^2(\ID)}$, and taking $\varepsilon$ sufficiently small concludes the proof of Lemma~\ref{l:regutemps}.
\end{proof}

\section{Time regularity of Wigner measures}
\label{s:timeregwign}
In this section we present a proof of the following (general) result on time regularity of semiclassical measures associated to solutions of the Schr\"odinger equation~\eqref{e:S}.
Even if not stated here, its microlocal counterpart also holds.

\begin{proposition}\label{p:regt}
Let $\mu_{sc}$ be obtained as a limit (\ref{e:Wsc}). Then there
exists $\mu\in
L^{\infty}\left(  \mathbb{R}_{t};\mathcal{M}_{+}(T^{\ast}\mathbb{R}%
^{2})\right)  $ such that, for every $a\in C_{c}^{\infty}\left(
T^{\ast
}\mathbb{R}^{2}\times T^{\ast}\mathbb{R}\right)  $ we have:%
\[
\int_{T^{\ast}\mathbb{R}^{2}\times T^{\ast}\mathbb{R}}a\left(
z,\xi
,t,H\right)  \mu_{sc}\left(  dz,d\xi,dt,dH\right)  =\int_{\mathbb{R}}%
\int_{T^{\ast}\mathbb{R}^{2}}a\left(  z,\xi,t,\frac{\left\vert
\xi\right\vert ^{2}}{2}\right)  \mu\left(  t,dz,d\xi\right)  dt.
\]

\end{proposition}

\begin{proof}
Let $u_{h}\left(  \cdot,t\right)  :=U_{V}\left(  t\right)
u_{h}^{0}$ and note
that the Wigner distributions:%
\[
\tilde{W}_{u_{h}}^{h}\left(  t\right)  :C_{c}^{\infty}\left(
T^{\ast }\mathbb{R}^{2}\right)  \ni l\longmapsto\left\langle
U_{V}\left(  t\right)
u_{h}^{0},\operatorname*{Op}\nolimits_{h}\left(  l\right)
U_{V}\left(
t\right)  u_{h}^{0}\right\rangle _{L^{2}(\IR^2)}\in\mathbb{C}%
\]
are uniformly bounded in $L^{\infty}\left(  \mathbb{R}_{t};\mathcal{D}%
^{\prime}(T^{\ast}\mathbb{R}^{2})\right)  $. Hence, possibly
after extracting a subsequence, we can assume that, for every
$b\in C_{c}^{\infty
}\left(  T^{\ast}\mathbb{R}^{2}\times\mathbb{R}\right)  $:%
\[
\lim_{h\rightarrow0^{+}}\int_{\mathbb{R}}\left\langle U_{V}\left(
t\right) u_{h}^{0},\operatorname*{Op}\nolimits_{h}\left(  b\left(
\cdot,t\right)
\right)  U_{V}\left(  t\right)  u_{h}^{0}\right\rangle _{L^{2}(\IR^2
)}=\int_{\mathbb{R}}\int_{T^{\ast}\mathbb{R}^{2}}b\left(
z,\xi,t\right) \tilde{\mu}_{sc}\left(  t,dz,d\xi\right)  dt
\]
and (using the sharp G{\aa}rding inequality) the limiting Wigner distribution is a nonnegative measure $\tilde{\mu}%
_{sc}\in L^{\infty}\left(  \mathbb{R}_{t};\mathcal{M}_{+}(T^{\ast}%
\mathbb{R}^{2})\right)  $. We next show that for any $b\in
C_{c}^{\infty }\left(
T^{\ast}\mathbb{R}^{2}\times\mathbb{R}\right)  $ with $b\geq0$ one
has:%
\begin{equation}
\int_{T^{\ast}\mathbb{R}^{2}\times T^{\ast}\mathbb{R}}b\left(
z,\xi,t\right)  \mu _{sc}\left(  dz,d\xi,dt,dH\right)
\leq\int_{\mathbb{R}}\int_{T^{\ast }\mathbb{R}^{2}}b\left(
z,\xi,t\right)  \tilde{\mu}_{sc}\left(
t,dz,d\xi\right)  dt.\label{e:marginH}%
\end{equation}
To see this, let $\chi\in C_{c}^{\infty}\left(  \mathbb{R}\right)
$ be a cut-off function satisfying $0\leq\chi\leq1$, strictly
positive in $\left( -3/2,3/2\right)  $, vanishing outside that
interval, and such that $\chi \rceil_{\left(  -1,1\right)  }\equiv1$.
Write, for $R>0$, $\chi_{R}:=\chi\left(
\cdot/R\right)  $ and $\sigma_{R}:=\sqrt{1-\chi_{R}}$. Then we have:%
\begin{equation}
\left\langle u_{h},\operatorname*{Op}\nolimits_{h}\left(  b\right)
\chi
_{R}\left(  h^{2}D_{t}\right)  u_{h}\right\rangle _{L^{2}(\IR^2%
\times\mathbb{R})}=\left\langle
u_{h},\operatorname*{Op}\nolimits_{h}\left(
b\right)  u_{h}\right\rangle _{L^{2}(\IR^2\times\mathbb{R})}%
+k_{h,R}\left(  b\right)  +O\left(  h\right)  ,\label{e:mutild}%
\end{equation}
where:%
\[
k_{h,R}\left(  b\right)  :=\left\langle \sigma_{R}\left(
h^{2}D_{t}\right) u_{h},\operatorname*{Op}\nolimits_{h}\left(
b\right)  \sigma_{R}\left( h^{2}D_{t}\right)  u_{h}\right\rangle
_{L^{2}(\IR^2\times\mathbb{R})}.
\]
Taking limits in (\ref{e:mutild}) as $h\rightarrow0^{+}$ we find that:%
\begin{multline}
\int_{T^{\ast}\mathbb{R}^{2}\times T^{\ast}\mathbb{R}}b\left(
z,\xi,t\right)  \chi
_{R}\left(  H\right)  \mu_{sc}\left(  dz,d\xi,dt,dH\right)  \\
=\int_{\mathbb{R}%
}\int_{T^{\ast}\mathbb{R}^{2}}b\left(  z,\xi,t\right)
\tilde{\mu}_{sc}\left( t,dz,d\xi\right)
dt+\lim_{h\rightarrow0^{+}}k_{h,R}\left(  b\right)
.\label{e:bmanip}%
\end{multline}
But clearly, as $b\geq0$, we always have
\[
\lim_{h\rightarrow0^{+}}k_{h,R}(b)=\lim_{h\rightarrow0^{+}}\int_{\mathbb{R}%
}\tilde{W}_{\sigma_{R}\left(  h^{2}D_{t}\right)  u_{h}}^{h}\left(
b\left( t,\cdot\right)  \right)  dt\geq0,
\]
for every $R>0$. Taking this into account and letting
$R\rightarrow\infty$ in (\ref{e:bmanip}) proves (\ref{e:marginH}).

Now, as a consequence of (\ref{e:marginH}) we have that the image
of $\mu _{sc}$ under the projection onto the $H$-component is of
the form $\mu\left(
t,\cdot\right)  dt$ for some $\mu\in L^{\infty}\left(  \mathbb{R}%
_{t};\mathcal{M}_{+}(T^{\ast}\mathbb{R}^{2})\right)  $. The
disintegration
theorem then ensures that $\mu_{sc}$ can be written as:%
\[
\mu_{sc}\left(  dz, d\xi, dt, dH\right)  =\mu_{x,\xi,t}\left(
dH\right) \mu\left( t,dz,d\xi\right)  dt.
\]
Since $\mu_{sc}\ $is supported on the characteristic set
$\left\vert \xi\right\vert ^{2}=2H$ we conclude that
$\mu_{x,\xi,t}\left(  dH\right) =\delta_{\left\vert \xi\right\vert
^{2}/2}\left(  dH\right)$ and the result follows.
\end{proof}

\bibliographystyle{alpha}
\bibliography{biblio}

\def\cprime{$'$}
\begin{thebibliography}{GMMP97}

\bibitem[AFKM14]{AnantharamanKMacia}
Nalini Anantharaman, Clotilde Fermanian-Kammerer, and Fabricio Maci\`{a}.
\newblock Semiclassical completely integrable systems : Long-time dynamics and
  observability via two-microlocal wigner measures.
\newblock {\em to appear in Amer. J. Math.}, 2014.

\bibitem[AL14]{AnL}
Nalini Anantharaman and Matthieu L\'eautaud.
\newblock Sharp polynomial decay rates for the damped wave equation on the
  torus.
\newblock {\em Anal. PDE}, 7(1):159--214, 2014.

\bibitem[ALM15]{ALMcras}
Nalini Anantharaman, Matthieu L\'eautaud, and Fabricio Maci\`{a}.
\newblock Delocalization of quasimodes on the disk.
\newblock {\em In preparation}, 2015.

\bibitem[AM10]{AM:10}
Daniel Azagra and Fabricio Maci{\`a}.
\newblock Concentration of symmetric eigenfunctions.
\newblock {\em Nonlinear Anal.}, 73(3):683--688, 2010.

\bibitem[AM12]{AMSurvey}
Nalini Anantharaman and Fabricio Maci{\`a}.
\newblock The dynamics of the {S}chr\"odinger flow from the point of view of
  semiclassical measures.
\newblock In {\em Spectral geometry}, volume~84 of {\em Proc. Sympos. Pure
  Math.}, pages 93--116. Amer. Math. Soc., Providence, RI, 2012.

\bibitem[AM14]{AnantharamanMaciaTore}
Nalini Anantharaman and Fabricio Maci\`{a}.
\newblock Semiclassical measures for the {S}chr\"odinger equation on the torus.
\newblock {\em to appear in J. Eur. Math. Soc.}, 2014.

\bibitem[AMJ12]{AJM}
Tayeb Aissiou, Fabricio Maci\`{a}, and Dmitry Jakobson.
\newblock Uniform estimates for the solutions of the schr\"odinger equation on
  the torus and regularity of semiclassical measures.
\newblock {\em Mathematical Research Letters}, 19(3):589--599, 2012.

\bibitem[Ana08]{NAQUE}
Nalini Anantharaman.
\newblock Entropy and the localization of eigenfunctions.
\newblock {\em Ann. of Math. (2)}, 168(2):435--475, 2008.

\bibitem[AR12]{AnRiv}
Nalini Anantharaman and Gabriel Rivi{\`e}re.
\newblock Dispersion and controllability for the {S}chr\"odinger equation on
  negatively curved manifolds.
\newblock {\em Anal. PDE}, 5(2):313--338, 2012.

\bibitem[BBZ13]{BBZ14}
Jean Bourgain, Nicolas Burq, and Maciej Zworski.
\newblock Control for {S}chr\"odinger operators on 2-tori: rough potentials.
\newblock {\em J. Eur. Math. Soc. (JEMS)}, 15(5):1597--1628, 2013.

\bibitem[BLR92]{BLR:92}
Claude Bardos, Gilles Lebeau, and Jeffrey Rauch.
\newblock Sharp sufficient conditions for the observation, control, and
  stabilization of waves from the boundary.
\newblock {\em SIAM J. Control Optim.}, 30(5):1024--1065, 1992.

\bibitem[Bou97]{BourgainLatt}
Jean Bourgain.
\newblock Analysis results and problems related to lattice points on surfaces.
\newblock In {\em Harmonic analysis and nonlinear differential equations
  ({R}iverside, {CA}, 1995)}, volume 208 of {\em Contemp. Math.}, pages
  85--109. Amer. Math. Soc., Providence, RI, 1997.

\bibitem[BR02]{BouzouinaRobert}
A.~Bouzouina and D.~Robert.
\newblock Uniform semiclassical estimates for the propagation of quantum
  observables.
\newblock {\em Duke Math. J.}, 111(2):223--252, 2002.

\bibitem[Bur13]{Burq:13}
Nicolas Burq.
\newblock Semiclassical measures for inhomogeneous {S}chr\"odinger equations on
  tori.
\newblock {\em Anal. PDE}, 6(6):1421--1427, 2013.

\bibitem[BZ04]{BurqZworski04}
Nicolas Burq and Maciej Zworski.
\newblock Control theory and high energy eigenfunctions.
\newblock In {\em Journ\'ees ``\'{E}quations aux {D}\'eriv\'ees
  {P}artielles''}, pages Exp. No. XIII, 10. \'Ecole Polytech., Palaiseau, 2004.

\bibitem[BZ12]{BZ:12}
Nicolas Burq and Maciej Zworski.
\newblock Control for {S}chr\"odinger operators on tori.
\newblock {\em Math. Res. Lett.}, 19(2):309--324, 2012.

\bibitem[CP82]{ChazarainPiriou}
Jacques Chazarain and Alain Piriou.
\newblock {\em Introduction to the theory of linear partial differential
  equations}, volume~14 of {\em Studies in Mathematics and its Applications}.
\newblock North-Holland Publishing Co., Amsterdam, 1982.
\newblock Translated from the French.

\bibitem[CV71]{CV:71}
Alberto-P. Calder{\'o}n and R\'emi Vaillancourt.
\newblock On the boundedness of pseudo-differential operators.
\newblock {\em J. Math. Soc. Japan}, 23:374--378, 1971.

\bibitem[FK00a]{Fermanian2micro}
Clotilde Fermanian-Kammerer.
\newblock Mesures semi-classiques 2-microlocales.
\newblock {\em C. R. Acad. Sci. Paris S\'er. I Math.}, 331(7):515--518, 2000.

\bibitem[FK00b]{FermanianShocks}
Clotilde Fermanian~Kammerer.
\newblock Propagation and absorption of concentration effects near shock
  hypersurfaces for the heat equation.
\newblock {\em Asymptot. Anal.}, 24(2):107--141, 2000.

\bibitem[FKG02]{FermanianGerardCroisements}
Clotilde Fermanian-Kammerer and Patrick G{\'e}rard.
\newblock Mesures semi-classiques et croisement de modes.
\newblock {\em Bull. Soc. Math. France}, 130(1):123--168, 2002.

\bibitem[G{\'e}r91a]{GerardMesuresSemi91}
Patrick G{\'e}rard.
\newblock Mesures semi-classiques et ondes de {B}loch.
\newblock In {\em S\'eminaire sur les \'{E}quations aux {D}\'eriv\'ees
  {P}artielles, 1990--1991}, pages Exp.\ No.\ XVI, 19. \'Ecole Polytech.,
  Palaiseau, 1991.

\bibitem[G{\'e}r91b]{GerardMDM91}
Patrick G{\'e}rard.
\newblock Microlocal defect measures.
\newblock {\em Comm. Partial Differential Equations}, 16(11):1761--1794, 1991.

\bibitem[G{\'e}r98]{GerardSobolev}
Patrick G{\'e}rard.
\newblock Description du d\'efaut de compacit\'e de l'injection de {S}obolev.
\newblock {\em ESAIM Control Optim. Calc. Var.}, 3:213--233 (electronic), 1998.

\bibitem[GL93]{GerLeich93}
Patrick G{\'e}rard and {\'E}ric Leichtnam.
\newblock Ergodic properties of eigenfunctions for the {D}irichlet problem.
\newblock {\em Duke Math. J.}, 71(2):559--607, 1993.

\bibitem[GMMP97]{GMMP}
Patrick G{\'e}rard, Peter~A. Markowich, Norbert~J. Mauser, and Fr{\'e}d{\'e}ric
  Poupaud.
\newblock Homogenization limits and {W}igner transforms.
\newblock {\em Comm. Pure Appl. Math.}, 50(4):323--379, 1997.

\bibitem[Har89]{Har:89plaque}
Alain Haraux.
\newblock S\'eries lacunaires et contr\^ole semi-interne des vibrations d'une
  plaque rectangulaire.
\newblock {\em J. Math. Pures Appl. (9)}, 68(4):457--465 (1990), 1989.

\bibitem[H{\"o}r76]{Hormander:LPDO}
Lars H{\"o}rmander.
\newblock {\em Linear partial differential operators}.
\newblock Springer Verlag, Berlin, 1976.

\bibitem[Jaf90]{JaffardPlaques}
St\'ephane Jaffard.
\newblock Contr\^ole interne exact des vibrations d'une plaque rectangulaire.
\newblock {\em Portugal. Math.}, 47(4):423--429, 1990.

\bibitem[Jak97]{JakobsonTori97}
Dmitry Jakobson.
\newblock Quantum limits on flat tori.
\newblock {\em Ann. of Math. (2)}, 145(2):235--266, 1997.

\bibitem[JZ99]{JZ:99}
Dmitry Jakobson and Steve Zelditch.
\newblock Classical limits of eigenfunctions for some completely integrable
  systems.
\newblock In {\em Emerging applications of number theory ({M}inneapolis, {MN},
  1996)}, volume 109 of {\em IMA Vol. Math. Appl.}, pages 329--354. Springer,
  New York, 1999.

\bibitem[Kom92]{Kom:92}
Vilmos Komornik.
\newblock On the exact internal controllability of a {P}etrowsky system.
\newblock {\em J. Math. Pures Appl. (9)}, 71(4):331--342, 1992.

\bibitem[Lag83]{Lagnese}
John Lagnese.
\newblock Control of wave processes with distributed controls supported on a
  subregion.
\newblock {\em SIAM J. Control Optim.}, 21(1):68--85, 1983.

\bibitem[Lau10]{Laurent}
Camille Laurent.
\newblock Global controllability and stabilization for the nonlinear
  {S}chr\"odinger equation on some compact manifolds of dimension 3.
\newblock {\em SIAM J. Math. Anal.}, 42(2):785--832, 2010.

\bibitem[Lau14]{Laurent14}
Camille Laurent.
\newblock Internal control of the {S}chr\"odinger equation.
\newblock {\em Math. Control Rel. Fields}, 4(2):161--186, 2014.

\bibitem[Leb92]{Leb:92}
Gilles Lebeau.
\newblock Contr\^ole de l'\'equation de {S}chr\"odinger.
\newblock {\em J. Math. Pures Appl. (9)}, 71(3):267--291, 1992.

\bibitem[Leb96]{Leb:96}
Gilles Lebeau.
\newblock \'{E}quation des ondes amorties.
\newblock In {\em Algebraic and geometric methods in mathematical physics
  ({K}aciveli, 1993)}, volume~19 of {\em Math. Phys. Stud.}, pages 73--109.
  Kluwer Acad. Publ., Dordrecht, 1996.

\bibitem[Lio88]{LionsSurvey88}
Jacques-Louis Lions.
\newblock Exact controllability, stabilization and perturbations for
  distributed systems.
\newblock {\em SIAM Rev.}, 30(1):1--68, 1988.

\bibitem[LP93]{LionsPaul}
Pierre-Louis Lions and Thierry Paul.
\newblock Sur les mesures de {W}igner.
\newblock {\em Rev. Mat. Iberoamericana}, 9(3):553--618, 1993.

\bibitem[LR97]{LR:97}
Gilles Lebeau and Luc Robbiano.
\newblock Stabilisation de l'\'equation des ondes par le bord.
\newblock {\em Duke Math. J.}, 86(3):465--491, 1997.

\bibitem[Mac08]{MaciaZoll}
Fabricio Maci{\`a}.
\newblock Some remarks on quantum limits on {Z}oll manifolds.
\newblock {\em Comm. Partial Differential Equations}, 33(4-6):1137--1146, 2008.

\bibitem[Mac09]{MaciaAv}
Fabricio Maci{\`a}.
\newblock Semiclassical measures and the {S}chr\"odinger flow on {R}iemannian
  manifolds.
\newblock {\em Nonlinearity}, 22(5):1003--1020, 2009.

\bibitem[Mac10]{MaciaTorus}
Fabricio Maci\`{a}.
\newblock High-frequency propagation for the {S}chr\"odinger equation on the
  torus.
\newblock {\em J. Funct. Anal.}, 258(3):933--955, 2010.

\bibitem[Mac11]{MaciaDispersion}
Fabricio Maci{\`a}.
\newblock The {S}chr\"odinger flow in a compact manifold: high-frequency
  dynamics and dispersion.
\newblock In {\em Modern aspects of the theory of partial differential
  equations}, volume 216 of {\em Oper. Theory Adv. Appl.}, pages 275--289.
  Birkh\"auser/Springer Basel AG, Basel, 2011.

\bibitem[Mil96]{MillerThesis}
Luc Miller.
\newblock {\em Propagation d'ondes semi-classiques \`{a} travers une interface
  et mesures 2-microlocales}.
\newblock PhD thesis, \'Ecole Polythecnique, Palaiseau, 1996.

\bibitem[Mil97]{Miller:97}
Luc Miller.
\newblock Short waves through thin interfaces and 2-microlocal measures.
\newblock In {\em Journ\'ees ``\'{E}quations aux {D}\'eriv\'ees {P}artielles''
  ({S}aint-{J}ean-de-{M}onts, 1997)}, pages Exp.\ No. XII, 12. \'Ecole
  Polytech., Palaiseau, 1997.

\bibitem[Nie96]{NierScat}
Francis Nier.
\newblock A semi-classical picture of quantum scattering.
\newblock {\em Ann. Sci. \'Ecole Norm. Sup. (4)}, 29(2):149--183, 1996.

\bibitem[PTZ14]{PTZ:14}
Yannick Privat, Emmanuel Tr\'elat, and Enrique Zuazua.
\newblock Optimal observability of the multi-dimensional wave and
  {S}chr\"odinger equations in quantum ergodic domains.
\newblock {\em J. Europ. Math. Soc., to appear}, 2014.

\bibitem[RTTT05]{RTTT:05}
Karim Ramdani, Tak\'eo Takahashi, G\'erald Tenenbaum, and Marius Tucsnak.
\newblock A spectral approach for the exact observability of
  infinite-dimensional systems with skew-adjoint generator.
\newblock {\em J. Funct. Anal.}, 226(1):193--229, 2005.

\bibitem[RZ09]{RZ:09}
Luc Robbiano and Claude Zuily.
\newblock The {K}ato smoothing effect for {S}chr\"odinger equations with
  unbounded potentials in exterior domains.
\newblock {\em Int. Math. Res. Not. IMRN}, (9):1636--1698, 2009.

\bibitem[Sie29]{Siegel29}
Carl~Ludwig Siegel.
\newblock {\"Uber einige Anwendungen diophantischer Approximationen.}
\newblock {\em {Abh. Preu{\ss}. Akad. Wiss., Phys.-Math. Kl.}}, 1929(1):70 s.,
  1929.

\bibitem[Wig32]{Wigner32}
Eugene~P. Wigner.
\newblock On the quantum correction for thermodynamic equilibrium.
\newblock {\em Phys. Rev.}, 40:749--759, 1932.

\bibitem[Zwo12]{Zwobook}
Maciej Zworski.
\newblock {\em Semiclassical analysis}, volume 138 of {\em Graduate Studies in
  Mathematics}.
\newblock American Mathematical Society, Providence, RI, 2012.

\end{thebibliography}

\end{document}